\documentclass[12pt,reqno]{amsart}
\usepackage{amssymb,latexsym,amsmath,amscd, graphicx}
\usepackage[colorlinks=false,pagebackref,hyperindex]{hyperref}
\usepackage{verbatim}
\usepackage{xypic}

\usepackage{amsfonts}
\usepackage{amscd}

\newcommand{\simto}{\overset\sim\to}

\renewcommand{\phi}{\varphi}
\renewcommand{\epsilon}{\varepsilon}
\renewcommand{\theta}{\vartheta}

\def\ZZ{{\mathbf Z}}

\def\CC{{\mathbf C}}
\def\AAA{{\mathbf A}}
\def\RR{{\mathbf R}}
\def\QQ{{\mathbf Q}}
\def\PP{{\mathbf P}}

\def\uy{\underline{y}}

\def\cI{\mathcal{I}}
\def\cJ{\mathcal{J}}

\def\cO{\mathcal{O}}

\def\cV{\mathcal{V}}

\def\Xan{X^\mathrm{an}}

\def\fra{\mathfrak{a}}
\def\frb{\mathfrak{b}}
\def\frc{\mathfrak{c}}

\def\frm{\mathfrak{m}}
\def\frp{\mathfrak{p}}
\def\frq{\mathfrak{q}}
\def\frr{\mathfrak{r}}

\DeclareMathOperator{\Arn}{Arn}
\DeclareMathOperator{\Bl}{Bl}

\DeclareMathOperator{\Exc}{Exc}
\DeclareMathOperator{\Fitt}{Fitt}

\DeclareMathOperator{\Int}{Int}

\DeclareMathOperator{\QM}{QM}
\DeclareMathOperator{\Spec}{Spec}
\DeclareMathOperator{\Supp}{Supp}
\DeclareMathOperator{\Val}{Val}

\DeclareMathOperator{\lct}{lct}

\DeclareMathOperator{\ord}{ord}

\DeclareMathOperator{\red}{red} 
\DeclareMathOperator{\trdeg}{trdeg} 
\DeclareMathOperator{\ratrk}{ratrk} 
\DeclareMathOperator{\val}{val}

\newcommand{\llbracket}{[\negthinspace[}
\newcommand{\rrbracket}{]\negthinspace]}

\newtheorem{lemma}{Lemma}[section]
\newtheorem{theorem}[lemma]{Theorem}
\newtheorem{corollary}[lemma]{Corollary}
\newtheorem{proposition}[lemma]{Proposition}

\newtheorem{conjecture}[lemma]{Conjecture}

\newtheorem*{conjB}{Conjecture B}
\newtheorem*{conjC}{Conjecture C}

\newtheorem*{thmA}{Theorem A}

\newtheorem*{thmD}{Theorem D}

\theoremstyle{definition}
\newtheorem{definition}[lemma]{Definition}
\newtheorem{remark}[lemma]{Remark}

\newtheorem{example}[lemma]{Example}

\theoremstyle{remark}
\newtheorem*{remark*}{Remark}
\newtheorem*{note*}{Note}

\numberwithin{equation}{section}

\begin{document}

\title{Valuations and asymptotic invariants for sequences of ideals}

\thanks{2000\,\emph{Mathematics Subject Classification}.
 Primary 14F18; Secondary 12J20, 14B05.
\newline 
The first author was partially supported by
NSF grants DMS-0449465 and DMS-1001740.
The second author was partially supported by
 NSF grant DMS-0758454 and
  a Packard Fellowship.}

\date\today
\author[M.~Jonsson]{Mattias Jonsson}
\author[M.~Musta\c{t}\u{a}]{Mircea~Musta\c{t}\u{a}}
\address{Department of Mathematics, University of Michigan,
Ann Arbor, MI 48109, USA}
\email{mattiasj@umich.edu, mmustata@umich.edu}

\begin{abstract}
  We study asymptotic jumping numbers for graded sequences of ideals, 
  and show that every such invariant is computed by a suitable real
  valuation of the function field. We conjecture that every valuation
  that computes an asymptotic jumping number is necessarily 
  quasi-monomial. This conjecture holds in dimension two.
  In general, we reduce it to the case of affine space and
  to graded sequences of valuation ideals.
  Along the way, we study the structure of a suitable valuation space.
\end{abstract}

\maketitle
\setcounter{tocdepth}{1}
\tableofcontents

\markboth{M.~JONSSON AND M.~MUSTA\c{T}\u{A}}{VALUATIONS AND ASYMPTOTIC INVARIANTS FOR SEQUENCES OF IDEALS}

\section*{Introduction}
Given a nonzero ideal $\fra$ on a smooth complex variety $X$, the log canonical threshold 
$\lct(\fra)$ is a fundamental invariant in both singularity theory and birational geometry
(see, for example,~\cite{positivity},~\cite{EM} or~\cite{Kol}). Analytically, it can be described as follows: arguing locally, we may assume that $\fra$ is generated by $f_1,\dots,f_m\in
\cO(X)$, in which case
\begin{equation*}
  \lct(\fra)=\sup\{s>0\mid (\sum_i|f_i|^2)^{-s}\ \text{is locally integrable}\}.
\end{equation*}
Alternatively, the invariant admits the following description in terms of valuations:
\begin{equation}\label{eq1_introduction}
\lct(\fra)=\inf_E\frac{A(\ord_E)}{\ord_E(\fra)},
\end{equation}
where $E$ varies over the prime divisors over $X$, and where $A(\ord_E)-1$ is
the coefficient of the divisor $E$ on $Y$ in the relative canonical class $K_{Y/X}$. 
In fact, in the above formula one can take the infimum over all real valuations of $K(X)$ with center on $X$. A key fact for the study of the log canonical threshold is that if 
$\pi\colon Y\to X$ 
is a log resolution of the ideal $\fra$, that is, $\pi$ is proper and birational, $Y$ is smooth, and
$\fra\cdot\cO_Y$ is the ideal of a divisor $D$ such that $D+K_{Y/X}$ has simple normal crossings,
then there is a prime divisor $E$ on $Y$ that achieves the infimum in~\eqref{eq1_introduction}.  
These divisors play an important role in understanding the singularities of $\fra$.

In this paper we undertake the systematic study of similar invariants in the case of sequences of ideals. We focus on \emph{graded sequences of ideals} $\fra_{\bullet}$: these are sequences of 
ideals $(\fra_m)_{m\geq 1}$ on $X$ such that $\fra_p\cdot\fra_q\subseteq\fra_{p+q}$ for all $p$ and 
$q$. In order to simplify the statements, in this introduction we also assume that all $\fra_m$ are nonzero. The main geometric example of a graded sequence is given by the ideals defining the
base locus of $|L^m|$, where $L$ is an effective line bundle  on the smooth projective variety $X$. Note that the interesting behavior of this sequence takes place when the section 
$\CC$-algebra $\oplus_{m\geq 0}\Gamma(X,L^m)$ is not finitely generated. 

Given a graded sequence $\fra_{\bullet}$, 
one can define an \emph{asymptotic log canonical threshold} $\lct(\fra_{\bullet})$ as the limit 
\begin{equation*}
  \lct(\fra_{\bullet}):=\lim_{m\to\infty}m\cdot\lct(\fra_m)
  =\sup_mm\cdot\lct(\fra_m)\in\RR_{\geq 0}\cup
  \{\infty\}.
\end{equation*}
We show that as above, we have
\begin{equation}\label{eq2_introduction}
  \lct(\fra_{\bullet})=\inf_E\frac{A(\ord_E)}{\ord_E(\fra_{\bullet})},
\end{equation}
where 
$\ord_E(\fra_{\bullet})=\lim_{m\to\infty}\frac{\ord_E(\fra_m)}{m}=\inf_m\frac{\ord_E(\fra_m)}{m}$.
More generally, one can define $v(\fra_{\bullet})$ and $A(v)$ for every  valuation $v$ of 
$K(X)$ with center on $X$, and in~\eqref{eq2_introduction} we may take the infimum over all such valuations different from the trivial one. 
It is easy to see that in this setting there might be no divisor $E$ such that the infimum 
in~\eqref{eq2_introduction} is achieved by $\ord_E$ (we give such an example with $\fra_{\bullet}$
a graded sequence of monomial ideals in \S 8).  
The following is (a special case of) our first main result:

\begin{thmA}
  For every graded sequence of ideals $\fra_{\bullet}$, there is a
  real valuation $v$ of $K(X)$ with center on $X$ that computes 
  $\lct(\fra_{\bullet})$, that is, such that
  $\lct(\fra_{\bullet})=\frac{A(v)}{v(\fra_{\bullet})}$. 
\end{thmA}
We make the following conjecture.
\begin{conjB}
  Let $\fra_\bullet$ be a graded 
  sequence of ideals on $X$ such that $\lct(\fra_\bullet)<\infty$.
  \begin{itemize}
  \item \textbf{Weak version}:
    there exists a quasi-monomial valuation $v$
    that computes $\lct(\fra_{\bullet})$.
  \item \textbf{Strong version}: 
    any valuation $v$ that computes
    $\lct(\fra_{\bullet})$ must be quasi-monomial.\footnote{Given $v$, 
    the existence of $\fra_{\bullet}$ such that $v$ computes 
    $\lct(\fra_{\bullet})<\infty$ is equivalent
    to the following two properties: $A(v)<\infty$, and for every valuation 
    $w$ of $K(X)$ with center 
    on $X$ such that $w(\fra)\geq v(\fra)$ for every ideal $\fra$ on $X$, we have $A(w)
    \geq A(v)$; see Theorem~\ref{equivalent_description}.}
  \end{itemize}
\end{conjB}
Recall that a quasi-monomial valuation is a valuation $v$ of $K(X)$ with the following property:
there is a proper birational morphism $Y\to X$, with $Y$ smooth,  and coordinates $y_1,\dots,y_r$
at a point $\eta\in Y$, as well as $\alpha_1,\dots,\alpha_r\in\RR_{\geq 0}$ such that 
if $f$ can be written at $\eta$ as $f=\sum_{\beta\in\ZZ_{\geq 0}^r}c_\beta y^\beta$, then
\begin{equation*}
  v(f)=\min\{\sum_i\alpha_i\beta_i\mid 
  \beta=(\beta_1,\dots,\beta_r)\in\ZZ^r_{\geq 0}, c_\beta\neq 0\}.
\end{equation*}
Equivalently, such valuations are known as Abhyankar valuations (see \S 3.2). 
A positive answer to the above conjecture could be interpreted as 
a finiteness property of graded sequences, even those sequence 
that are not finitely generated, that is, for which the 
$\cO_X$-algebra $\bigoplus_{m\geq 0}\fra_m$ is not finitely generated. 

As a consequence of Theorem A, we show that in order to prove Conjecture B
it suffices to consider certain special graded sequences $\fra_{\bullet}$,
namely  those attached to an arbitrary real valuation 
$w$ of $K(X)$, by taking $\fra_m=\{f\mid w(f)\geq m\}$. 
See Theorem~\ref{main_thm3}.

Our second main result reduces Conjecture~B to the case 
of affine space over an algebraically closed field. Furthermore, in the strong version we may assume that $v$ is a valuation of transcendence degree zero. 
In order to get such a statement, we need to work in a 
slightly more general setting that we now explain.
To a nonzero ideal $\fra$, one associates its multiplier ideals 
$\cJ(\fra^t)$, where $t\in\RR_{\geq 0}$.
These are ideals on $X$ with $\cJ(\fra^{t_1})\subseteq\cJ(\fra^{t_2})$
if $t_1> t_2$, and $\cJ(\fra^t)=\cO_X$ for $0\leq t\ll 1$. One knows that there is an unbounded sequence of positive rational numbers
$0<t_1<t_2<\dots$ such that $\cJ(\fra^t)$ is constant for $t\in [t_{i-1},t_i)$ 
and $\cJ(\fra^{t_i})\neq\cJ(\fra^{t_{i-1}})$ for all $i\geq 1$
(with the convention that $t_{0}=0$). These $t_i$ are the \emph{jumping numbers} of $\fra$, introduced and studied in~\cite{ELSV}. From this point of view, the log canonical threshold $\lct(\fra)$
is simply the smallest jumping number $t_1$. 

We index the jumping numbers of $\fra$ as follows. Given a nonzero ideal $\frq$ on $X$,
let $\lct^{\frq}(\fra)$ be the smallest $t$ such that $\frq\not\subseteq\cJ(\fra^t)$. In particular,
we recover $\lct(\fra)$ as $\lct^{\cO_X}(\fra)$. 
The advantage of considering higher jumping numbers comes from the fact that it allows
replacing $X$ by any smooth $X'$, where $X'$ is proper and birational over $X$: in this case
$\lct^{\frq}(\fra)=\lct^{\frq'}(\fra')$, where $\fra'=\fra\cdot\cO_{X'}$ and 
$\frq'=\frq\cdot\cO_{X'}(-K_{X'/X})$. Many of the subtle properties of the log canonical threshold
are not shared by the higher jumping numbers. However, for our purposes, considering also 
$\lct^{\frq}(\fra)$ does not create any additional difficulties. 

In particular, given a graded sequence of ideals $\fra_{\bullet}$, one defines as above
\begin{equation*}
\lct^{\frq}(\fra_{\bullet}):=\lim_{m\to\infty}m\cdot\lct^{\frq}(\fra_m)=\sup_mm\cdot\lct^{\frq}(\fra_m)\in \RR_{\geq 0}\cup\{\infty\},\end{equation*}
and we have
\begin{equation}\label{eq3_introduction}
\lct^{\frq}(\fra_{\bullet})=\inf_v\frac{A(v)+v(\frq)}{v(\fra_{\bullet})},
\end{equation}
where the infimum is over all real valuations of $K(X)$ with center on $X$, and different from the trivial one. With this notation we have versions of Theorem~A and Conjecture~B for $\lct^{\frq}(\fra_{\bullet})$. Furthermore, we reduce the general version of Conjecture~B to the following conjecture about valuations.

\begin{conjC}
  Let $X=\AAA_k^n$, where 
  $k$ is an algebraically closed field of characteristic zero
  and where $n\ge 1$.
  Let $\fra_{\bullet}$ be a graded sequence of ideals on $X$ and 
  $\frq$ a nonzero ideal on $X$ such that 
  $\lct^\frq(\fra_\bullet)<\infty$ and such that 
  $\fra_1\supseteq\frm^p$, where $p\ge 1$ and 
  $\frm=\frm_\xi$ is the ideal defining a closed point $\xi\in X$.
 \begin{itemize}
  \item \textbf{Weak version}:
    there exists a quasi-monomial valuation $v$ 
    computing $\lct^\frq(\fra_\bullet)$ and 
    having center $\xi$ on $X$.
  \item \textbf{Strong version}: 
    any valuation of transcendence degree 0
    computing $\lct^\frq(\fra_\bullet)$ and having center
    $\xi$ on $X$, must be quasi-monomial.
  \end{itemize} 
\end{conjC}
\begin{thmD}
  If Conjecture~C holds for all $n\leq d$, then 
  Conjecture~B holds for all $X$ with 
  $\dim(X)\leq d$.
\end{thmD}
We give a proof of the strong version of 
Conjecture~C in dimension $\leq 2$. 
The argument is similar to the one used in~\cite{FJ3}, where 
a version of Conjecture~B is proved. 
However, as opposed to~\cite{FJ3},
the proof given here does not use the detailed tree 
structure of the valuation space at a point.

As is always the case when dealing with graded sequences of ideals (see~\cite{positivity},
\cite{ELMNP},~\cite{Mustata} and~\cite{FJ3}), a key tool is provided by the corresponding system of asymptotic multiplier ideals $\frb_{\bullet}=(\frb_t)_{t\in\RR_{>0}}$. These are defined by 
$\frb_t=\cJ(\fra_m^{t/m})$ for $m$ divisible enough. The invariant $\lct^{\frq}(\fra_{\bullet})$ can be recovered 
as the smallest $\lambda$ such that $\frq\not\subseteq\frb_{\lambda}$. 
A fundamental property 
of $\frb_{\bullet}$ is provided by the Subadditivity Theorem~\cite{DEL}, which
says that $\frb_{s+t}\subseteq\frb_s\frb_t$ for every $s,t\ge0$.
Given a general such \emph{subadditive} systems of ideals $\frb_{\bullet}$ 
(not necessarily associated to a graded system) we 
introduce and study asymptotic invariants. 
A key property for us is that a graded sequence
$\fra_{\bullet}$ has, roughly speaking, the same 
asymptotic invariants as its system $\frb_\bullet$
of multiplier ideals.

We now describe the key idea in the proof of Theorems A and D. Given a graded sequence
$\fra_{\bullet}$ and a nonzero ideal $\frq$ with $\lambda=\lct^{\frq}(\fra_{\bullet})<\infty$, let $\xi$ be the generic point of an irreducible component of the subscheme defined by $(\frb_\lambda\colon\frq)$. After localizing and completing at $\xi$, we may assume that 
$X=\Spec\,k\llbracket x_1,\dots,x_n\rrbracket$ for a characteristic zero field $k$, and that $\xi$ is the closed point.
We show that if $\frm$ is the ideal defining $\xi$, and 
$p\gg 0$, then $\lct^{\frq}(\fra_{\bullet})=\lct^{\frq}(\frc_{\bullet})$, where $\frc_{\ell}=\sum_{i=0}^{\ell}
\fra_i\cdot\frm^{p(\ell-i)}$. 
Using a compactness argument for the space of normalized valuations with center at $\xi$, we construct a valuation $v$ with center at $\xi$, which computes $\lct^{\frq}(\frc_{\bullet})$. It is now easy to see that $v$ also computes $\lct^{\frq}(\fra_{\bullet})$. This proves  the general version of Theorem A. In order to prove Theorem D for the weak versions of the conjectures, we need two extra steps: we show that after replacing $X$ by a higher model, the valuation $v$ that we construct has transcendence degree zero, and then we show that we may replace $k$ by its algebraic closure $\overline{k}$, and 
$\Spec\,\overline{k}\llbracket x_1,\dots,x_n\rrbracket$ by $\AAA_{\overline{k}}^n$.
In this case, assuming Conjecture C, we can choose $v$ to be quasi-monomial.

A general principle in our work is to study a graded system 
$\fra_\bullet$ of ideals on $X$ through the induced function
$v\mapsto v(\fra_\bullet)$ on the space $\Val_X$ of
real-valued valuations on $K(X)$ admitting a center on $X$.
We show in Theorem~\ref{T301} 
that $\Val_X$ can be viewed as a projective limit of simplicial cone 
complexes equipped with an integral affine structure,
a description which leads us to extend the log discrepancy
from divisorial to arbitrary valuations. 
In fact, the precise understanding of the log discrepancy
plays a key role in the proof of Theorem~D. 

Spaces of valuations, such as Berkovich spaces~\cite{Ber1},
are fundamental objects in non-Archimedean geometry.
More surprisingly, they have recently seen a 
number of applications to problems over
the complex numbers~\cite{Ber2,KoSo,FJ2,FJ3,BFJ1,BdFF,Ked1,Ked2}.
The space $\Val_X$ is a dense subset of the Berkovich analytic 
space $\Xan$ and has the advantage of being birationally
invariant (as a set). It is also closely related to the valuation
space considered in~\cite{BFJ1}.
See~\S\ref{comparison} for more details.
Expecting the space $\Val_X$ to be useful for further studies,
we spend some time analyzing it in detail.
However, on a first reading, the reader may want to skim 
through~\S\S\ref{S304}-\ref{S306}.

We mention that part of our motivation comes from the Openness Conjecture of Demailly and 
Koll\'{a}r~\cite{DK} for plurisubharmonic (psh) functions. The connection between valuation theory and this conjecture has been highlighted by the two-dimensional result in~\cite{FJ3}, and by the higher-dimensional
framework in~\cite{BFJ1}. In the setting of psh functions,
one can define analogues of the invariant
$\lct(\fra_{\bullet})$, and one can formulate an analogue of Conjecture~B, which would imply in particular the Openness Conjecture.
 While in general there is no graded sequence associated to a psh function 
$\phi$, Demailly's approximation technique (see~\cite{DK}) 
allows one to get a subadditive system of ideals $\frb_{\bullet}$.
We expect that methods similar to the ones used in this paper should give analogues of Theorems~A and D for psh functions (in particular, this would reduce an analytic statement, the Openness Conjecture, to the valuation-theoretic Conjecture~C above). We hope to treat the
case of psh functions in future work.

As explained above, we make use of localization and completion.
Furthermore, when working in the analytic setting it is convenient to consider
schemes of finite type over rings of convergent power series over $\CC$. 
In order to cover all such cases, we work from the beginning with regular excellent schemes 
over $\QQ$, as in~\cite{dFM}. The basic results about log canonical thresholds and multiplier ideals carry over to this setting. Some of the more subtle results, whose proofs use vanishing theorems, are reduced to the familiar setting in the appendix.

The paper in structured as follows. In~\S\ref{S301} we set up some notation and definitions, 
and in~\S\ref{S302} we introduce  the asymptotic invariants for graded sequences and subadditive systems of ideals. We prove here their basic properties, and in particular, we relate the invariants of a graded
sequence and those of the corresponding subadditive system of asymptotic multiplier ideals.
In~\S\ref{S304} we introduce the quasi-monomial valuations and prove some general properties 
that will be needed later. Section~\ref{S305} contains some results concerning the structure of the valuation space,
while in~\S\ref{S306} we use this framework to extend the log discrepancy function to the whole valuation space. In~\S\ref{S305} and~\S\ref{S306} we follow the approach in~\cite{BFJ1}, 
with some modifications due to the fact that we do not restrict to valuations centered at a given point.
In~\S\ref{S307} we return to subadditive and graded sequences, and extend some results that we proved
for divisorial valuations to arbitrary valuations. 
Section~\ref{S308} is the central section of the paper, in which we prove our main results. In \S\ref{S101} we consider a special case,
that of graded sequences of monomial ideals. In this case the picture can be completely described, and in particular, we
see that Conjecture~B has a positive answer.
We give a proof of Conjecture~C in the two-dimensional case in~\S\ref{S310}.
The appendix shows how to extend some basic results about multiplier ideals, the Restriction and the Subadditivity Theorems, from the case of varieties over a field to our more general setting. 

\noindent\textbf{Acknowledgment.}
This work started as a joint project with Rob Lazarsfeld.  We remain
indebted to him for many inspiring discussions on this subject, and
also for sharing with us over the years his insights about multiplier
ideals and asymptotic invariants. The first author has also 
benefitted greatly from discussions with S\'ebastien Boucksom and Charles Favre.
Finally we are grateful to Michael Temkin for patiently 
answering our questions about resolution of singularities
and to the referee for a careful reading of the paper.

\section{Preliminaries}\label{S301}
Our main interest is in smooth algebraic varieties. However,
as we have already explained, it is more convenient to develop the whole 
theory in a general setting, when
 the ambient scheme $X$ is separated, regular, connected, 
 and excellent (we review the definition of excellent schemes in~\S1.1 below).
However, most of the time the reader will not lose much by 
assuming that we deal with separated, smooth algebraic varieties over
an algebraically closed field.

The main tool in our study is provided by multiplier ideals. For the theory of multiplier ideals
in the case of varieties over a field $k$ we refer to~\cite{positivity}. The definition and basic properties carry over easily  to our framework, see~\cite{dFM} and~\cite{dFEM}. 
For a small subtlety in computing the log discrepancy divisor in our setting, see 
\S 1.3 below.
The key fact that we have log resolutions in this setting follows from~\cite{Temkin}.
Certain care is only required when extending the results
that rely on vanishing theorems, since such results are not known in our framework. We explain in the appendix how the extension of some basic results, the Restriction and the Subadditivity Theorems,
can be carried out. From now on, without further discussion, we will not distinguish between the classical setting and ours when dealing with multiplier ideals. 

\subsection{Excellent schemes and regular morphisms}
Recall that a Noetherian scheme is \emph{regular} if all its local rings are regular. 
An effective divisor $D$ on a regular scheme $X$ has \emph{simple normal crossings} if 
at every point $\xi\in X$ there are algebraic coordinates $x_1,\ldots,x_r$ at $\xi$
(that is, a regular system of parameters of $\cO_{X,\xi}$) such that 
$D$ is defined at $\xi$ by $x_1^{a_1}\cdots x_r^{a_r}$, for some
$a_1,\ldots,a_r\in\ZZ_{\geq 0}$. 

A morphism $\mu\colon X'\to X$ between Noetherian schemes is \emph{regular}
if it is flat and all its fibers are geometrically regular 
(since all our schemes are schemes over $\QQ$, this simply means regular). 
An immediate consequence of the definition is that if $\mu$ is regular 
and $Y\to X$ is any morphism with $Y$ Noetherian, then $Y'\to Y$ is regular, where
$Y'=Y\times_XX'$.
In particular, if $Y$ is a regular scheme, then so is $Y'$;
similarly, if $D$ is a divisor on $Y$ having simple normal crossings, 
then so does its inverse image on $Y'$.

For an introduction to regular morphisms, see~\cite[Chapter 32]{Matsumura}. 
\begin{example}\label{E401}
  Let $K/k$ be an extension of fields of characteristic zero.
  Then the induced morphism $\varphi\colon\AAA^n_K\to\AAA^n_k$
  is regular and faithfully flat.
\end{example}

Recall that a Noetherian ring $A$ is \emph{excellent} if the
following hold:
\begin{enumerate}
\item[1)] For every prime ideal $\frp$ in $A$, the completion
morphism $A_{\frp}\to \widehat{A_{\frp}}$  corresponds to a regular
scheme morphism.
\item[2)] For every $A$-algebra of finite type $B$, the regular
locus of ${\rm Spec}(B)$ is open.
\item[3)] $A$ is universally catenary.
\end{enumerate}
A Noetherian scheme $X$ is
\emph{excellent} if it admits an open cover by spectra of excellent
rings. Note that by definition, if $X$ is an excellent scheme, then for every point $\xi\in X$
the canonical morphism $\Spec \widehat{\cO_{X,\xi}}\to X$ is regular.

For the basics on excellent rings we refer to~\cite[Chapter~32]{Matsumura},
and the references therein.
It is known that a localization of an algebra of finite type over an excellent ring
is excellent. Another important example of excellent rings is provided by local
complete Noetherian rings. In particular, formal power series rings 
over a field are excellent
(and the same holds for rings of convergent power series over $\CC$).

\subsection{Valuations}
From now on, we assume that $X$ is a separated, regular, connected, excellent scheme 
over $\QQ$.
We will consider the set 
$\Val_X$ of all real valuations of the function field $K(X)$ of $X$
that admit a center on $X$. The 
last condition means that if $\cO_v$ is the valuation ring of $v$, then there is a point $\xi=c_X(v)\in X$, the \emph{center} of $v$, such that we have
a local inclusion of local rings $\cO_{X,\xi}\hookrightarrow \cO_v$. 
Note that since $X$ is separated, the center is unique.
We sometimes call the closure
of $c_X(v)$ the center of $v$, too.
The \emph{trivial valuation} is the valuation with center 
at the generic point of $X$, or equivalently, whose restriction to $K(X)^*$
is identically zero.
We denote by $\Val_X^*\subseteq\Val_X$ the subset of nontrivial valuations.
Notice that if $X$ is a variety over a field $k$, 
then the restriction of any $v\in\Val_X$ to $k$ is the trivial valuation.

It is clear that for every $v$ as above, since the ring $\cO_{X,\xi}$ is Noetherian, there is no infinite decreasing sequence $v(f_1)>v(f_2)>\dots$, with all $f_i$ in $\cO_{X,\xi}$. Indeed,
the sequence of ideals $\fra_i=\{f\in\cO_{X,\xi}\mid v(f)\geq v(f_i)\}$ would be strictly increasing. In particular, we see that there is a minimal $v(f)$, where  $f$ varies over the maximal ideal of 
$\cO_{X,\xi}$.

Let $v\in\Val_X$, $\xi=c_X(v)$ and $\frm$ the maximal ideal of $\cO_{X,\xi}$.
By $\frm$-adic continuity, $v$ then extends uniquely as a 
\emph{semivaluation} on the completion $\widehat{\cO_{X,\xi}}$,
that is, a function $v:\widehat{\cO_{X,\xi}}\to\RR_{\ge0}\cup\{+\infty\}$
satisfying the usual valuation axioms.

If $v\in\Val_X$, and if $\fra$ is 
an ideal\footnote{By ``ideal on $X$'' we shall mean 
  ``coherent  ideal sheaf on $X$'' throughout the paper.}
on $X$, then we put 
$v(\fra):=\min_fv(f)$, where the minimum is over 
local sections of $\fra$ that are defined in a neighborhood of 
$c_X(v)$. 
If $Z$ is the subscheme defined by $\fra$, we also write this as $v(Z)$.
In fact, it turns out to be natural to instead
view a valuation as taking values on ideals rather than 
rational functions. 
Let $\cI$ be the set of nonzero ideals on $X$.
It has the structure of an ordered semiring, with the order given by inclusion, and the operations given by addition and multiplication.
The set $\RR_{\ge0}$ also has an ordered  semiring structure, with 
operations given by minimum and addition.
As above, a valuation $v\in\Val_X$ induces a function 
$v\colon \cI\to\RR_{\ge0}$
by $v(\fra):=\min\{v(f)\mid f\in\fra\cdot\cO_{X,\xi}\}$, 
where $\xi=c_X(v)$,
and this function is easily seen to be a homomorphism of semirings:
\begin{equation}\label{e210}
  v(\fra\cdot\frb)=v(\fra)+v(\frb)
  \quad\text{and}\quad
  v(\fra+\frb)=\min\{v(\fra),v(\frb)\}.
\end{equation}
Note that such a homomorphism is automatically order-preserving
in the sense $v(\fra)\geq v(\frb)$ if $\fra\subseteq\frb$. 
Indeed, $\fra\subseteq\frb$ implies $\fra+\frb=\frb$.
Moreover, the above homomorphism has $\xi$ as a center on $X$ in the sense
that $v(\fra)>0$ if and only if $\xi\in V(\fra)$.
Conversely, if $v\colon \cI\to\RR_{\ge0}$ is a semiring homomorphism
admitting $\xi\in X$ as a center, then $v$ induces a valuation
in $\Val_X$ centered at $\xi$. 
Indeed, if $f\in\cO_{X,\xi}$, then we define
$v(f):=v(\fra)$ for any ideal $\fra$ on $X$
such that $\fra\cdot\cO_{X,\xi}$ is principal, 
generated by $f$. 
One can check that this is well-defined, and it extends to a valuation of $K(X)$ having center 
at $\xi$. It is clear that these two maps between $\Val_X$ and
semiring homomorphisms $\cI\to\RR_{\geq 0}$ with center on $X$ are mutual inverses.

\subsection{Divisorial valuations and log discrepancy}
A distinguished role is played by the \emph{divisorial} valuations $\ord_E$, where $E$ is a 
\emph{divisor over} $X$, that is, a
prime divisor on a normal scheme $Y$, having a proper birational
morphism $\pi\colon Y\to X$. It follows from results on resolution of
singularities in this setting (see~\cite{Temkin}) that we may always
choose $Y$ regular, with $E$ a regular divisor. 

If $\pi\colon Y\to X$ is a proper, birational morphism between schemes as above (both of them regular),
we consider the $0^{\rm th}$ Fitting ideal $\Fitt_0(\Omega_{Y/X})$ of the relative sheaf
of differentials. 
Note
that $\Fitt_0(\Omega_{Y/X})$ defines the exceptional locus of $\pi$. 
As we will see in Corollary~\ref{K_divisor} below, this is 
a locally principal ideal, hence it defines
an effective divisor, the \emph{relative canonical divisor} $K_{Y/X}$. 
The \emph{log discrepancy} $A(\ord_E)$
is defined as 
\begin{equation*}
  A(\ord_E)
  =\ord_E(K_{Y/X})+1
  =\ord_E(\Fitt_0(\Omega_{Y/X}))+1.
\end{equation*}
The log discrepancy depends on the variety $X$; whenever there is some ambiguity, we denote it by $A_X(\ord_E)$.

There is some subtlety involved in the notion of 
log discrepancy in our setting, so we discuss
this briefly, using some notions and results from~\cite{dFEM}. 
The difficulty comes from the fact that our schemes
are not
of finite type over a field. 
In order to deal with this issue, we first consider the case when the schemes are of finite type
over a formal power series ring $R=k\llbracket t_1,\dots,t_n\rrbracket$, with $k$ a field. For each such scheme $X$ one introduced in~\cite{dFEM} a coherent \emph{sheaf of special differentials}
$\Omega'_{X/k}$, together with a (special) $k$-derivation $d'\colon\cO_X\to\Omega'_{X/k}$. 
If $\xi\in X$ is a regular point, then $\Omega'_{X/k,\xi}$
is a free $\cO_{X,\xi}$-module of rank $\dim(\cO_{X,\xi})+\dim_{k(\xi)}\Omega'_{k(\xi)/k}$, where
$k(\xi)$ is the residue field of $\xi$. 
Furthermore, if $x_1,\dots,x_r$ form a regular system of parameters at $\xi$, then $d'(x_1)_{\xi},\dots,d'(x_r)_{\xi}$ 
are part of a basis of $\Omega'_{X/k,\xi}$. 

Suppose now that $\pi\colon Y\to X$ is a proper birational morphism between schemes
as above. 
If $\eta\in Y$ and $\xi=\pi(\eta)\in X$, then the Dimension Formula (see 
\cite[Theorem~15.6]{Matsumura}) gives $\dim(\cO_{Y,\eta})=\dim(\cO_{X,\xi})-
\trdeg(k(\eta)/k(\xi))$. On the other hand, there are exact sequences
\begin{equation}\label{exact1}
\Omega'_{X/k,\xi}\otimes_{\cO_{X,\xi}}\cO_{Y,\eta}
\overset{T}\to\Omega'_{Y/k,\eta}\to\Omega_{Y/X,\eta}\to 0,
\end{equation}
\begin{equation}\label{exact2}
0\to \Omega'_{k(\xi)/k}\otimes_{k(\xi)}k(\eta)\to\Omega_{k(\eta)/k}\to\Omega_{k(\eta)/k(\xi)}\to 0.
\end{equation}
Note also that by definition $T(d'(x)\otimes 1)=d'(\pi^*(x))$. 
We see that $T$ is a morphism between free $\cO_{Y,\eta}$-modules of the same rank, hence
$\Fitt_0(\Omega_{Y/X})$ is generated at $\eta$ by ${\rm det}(T)$. In particular,
$\Fitt_0(\Omega_{Y/X})$ is a locally principal ideal.

The next lemma will allow us to reduce the case of arbitrary excellent schemes
to that of schemes of finite type over a formal power series ring over a field.

\begin{lemma}\label{lem_regular_morphism}
If $X'$, $X$, and $Y$ are regular, connected, excellent schemes, and
$\pi\colon Y\to X$ are $\mu\colon X'\to X$ are morphisms,
with $\pi$ proper and birational, and $\mu$ regular, then the following hold:
\begin{enumerate}
\item[(i)] $Y':=Y\times_XX'$ is regular and connected, and the canonical projection
$\pi'\colon Y'\to X'$ is proper and birational.
\item[(ii)] We have $\Fitt_0(\Omega_{Y'/X'})=\Fitt_0(\Omega_{Y/X})\cdot\cO_{Y'}$.
\end{enumerate}
\end{lemma}

\begin{proof}
  Since $\mu$ is a regular morphism, 
  its base-change $\nu\colon Y'\to Y$ is regular, too. 
  We deduce that $Y'$
  is a regular scheme, since $Y$ has this property. It is clear that $\pi'$ is proper,
  has connected fibers, and is
  an isomorphism over an open subset of $X'$. Therefore $Y'$ is connected and
 $\pi'$ is birational.
 The assertion in  (ii) follows from the fact that $\Omega_{Y'/X'}={\nu}^*(\Omega_{Y/X})$,
 while taking Fitting ideals commutes with pull-back. 
\end{proof}

\begin{corollary}\label{cor_regular_morphism}
  With the notation in Lemma~\ref{lem_regular_morphism}, 
  suppose $\eta\in Y$, $\xi=\pi(\eta)$,
  and $\mu\colon X'=\Spec \widehat{\cO_{X,\xi}}\to X$ 
  is the canonical morphism.
  If $\eta'\in Y'=Y\times_XX'$ lies over the closed point 
  $\xi'\in X'$ and over $\eta\in Y$, then the following hold:
  \begin{enumerate}
  \item[(i)] 
    $\cO_{Y',\eta'}\otimes_{\cO_{Y,\eta}}k(\eta)=k(\eta)$;
  \item[(ii)] 
    if $\Fitt_0(\Omega_{Y'/X'})$ is locally principal, 
    then $\Fitt_0(\Omega_{Y/X})$
    is principal at $\eta$;
  \item[(iii)] 
    $\dim(\cO_{Y',\eta'})=\dim(\cO_{Y,\eta})$
    and if  $y_1,\dots,y_s$ give algebraic coordinates at $\eta$, then
    so do ${\nu}^*(y_1),\ldots,{\nu}^*(y_s)$ at $\eta'$,
    where $\nu\colon Y'\to Y$ is the base change.
  \end{enumerate}
\end{corollary}

\begin{proof}
  Note first that since $X$ is excellent, 
  the morphism $\mu$ is indeed regular,
  hence so is the base change 
  $\nu\colon Y'\to Y$.
  The assertion in (i) follows from the fact that 
  $\nu(\eta')=\eta$ and $k(\xi')=k(\xi)$. 
  We deduce from (i) that $k(\eta')=k(\eta)$, and using
  Lemma~\ref{lem_regular_morphism}, that
  \begin{equation*}
    \dim_{k(\eta')}(\Fitt_0(\Omega_{Y'/X'})\otimes k(\eta'))
    =\dim_{k(\eta)}(\Fitt_0(\Omega_{Y/X})\otimes k(\eta)),
  \end{equation*}
  hence the minimal number of generators
  of $\Fitt_0(\Omega_{Y'/X'})$ at $\eta'$ 
  and that of $\Fitt_0(\Omega_{Y/X})$
  at $\eta$ are equal. This gives (ii). 
  It also follows from (i) that the extension of the 
  maximal ideal in $\cO_{Y,\eta}$ to
  $\cO_{Y',\eta'}$ is equal to the maximal ideal. 
  Since $\nu$ is flat, this implies the equality of dimensions in (iii),
  and the last assertion is clear, too.
\end{proof}

\begin{corollary}\label{K_divisor}
For every proper birational morphism $\pi\colon Y\to X$ between regular, connected, excellent schemes as above, the ideal $\Fitt_0(\Omega_{Y/X})$ is locally principal.
\end{corollary}

\begin{proof}
Let us show that $\Fitt_0(\Omega_{Y/X})$ is principal at 
any given point $\eta\in Y$. Let $\xi=\pi(\eta)$, and 
consider the regular morphism 
$\mu\colon X'=\Spec \widehat{\cO_{X,\xi}}\to X$. 
We keep the notation in Corollary~\ref{cor_regular_morphism}.
By Cohen's structure theorem, $\widehat{\cO_{X,\xi}}$ is isomorphic to 
a formal power series ring over $k(\xi)$, hence as we have seen, $\Fitt_0(\Omega_{Y'/X'})$
is locally principal. Therefore $\Fitt_0(\Omega_{Y/X})$ is principal at $\eta$ by
Corollary~\ref{cor_regular_morphism}. 
\end{proof}

\begin{lemma}\label{lem3_val}
  Let $\phi\colon Y'\to Y$ be a proper birational morphism 
  between regular, connected, excellent schemes.
  Consider $\eta'\in Y'$ and $\eta=\phi(\eta')\in Y$, and let us choose
  regular systems of parameters $\uy=(y_1,\dots,y_r)$ and
  $\underline{y'}=(y'_1,\dots,y'_s)$ at $\eta$ and $\eta'$, respectively.
  Suppose that
  \begin{equation*}
    \phi^*(y_i)=u_i\cdot\prod_{j=1}^s(y'_j)^{b_{i,j}},
  \end{equation*}
  for every $1\leq i\leq r$, and suitable $u_i\in\cO_{Y',\eta'}$. 
  If $D'_j$ denotes the closure of $V(y'_j)$, then
  \begin{enumerate}
  \item[(i)] 
    we have $A_Y(\ord_{D'_j})\geq\sum_{i=1}^rb_{i,j}$;
  \item[(ii)] 
    if $r=s$ and if the image of each $u_i$ in $k(\eta')$ 
    is nonzero, then we have equality 
    in~$\mathrm{(i)}$ if and only if  $\det(b_{i,j})\neq 0$. 
  \end{enumerate}
\end{lemma}
\begin{proof}
  Note first that we may assume that $Y$ and $Y'$ are schemes of finite type 
  over a formal power series ring over a field. 
  Indeed, let $\mu\colon Z=\Spec \widehat{\cO_{Y,\eta}}\to Y$ 
  be the canonical morphism; this is regular since $Y$ is excellent. 
  Set $Z':=Y'\times_YZ$ and denote the two projections by
  $\mu'\colon Z'\to Y'$ and $\phi'\colon Z'\to Z$.
  Let $\zeta\in Z$ denote the closed point, and let 
  $\zeta'\in Z'$ be a point such that $\mu'(\zeta')=\eta'$
  and $\phi'(\zeta')=\zeta$.
  It follows from Corollary~\ref{cor_regular_morphism} that
  $({\mu'}^*(y'_j))_j$ gives a regular system of parameters at $\zeta'$, 
  and it is clear that $(\mu^*(y_i))_i$ is a regular system of parameters at $\zeta$. 
  Using Lemma~\ref{lem_regular_morphism}, we see that it is enough to prove the statement for
  $\phi'$. By Cohen's structure theorem, $\widehat{\cO_{Y,\eta}}$ is isomorphic to
  $k(\eta)\llbracket t_1,\ldots,t_r\rrbracket$, hence we may assume 
  that $Y$ and $Y'$ are of finite type over a formal power series ring over a field.
  
  With notation analogous to~\eqref{exact1}, we see that 
  \begin{equation*}
    T(d'(y_{\ell}))\in B\cdot\Omega'_{Y'/k,\eta'}
    +\sum_{j}\frac{B}{y'_j}d'(y'_j),
  \end{equation*}
  where $B=\prod_{j=1}^s(y'_j)^{b_{i,j}}$, and the sum is over those $j$ with
  $b_{i,j}>0$.  
  The assertion in~(i) follows from this and from our description
  of $\Omega_{Y/k,\eta}$ and $\Omega_{Y'/k,\eta'}$. 
  Furthermore, an easy (and well-known) computation shows that if $r=s$, 
  and if we write $\det(T)=\prod_{j=1}^s(y'_j)^{b_j}\cdot g$, where
  $b_j+1=\sum_{i=1}^rb_{i,j}$ for every $j$, then the image of $g$ in  $k(\eta')$
  is equal to $\det(b_{i,j})\cdot \prod_{i=1}^s\overline{u_i}$, 
  where $\overline{u_i}$ denotes the image of $u_i$, which is nonzero. 
  This gives the assertion in~(ii).
\end{proof}
\begin{remark}
  The estimate in Lemma~\ref{lem3_val} (i)
  was claimed without any details in the proof of~\cite[Proposition~2.2]{dFM}. 
  The sheaves of special differentials were introduced 
  in~\cite{dFEM} partly to explain this estimate.
  As we have seen above, when working with regular excellent schemes, 
  one can always reduce to the case of schemes of finite type over a 
  formal power series ring over a field. However, when 
  considering also singular schemes, as in~\cite{dFEM}, 
  the situation is more complicated, since
  the sheaves of special differentials also appear in the 
  definition of the relative canonical divisor. 
\end{remark}

An important result of~\cite{Temkin}, generalizing Hironaka's theorem for varieties over a field, guarantees the
existence of log resolutions in our setting: given an ideal $\fra$ on $X$, there is a projective
birational morphism $\pi\colon Y\to X$ such that $Y$ is regular, $\fra\cdot \cO_Y$
is the ideal of a divisor $D$, and $D+K_{Y/X}$ is a divisor with simple normal crossings.
This is what allows one to develop the theory of multiplier ideals in this setting.

Recall that given a nonzero ideal $\fra$ and $\lambda\in\RR_{\geq 0}$, the multiplier ideal $\cJ(\fra^{\lambda})$ is the ideal on $X$
consisting of those local sections $f$ of $\cO_X$ such that
\begin{equation*}\ord_E(f)+A(\ord_E)>\lambda\cdot\ord_E(\fra)\end{equation*}
for all divisors $E$ over $X$ such that $f$ is defined at $c_X(\ord_E)$. In fact, it is enough to 
only consider
those divisors $E$ that appear on any given log resolution of $\fra$. This follows as in the case of 
schemes of finite type over a field once we have the inequality in Lemma~\ref{lem3_val}~(i).
We make the convention that if $\fra=(0)$, then $\cJ(\fra^{\lambda})=\cO_X$ if $\lambda=0$,
and it is the zero ideal, otherwise.

\subsection{Jumping numbers}
For every  ideal $\fra$ on $X$, we index the jumping numbers of $\fra$, as follows.
Given a nonzero ideal $\frq$ on $X$, we consider the \emph{log canonical threshold of $\fra$
with respect to $\frq$}
\begin{equation*}\lct^{\frq}(\fra):=\min\{\lambda\geq 0\mid \frq\not\subseteq \cJ(\fra^{\lambda})\}\end{equation*}
(with the convention $\lct^{\frq}(\fra)=\infty$ if $\fra=\cO_X$).
Note that when $\frq=\cO_X$, this is simply the \emph{log canonical threshold} 
$\lct(\fra)$ and as we vary $\frq$, we recover in this way all the jumping numbers of
$\fra$, in the sense of~\cite{ELSV}. 
It is convenient to also consider the reciprocals of these numbers.
We define the \emph{Arnold multiplicity} of $\fra$ with respect to $\frq$ to be
$\Arn^{\frq}(\fra):=\lct^{\frq}(\fra)^{-1}$ (if $\frq=\cO_X$, we simply write $\Arn(\fra)$). 
If $Z$ is the subscheme defined by $\fra$ we sometimes write $\Arn^{\frq}(Z)$ for $\Arn^{\frq}(\fra)$.
Note that
$\Arn^{\frq}(\fra)=0$ if and only if $\fra=\cO_X$, and $\Arn^{\frq}(\fra)=\infty$
if and only if $\fra=(0)$. 

\begin{lemma}\label{lem1}
  If $\pi\colon Y\to X$ is a log resolution of the 
  nonzero ideal $\fra$, and if $\fra\cdot \cO_Y=\cO_Y(-\sum_i\alpha_iE_i)$
  and $K_{Y/X}=\sum_i\kappa_iE_i$, then for every nonzero ideal $\frq$ 
  \begin{equation}\label{eq_lem1}
    \Arn^{\frq}(\fra)
    =\max_i\frac{\alpha_i}{\kappa_i+1+\ord_{E_i}(\frq)}
    =\max_i\frac{\ord_{E_i}(\fra)}{A(\ord_{E_i})+\ord_{E_i}(\frq)}.
  \end{equation}
  Moreover, we have
  \begin{enumerate}
  \item[(i)] 
    If $\fra\subseteq\frb$, then $\Arn^{\frq}(\fra)\geq\Arn^{\frq}(\frb)$;
  \item[(ii)] 
    $\Arn^{\frq}(\fra^m)=m\cdot\Arn^{\frq}(\fra)$ for every $m\geq 1$;
  \item[(iii)] 
    $\Arn^{\frq_1+\frq_2}(\fra)=\max_{i=1,2}\Arn^{\frq_i}(\fra)$;
  \item[(iv)]
    $\Arn^{\frq}(\fra\cdot\frb)\leq\Arn^{\frq}(\fra)+\Arn^{\frq}(\frb)$
    for every ideals $\fra$ and $\frb$.
  \end{enumerate}
\end{lemma}
\begin{proof}
  Equation~\eqref{eq_lem1} is a consequence of the
  description of multiplier ideals in terms of a log resolution.
  Properties~(i)--(iii) follow from the definition whereas~(iv)
  is a consequence of~\eqref{eq_lem1}.
\end{proof}

If the maximum in~\eqref{eq_lem1} is achieved for $E_i$, we say that $\ord_{E_i}$
\emph{computes} $\lct^{\frq}(\fra)$ (or $\Arn^{\frq}(\fra)$).
It is natural to consider the invariants $\Arn^{\frq}(\fra)$ also for
$\frq\neq\cO_X$, 
since this case 
naturally appears when considering pull-backs by birational morphisms, as in 
\begin{corollary}\label{cor0}
  If $\phi\colon X'\to X$ is a proper birational morphism, with both $X$ and $X'$ regular,
  then for all ideals $\fra$, $\frq$ on $X$, with $\frq$ nonzero, we have
  \begin{equation*}
    \Arn^{\frq}(\fra)=\Arn^{\frq'}(\fra'),
  \end{equation*}
  where $\fra'=\fra\cdot\cO_{X'}$ and $\frq'=\frq\cdot\cO_{X'}(-K_{X'/X})$. 
\end{corollary}
\begin{proof}
  If $\fra=(0)$, then the assertion is clear. If this is not the case, 
  let $\phi'\colon X''\to X'$
  be a log resolution of $\fra'\cdot\cO_{X'}(-K_{X'/X})$;
  in particular $\phi\circ\phi'$ is a log resolution of $\fra$.
  The assertion in the corollary follows from (\ref{eq_lem1}), using the fact that
  for every divisor $E$ on $X''$, we have
  $\ord_E(K_{X''/X})=\ord_E(K_{X''/X'})+\ord_E((\phi')^*(K_{X'/X}))$.
\end{proof}
\begin{proposition}\label{P302}
  Let $\fra$ and $\frq$ be nonzero ideals on $X$.
  Let $\phi\colon X'\to X$ be a regular morphism
  and write $\fra':=\fra\cdot\cO_{X'}$, $\frq':=\frq\cdot\cO_{X'}$.
  Then $\cJ(\fra'^t)=\cJ(\fra^t)\cdot\cO_{X'}$ 
  for every $t\geq 0$. 
  In particular,  $\lct^{\frq'}(\fra')\geq\lct^{\frq}(\fra)=:\lambda$
  with equality  if 
  $V(\cJ(\fra^{\lambda})\colon\frq)\cap\phi(X')\neq\emptyset$.
  Further, the latter condition holds if $\phi$ is faithfully flat.
\end{proposition}
\begin{proof}
  Let $\pi\colon Y\to X$ be a log resolution
  of $\fra$, with $\fra\cdot\cO_Y=\cO_Y(-D)$.
  It follows from Lemma~\ref{lem_regular_morphism} that $Y'=Y\times_XX'$
  is regular and connected, and 
  $\pi'\colon Y'\to X'$ is birational and proper (in fact projective, 
  since $\pi$ is projective). 
  In addition, 
  if $\psi\colon Y'\to Y$ is the projection,
  then $\fra'\cdot\cO_{Y'}=\cO_{Y'}(-\psi^*(D))$, 
  and $\psi^*(D)+K_{Y'/X'}=\psi^*(D+K_{Y/X})$ 
  has simple normal crossings. 
  It now follows from 
  base-change with respect to flat morphisms that
  $\cJ(\fra'^t)=\cJ(\fra^t)\cdot\cO_{X'}$ for every $t\in\RR_{\geq 0}$. 

  If $0\le t<\lambda$, then  $\frq\subseteq\cJ(\fra_{\bullet}^t)$, hence
  $\frq'\subseteq\cJ(\fra^t)\cdot\cO_{X'}$. 
  Therefore,  $\lct^{\frq'}(\fra'_{\bullet})\geq\lct^{\frq}(\fra_{\bullet})$.
  Now suppose 
  $\phi^{-1}(V(\cJ(\fra_{\bullet}^{\lambda})\colon\frq))\neq\emptyset$. 
  In this case, since $\phi$ is flat we have
  $(\cJ(\fra'^t)\colon\frq')=(\cJ(\fra^t)\colon\frq)\cdot\cO_{X'}\neq\cO_{X'}$,
  hence $\frq'\not\subseteq\cJ(\fra'^t)$ and so
  $\lct^{\frq'}(\fra')=\lambda$. 
  Finally, note that since 
  $(\cJ(\fra_{\bullet}^{\lambda})\colon\frq)\neq\cO_X$, 
  if $\phi$ is faithfully flat, then it is surjective, hence 
  $\phi^{-1}(V(\cJ(\fra_{\bullet}^{\lambda})\colon\frq))$ is clearly nonempty.
\end{proof}

\section{Graded and subadditive systems of ideals}\label{S302}
We now introduce the main objects that we wish to study.
\subsection{Graded sequences}
A \emph{graded sequence} of ideals $\fra_{\bullet}=(\fra_m)_{m\in\ZZ_{>0}}$ 
is a sequence of ideals on $X$ that satisfies
$\fra_p\cdot\fra_q\subseteq\fra_{p+q}$ for every $p$, $q\geq 1$. 
We always assume that such a sequence is \emph{nonzero}, 
in the sense that $\fra_m\neq (0)$
for some $m$. 
Then $S=S(\fra_{\bullet}):=\{m\in\ZZ_{>0}\mid\fra_m\neq (0)\}$ 
is a subsemigroup of the positive integers (with respect to
addition). By convention we put $\fra_0=\cO_X$.

\begin{example}\label{E403}
 The most interesting geometric examples arise as follows: 
  suppose that $L$ is a line bundle
  of nonnegative Kodaira dimension 
  on a smooth projective variety $X$. If $\fra_m$ 
  is the ideal defining the base locus of 
  $L^m$, then $(\fra_m)_{m\geq 1}$ is a graded sequence of ideals. 
\end{example}
\begin{example}\label{E404}
  To any valuation $v\in\Val^*_X$, we can associate 
  a graded sequence $\fra_\bullet=\fra_\bullet(v)$ 
  of \emph{valuation ideals} given by 
  $\fra_m(v)=\{v\ge m\}$.
  More precisely, 
  for an affine open subset $U$ of $X$ we have 
  $\Gamma(U,\fra_m)=\{f\in\cO_X(U)\mid v(f)\geq m\}$ 
  if $c_X(v)\in U$, and $\Gamma(U,\fra_m)=\cO_X(U)$, otherwise. 
  Note that $\fra_m$ is nonzero since $v$ is nontrivial.
\end{example}
Following~\cite{ELMNP}, one can attach asymptotic invariants to graded sequences
of ideals, via the following well-known result. We include a proof, for the reader's convenience.

\begin{lemma}\label{lem2_1}
  Let $(\alpha_m)_{m\geq 1}$ be a sequence of elements in 
  $\RR_{\geq 0}\cup\{\infty\}$, that satisfies
  $\alpha_{p+q}\leq\alpha_p+\alpha_q$ for all $p$ and $q$. If the set
  $S:=\{m\mid \alpha_m<\infty\}$ is nonempty, then $S$ is a 
  subsemigroup of $\ZZ_{>0}$, and
  \begin{equation*}
    \lim_{m\to\infty, m\in S}\frac{\alpha_m}{m}
    =\inf_{m\geq1}\frac{\alpha_m}{m}.
  \end{equation*}
\end{lemma}
\begin{proof}
  Let $T:=\inf_{m\geq 1}\alpha_m/m$. 
  We need to show that for every $\tau>T$, we have
  $\alpha_p/p<\tau$ if $p\gg 0$, with $p\in S$. 
  Let $m$ be such that $\alpha_m/m<\tau$.
  It is enough to show that for every integer $q$ with $0\leq q <m$, 
  if $p=m\ell+q\in S$ with $\ell\gg 0$, then $\alpha_p/p<\tau$.

  If there is no $\ell$ such that $m\ell+q\in S$, then there is
  nothing to prove. 
  Otherwise, let us
  choose $\ell_0$ with $m\ell_0+q\in S$. For $\ell\geq\ell_0$ we have
  \begin{equation*}
    \frac{\alpha_{m\ell+q}}{m\ell+q}
    \leq\frac{\alpha_{m\ell_0+q}+(\ell-\ell_0)\alpha_m}{m\ell+q}.
  \end{equation*}
  Since the right-hand side converges to $\alpha_m/m<\tau$ 
  for $\ell\to\infty$, it follows that
  \begin{equation*}
    \frac{\alpha_{m\ell+q}}{m\ell+q}<\tau\ \text{for}\ \ell\gg 0,
  \end{equation*}
  which completes the proof.
\end{proof}

Suppose now that $\fra_{\bullet}$ is a graded sequence of ideals, 
and $v\in\Val_X$. By taking $\alpha_m=v(\fra_m)$, we define as in 
\cite{ELMNP}
\begin{equation*}v(\fra_{\bullet}):=\inf_{m\geq 1}\frac{v(\fra_m)}{m}=\lim_{m\to\infty,m\in S(\fra_{\bullet})}
\frac{v(\fra_m)}{m}.\end{equation*}
By the definition of a graded sequence we have $v(\fra_{p+q})\leq
v(\fra_p\cdot\fra_q)=v(\fra_p)+v(\fra_q)$.

\begin{lemma}\label{L301}
 Let $\fra_\bullet(v)$ be the graded sequence of 
  valuation ideals associated to a nontrivial valuation
  $v\in\Val_X^*$.
  Then, for any $w\in\Val_X$ we have
 \begin{equation*}
   w(\fra_{\bullet}(v))=\inf\frac{w(\frb)}{v(\frb)},
 \end{equation*}
 where $\frb$ ranges over ideals on $X$ for which $v(\frb)>0$.
 In particular, $v(\fra_{\bullet}(v))=1$.
\end{lemma}
\begin{proof}
 Note first that if 
  $c:=\inf\{w(\frb)/v(\frb)\mid v(\frb)>0\}$, then by definition we have
  $w(\fra_m(v))\geq c\cdot v(\fra_m(v))\geq cm$. Dividing by $m$ and
  letting $m\to\infty$ gives
  $w(\fra_{\bullet}(v))\geq c$. 
  For the reverse inequality, it is enough to show that 
  for every $\epsilon>0$, we have $w(\fra_{\bullet}(v))<c+\epsilon$. 
  By definition of $c$,
  there is an ideal $\frb$ on $X$ such that $v(\frb)>0$ 
  and $w(\frb)/v(\frb)<c+\epsilon$. 
  For every $m\ge 1$, 
  we have $\frb^m\subseteq\fra_{\lfloor m\cdot v(\frb)\rfloor}(v)$. 
  Therefore $w(\fra_{\lfloor m\cdot v(\frb)\rfloor}(v))\leq m\cdot w(\frb)$, 
  and so
  \begin{equation*}
    \frac{w(\fra_{\lfloor m\cdot v(\frb)\rfloor}(v))}
    {\lfloor m\cdot v(\frb)\rfloor}
    \leq
    \frac{m\cdot w(\frb)}{\lfloor m\cdot v(\frb)\rfloor}.
  \end{equation*}
  As $m\to\infty$ we get 
  $w(\fra_{\bullet}(v))\leq\frac{w(\frb)}{v(\frb)}<c+\epsilon$. 
  The last assertion about $v$ is clear.
\end{proof}

Similarly, by taking 
$\alpha_m=\Arn^{\frq}(\fra_m)$, where $\frq$ is a nonzero ideal, we get
\begin{equation*}\Arn^{\frq}(\fra_{\bullet}):=\inf_{m\geq 1}\frac{\Arn^{\frq}(\fra_m)}{m}=\lim_{m\to\infty,
m\in S(\fra_{\bullet})}\frac{\Arn^{\frq}(\fra_m)}{m}.\end{equation*}
The fact that the conditions in the lemma are satisfied follows from
the defining property of a graded sequence, together with
Lemma~\ref{lem1}~(iv).

We also write $\lct^{\frq}(\fra_{\bullet})=1/\Arn^{\frq}(\fra_{\bullet})$ 
(with the convention 
$\lct^{\frq}(\fra_{\bullet})=\infty$ if $\Arn^{\frq}(\fra_{\bullet})=0$). 
Note that both $v(\fra_{\bullet})$ and $\Arn^{\frq}(\fra_{\bullet})$ are finite.

\begin{proposition}\label{prop_birational}
  If $X'\to X$ is a proper birational morphism, 
  with both $X$ and $X'$ regular,
  then for every graded sequence of ideals 
  $\fra_{\bullet}$ on $X$, and every nonzero ideal  $\frq$ on $X$, we have
  \begin{equation*}
    \Arn^{\frq}(\fra_{\bullet})=\Arn^{\frq'}(\fra'_{\bullet}),
  \end{equation*}
  where $\fra_m'=\fra_m\cdot\cO_{X'}$ and $\frq'=\frq\cdot\cO_{X'}(-K_{X'/X})$. 
\end{proposition}
\begin{proof}
  The assertion follows by applying Corollary~\ref{cor0} 
  to each $\fra_m\ne0$, and then letting $m$ go to infinity.
\end{proof}

\subsection{Subadditive systems}
A \emph{subadditive system of ideals} $\frb_{\bullet}$
is a one-parameter family $(\frb_t)_{t\in\RR_{>0}}$  
of nonzero ideals satisfying $\frb_{s+t}\subseteq\frb_s\cdot\frb_t$ 
for every $s,t\in\RR_{>0}$. Note that this implies 
$\frb_s\subseteq\frb_t$ for $s\ge t$. 
By convention we put $\frb_0=\cO_X$.

To any such system of ideals, we can associate 
various invariants via the following lemma
(compare with Lemma~2.2 in~\cite{Mustata}).

\begin{lemma}\label{lem2}
  If $\phi\colon\RR_{>0}\to \RR_{\geq 0}$ is an increasing
  function such that  
  $\phi(mt)\geq m\phi(t)$ for every $t\in\RR_{>0}$ and $m\in\ZZ_{>0}$, then
  $\lim_{t\to\infty}\frac{\phi(t)}{t}$ exists in 
  $\RR_{\geq0}\cup\{\infty\}$, 
  and equals $\sup_{t>0}\frac{\phi(t)}{t}$.
\end{lemma}
\begin{proof}
  It is enough to show that for every $\tau<\sup_{t>0}\frac{\phi(t)}{t}$, we have
  $\frac{\phi(t)}{t}>\tau$ for $t\gg 0$. Choose $s>0$ such that $\frac{\phi(s)}{s}>\tau$.
  We claim that $\frac{\phi(t)}{t}>\tau$ as long as $1-\frac{s}{t}>\frac{\tau s}{\phi(s)}$. Indeed,
  in this case if we choose $m\in\ZZ_{>0}$ such that $ms\leq t<(m+1)s$, then
  \begin{equation*}\frac{\phi(t)}{t}\geq \frac{\phi(ms)}{ms}\cdot \frac{ms}{t}\geq \frac{\phi(s)}{s}\cdot \left(1-\frac{s}{t}\right)
  >\tau,\end{equation*}
  which completes the proof.
\end{proof}
The lemma applies to $\phi(t):=v(\frb_t)$ where $v\in\Val_X$ and 
$\frb_{\bullet}$ is a subadditive system of ideals.
Indeed, if $s\ge t$, then $\frb_s\subseteq\frb_t$ and so $v(\frb_s)\ge v(\frb_t)$.
Similarly, $\frb_{mt}\subseteq\frb_t^m$, hence 
$v(\frb_{mt})\ge v(\frb_t^m)=mv(\frb_t)$.
We put $v(\frb_{\bullet}):=\lim_{t\to\infty}\frac{v(b_t)}{t}$.

\begin{example}\label{toofast}
  We may have $v(\frb_{\bullet})=\infty$. For example, if $\fra$ is the ideal defining
  a closed point $\xi\in X$, then we have a subadditive system of ideals $\frb_{\bullet}$, where
  $\frb_t=\fra^{\lfloor t^2\rfloor}$ for all $t>0$. It is clear that for every $v\in\Val_X$ with
  $c_X(v)=\xi$, we have $v(\frb_{\bullet})=\infty$.
  For more interesting examples, see \S\ref{comparison}.
\end{example}

Similarly, if $\frq$ is a nonzero ideal, it follows from 
Lemma~\ref{lem1} that $\phi(t):=\Arn^{\frq}(\frb_t)$
satisfies the hypotheses in Lemma~\ref{lem2}. We 
define the \emph{asymptotic Arnold multiplicity} of $\frb_{\bullet}$ with respect to $\frq$ by
$\Arn^{\frq}(\frb_{\bullet}):=\lim_{t\to\infty}\frac{\Arn^{\frq}(\frb_t)}{t}$.
We also put $\lct^{\frq}(\frb_{\bullet})=1/\Arn^{\frq}(\frb_{\bullet})$.

One can give an alternative description of the asymptotic Arnold multiplicity:
\begin{proposition}\label{prop1}
If $\frb_{\bullet}$ is a subadditive system of ideals, and if $\frq$ is a nonzero ideal, then
\begin{equation}\label{formula_Arn}
\Arn^{\frq}(\frb_{\bullet})=\sup_E\frac{\ord_E(\frb_{\bullet})}{A(\ord_E)+\ord_E(\frq)},
\end{equation}
where the supremum is over all divisors $E$ over $X$.
\end{proposition}

\begin{proof}
  Given any divisor $E$ over $X$, we have by Lemma~\ref{lem1}
  \begin{equation*}
    \Arn^{\frq}(\frb_t)\geq\frac{\ord_E(\frb_t)}{A(\ord_E)+\ord_E(\frq)}.
  \end{equation*}
  Dividing by $t$, and letting $t$ go to infinity gives 
  ``$\geq$" in~\eqref{formula_Arn}.

  For the reverse inequality, fix $\tau<\Arn^{\frq}(\frb_{\bullet})$,
  pick $t$ such that $\tau<\Arn^{\frq}(\frb_t)/t$, and choose a divisor $E$ over $X$
  such that $\Arn^{\frq}(\frb_t)=\frac{\ord_E(\frb_t)}{A(\ord_E)+\ord_E(\frq)}$. 
  Then
  \begin{equation*}
    \tau<\frac{\Arn^{\frq}(\frb_t)}{t}
    =\frac{\ord_E(\frb_t)}{t(A(\ord_E)+\ord_E(\frq))}
    \leq\frac{\ord_E(\frb_{\bullet})}{A(\ord_E)+\ord_E(\frq)},
  \end{equation*}
  which proves ``$\leq$" in~\eqref{formula_Arn}.
\end{proof}

As Example~\ref{toofast} indicates, the ideals $\frb_t$ in a subadditive system 
$\frb_\bullet$ can ``grow'' very fast as $t\to\infty$. 
On the other hand, as we shall see in the next subsection, this does not happen for 
subadditive systems arising from graded sequences.
In order to formalize things, we introduce
\begin{definition}\label{growth}
  A subadditive system $\frb_\bullet$ has \emph{controlled growth}   if
  \begin{equation}\label{cond3}
    \frac{\ord_E(\frb_t)}{t}> \ord_E(\frb_{\bullet})-\frac{A(\ord_E)}{t}
  \end{equation}
  for every divisor $E$ over $X$ and every $t>0$. 
\end{definition}
In particular, if $\frb_\bullet$ has controlled growth, 
$\ord_E(\frb_{\bullet})$ is finite for all $E$.
\begin{lemma}\label{lem3}
  If $\frb_{\bullet}$ is a subadditive system of controlled growth, then for every nonzero ideal
  $\frq$ and every $t>0$, we have
  \begin{equation*}
    \frac{\Arn^{\frq}(\frb_t)}{t}\geq\Arn^{\frq}(\frb_{\bullet})-\frac{1}{t}.
  \end{equation*}
\end{lemma}

\begin{proof}
  By the definition of $\Arn^\frq(\frb_\bullet)$ it is enough to show that 
  \begin{equation*}
    \frac{\Arn^{\frq}(\frb_t)}{t}>\frac{\Arn^{\frq}(\frb_s)}{s}-\frac{1}{t}
  \end{equation*}
  for every $s>0$. Choose a divisor $E$ over $X$ such that 
  $\Arn^{\frq}(\frb_s)=\frac{\ord_E(\frb_s)}{A(\ord_E)+\ord_E(\frq)}$. 
  Using Lemma~\ref{lem1} and condition~\eqref{cond3}, we deduce
  \begin{multline*}
    \frac{\Arn^{\frq}(\frb_t)}{t}
    \geq \frac{\ord_E(\frb_t)}{t(A(\ord_E)+\ord_E(\frq))}
    >\left(\frac{\ord_E(\frb_{s})}{s}-\frac{A(\ord_E)}{t}\right)
    \cdot\frac{1}{A(\ord_E)+\ord_E(\frq)}\\
    =\frac{\Arn^{\frq}(\frb_s)}{s}-\frac{1}{t}\cdot\frac{A(\ord_E)}{A(\ord_E)+\ord_E(\frq)}
    \geq \frac{\Arn^{\frq}(\frb_s)}{s}-\frac{1}{t},
  \end{multline*}
  concluding the proof.
\end{proof}
\begin{corollary}
If $\frb_{\bullet}$ is a subadditive system of ideals of controlled growth, then 
$\Arn^{\frq}(\frb_{\bullet})$ is finite for every nonzero ideal $\frq$.
\end{corollary}

\subsection{Asymptotic multiplier ideals}
Recall that the \emph{asymptotic multiplier ideals} of 
a graded sequence $\fra_{\bullet}$ are defined by
$\frb_t:=\cJ(\fra_{\bullet}^t):=\cJ(\fra_m^{t/m})$, where $m$ is divisible enough 
(depending on $t>0$). We have $\fra_m\subseteq\frb_m$ for all $m$
and it follows from the Subadditivity Theorem
that $(\frb_t)_{t>0}$ is a subadditive system of ideals. 
As above, we set $\frb_0:=\cO_X$. 
For more about graded sequences
and their corresponding asymptotic multiplier ideals we refer to~\cite{positivity}. 

Next we show that the asymptotic invariants $\lct^{\frq}(\fra_{\bullet})$
can be described in terms of the jumps of the system $\frb_{\bullet}$:

\begin{proposition}\label{prop2_3}
  If $\fra_{\bullet}$ is a graded sequence of ideals, and $\frb_{\bullet}$ is the system 
  of asymptotic multiplier ideals of $\fra_{\bullet}$, then 
  \begin{equation*}\lct^{\frq}(\fra_{\bullet})=\min\{\lambda\geq 0\mid\frq\not\subseteq\frb_{\lambda}\}\end{equation*}
  for every nonzero ideal $\frq$. 
\end{proposition}
\begin{proof}
  By definition, $\lct^{\frq}(\fra_{\bullet})=\sup_{m\geq 1}m\cdot\lct^{\frq}(\fra_m)$.
  Hence $t\geq\lct^{\frq}(\fra_{\bullet})$ if and only if $t/m\geq\lct^{\frq}(\fra_m)$, 
  or, equivalently, $\frq\not\subseteq\cJ(\fra_m^{t/m})$ for all $m$. 
  On the other hand, we have $\cJ(\fra_m^{t/m})\subseteq\frb_t$, 
  with equality if $m$ is divisible enough. The result follows.
\end{proof}
We now compare the invariants defined for $\fra_{\bullet}$ and for the corresponding 
system of asymptotic multiplier ideals $\frb_{\bullet}$. In the process we will see that
$\frb_{\bullet}$ has controlled growth.

\begin{proposition}\label{prop2_1}
  If $\fra_{\bullet}$ is a graded sequence of ideals, and $\frb_{\bullet}$ is the corresponding
  subadditive system given by the asymptotic multiplier ideals of $\fra_{\bullet}$, then:
  \begin{enumerate}
  \item[(i)] the system $\frb_{\bullet}$ has controlled growth;
  \item[(ii)] we have $\ord_E(\fra_{\bullet})=\ord_E(\frb_{\bullet})$ 
    for every divisor $E$ over $X$. 
  \end{enumerate}
\end{proposition}
We shall later extend~(ii) and show that $v(\fra_\bullet)=v(\frb_\bullet)$
for many non-divisorial valuations $v\in\Val_X$. See
Proposition~\ref{prop2_arbitrary_val}.
\begin{proof}
Given $t>0$, consider $m$ such that $\frb_t=\cJ(\fra_m^{t/m})$. By the definition of multiplier ideals,
we have
$\ord_E(\cJ(\fra_m^{t/m}))>t\cdot\frac{\ord_E(\fra_m)}{m}-A(\ord_E)$,
hence 
\begin{equation}\label{eq_prop2_1}
\frac{\ord_E(\frb_t)}{t}>\frac{\ord_E(\fra_m)}{m}-\frac{A(\ord_E)}{t}\geq
\ord_E(\fra_{\bullet})-\frac{A(\ord_E)}{t}.
\end{equation}

By letting $t$ go to infinity in~\eqref{eq_prop2_1} we get
$\ord_E(\frb_{\bullet})\geq\ord_E(\fra_{\bullet})$.
On the other hand, since $\fra_m\subseteq\frb_m$ for every $m$, we deduce
$\ord_E(\frb_m)\leq\ord_E(\fra_m)$. Dividing by $m$ and letting $m$ go to infinity
gives $\ord_E(\frb_{\bullet})\leq\ord_E(\fra_{\bullet})$. Therefore we have~(ii), and now
the assertion in~(i) follows from~\eqref{eq_prop2_1}.
\end{proof}

\begin{proposition}\label{prop2_2}
If $\fra_{\bullet}$ is a  graded sequence of ideals, and $\frq$ is a nonzero ideal,
then $\Arn^{\frq}(\fra_{\bullet})=\Arn^{\frq}(\frb_{\bullet})$, where $\frb_{\bullet}$ is the
subadditive system given by the asymptotic multiplier ideals of $\fra_{\bullet}$.
\end{proposition}

The case $\frq=\cO_X$ is Theorem~3.6 in~\cite{Mustata}. We include the proof
of the general case for the convenience of the reader. 
The key ingredient is the lemma below, which corresponds to
Lemma~3.7 in~\cite{Mustata}.

\begin{lemma}\label{lem2_3}
  If $\fra$ and $\frq$ are nonzero ideals on $X$, then
  $\Arn^{\frq}(\cJ(\fra^{\lambda}))\geq
  \lambda\cdot\Arn^{\frq}(\fra)-1$ for every $\lambda\in\RR_{\geq 0}$.
\end{lemma}

\begin{proof}
  Write $J=\cJ(\fra^{\lambda})$.
  It follows from the definition of the multiplier ideal that
  $\ord_E(J)>\lambda\cdot \ord_E(\fra)-A(\ord_E)$
  for any divisor $E$ above $X$.
  Since $\ord_E(\frq)\ge0$, this implies
  \begin{equation*}
    \Arn^{\frq}(J)\geq \frac{\ord_E(J)}{A(\ord_E)+\ord_E(\frq)}
    \ge\lambda\cdot \frac{\ord_E(\fra)}{A(\ord_E)+\ord_E(\frq)}
    -1.
  \end{equation*}
  We obtain the desired inequality by picking $E$
  that computes $\Arn^\frq(\fra)$.
\end{proof}

\begin{proof}[Proof of Proposition~\ref{prop2_2}]
  Given $t>0$, let us choose $m$ such that $\frb_t=\cJ(\fra_m^{t/m})$. 
We deduce from Lemma~\ref{lem2_3} that
\begin{equation*}\Arn^{\frq}(\frb_t)\geq t\cdot\frac{\Arn^{\frq}(\fra_m)}{m}-1\geq t\cdot\Arn^{\frq}(\fra_{\bullet})-1.\end{equation*}
Dividing by $t$ and letting $t$ go to infinity gives $\Arn^{\frq}(\frb_{\bullet})\geq
\Arn^{\frq}(\fra_{\bullet})$. The opposite inequality follows from the definition of asymptotic Arnold multiplicities and the inclusions $\fra_m\subseteq\frb_m$ for all $m$. 
\end{proof}

\begin{corollary}\label{cor_Arnold_graded}
If $\fra_{\bullet}$ is a graded sequence of ideals and $\frq$ is a nonzero ideal, then 
\begin{equation*}\Arn^{\frq}(\fra_{\bullet})=\sup_E\frac{\ord_E(\fra_{\bullet})}{A(\ord_E)+\ord_E(\frq)},\end{equation*}
where the supremum is over all divisors $E$ over $X$.
\end{corollary}

\begin{proof}
The assertion follows by combining Propositions~\ref{prop1},~\ref {prop2_1} and~\ref{prop2_2}.
\end{proof}

\section{Quasi-monomial valuations}\label{S304}
We now want to extend the considerations in~\S\ref{S301}--\S\ref{S302}
from divisorial to general real valuations.
As an important intermediate step, we first 
study quasi-monomial valuations.

\subsection{Quasi-monomial valuations}
Let $X$ be a scheme as before. A \emph{quasi-monomial} valuation
in $\Val_X$ is a valuation that is monomial in some local coordinates
on some birational model of $X$. For further reference, we shall 
describe this concept in detail.

Suppose that $\pi\colon Y\to X$
is a proper birational morphism, with $Y$ regular and connected, and 
$\uy=(y_1,\dots,y_r)$ is a system of 
algebraic coordinates at a point $\eta\in Y$. 
We use the notation $y^{\beta}=\prod_{i=1}^ry_i^{\beta_i}$ 
when $\beta=(\beta_1,\dots,\beta_r)\in\ZZ_{\geq 0}^r$, 
and $\langle \alpha,\beta\rangle:=\sum_{i=1}^r\alpha_i\beta_i$
when $\alpha,\beta\in\RR^r$. 
We also write $\alpha\le\beta$ as a shorthand for $\alpha_i\le\beta_i$,
$1\le i\le r$.
\begin{proposition}\label{lem1_val}
  To every $\alpha=(\alpha_1,\dots,\alpha_r)\in\RR_{\geq 0}^r$ 
  one can associate a unique valuation 
  $\val_{\alpha}=\val_{\uy,\alpha}\in\Val_X$ 
  with the following property: 
  whenever $f\in\cO_{Y,\eta}$ is written in 
  $\widehat{\cO_{Y,\eta}}$
  as $f=\sum_{\beta\in\ZZ_{\geq 0}^r}c_{\beta}y^{\beta}$, 
  with each $c_{\beta}\in\widehat{\cO_{Y,\eta}}$ either zero or
  a unit, we have
  \begin{equation}\label{eq_lem1_val}
    \val_{\alpha}(f)=\min\{\langle\alpha,\beta\rangle\mid c_{\beta}\neq 0\}.
  \end{equation}
  The function $\RR_{\ge0}^r\ni\alpha\to\val_\alpha(f)$ is continuous
  for each $f\in K(X)$.
  Furthermore, suppose, after re-indexing, that 
  $\alpha_i>0$ for $1\le i\le r'$ and $\alpha_i=0$ for 
  $r'<i\le r$ and let $\eta'$ be the generic point 
  of $\bigcap_{1\le i\le r'}V(y_i)$. Then we have:
  \begin{itemize}
  \item[(i)]
    the center of $\val_\alpha$ on $Y$ is $\eta'$; in particular, 
    $\val_\alpha$ is the trivial valuation if and only if 
    $\alpha_i=0$ for $1\le i\le r$;
  \item[(ii)]
    if $\uy'=(y'_1,\dots,y'_{r'})$ is a system of algebraic coordinates
    at $\eta'$ such that $V(y'_i)=V(y_i)$ at $\eta'$,
    then $\val_\alpha=\val_{(y'_1,\dots,y'_{r'}),(\alpha_1,\dots,\alpha_{r'})}$;
  \item[(iii)]
    if $v\in\Val_X$ is a  valuation whose center $\zeta$ on $Y$ 
    is contained in the closure of $\eta'$
    and such that $v(y_i)=\alpha_i$ for $1\le i\le r'$, 
    then $v(f)\ge\val_\alpha(f)$ for all $f\in\cO_{Y,\zeta}$.
  \end{itemize}
\end{proposition}
A valuation as above is called a quasi-monomial valuation.
Note that the trivial valuation on $X$ is considered quasi-monomial.
\begin{proof}
  Let us say that an expansion of 
  $f\in\widehat{\cO_{Y,\eta}}$ as $f=\sum_\beta c_\beta y^\beta$,
  with $c_\beta\in\widehat{\cO_{Y,\eta}}$, is \emph{admissible}
  if, for each $\beta$, $c_\beta$ is either zero or a unit.
  Any $f\in\widehat{\cO_{Y,\eta}}$ admits
  an admissible expansion. 
  Indeed, by Cohen's structure theorem there exists a
  (non-canonical) isomorphism 
  $\iota:\widehat{\cO_{Y,\eta}}\simto k(\eta)\llbracket y_1,\ldots,y_r\rrbracket$,
  where $k(\eta)$ is the residue field of $\cO_{Y,\eta}$.
  Fix such an isomorphism for the duration of the proof.
  If $\iota(f)=\sum_\beta a_\beta y^\beta$ with $a_\beta\in k(\eta)$, 
  then $f=\sum_\beta c_\beta y^\beta$ with 
  $c_\beta=\iota^{-1}(a_\beta)$ is an admissible expansion.

  In general, $f$ may admit many admissible expansions 
  but we claim that the quantity 
  $\min\{\langle\alpha,\beta\rangle\mid c_\beta\ne0\}$
  is the same for all of them.
  To see this, again write $\iota(f)=\sum_\beta a_\beta y^\beta$
  with $a_\beta\in k(\eta)$.
  It suffices to prove that 
  \begin{equation}\label{e301}
    \min\{\langle\alpha,\beta\rangle\mid c_\beta\ne0\}
    =\min\{\langle\alpha,\beta\rangle\mid a_\beta\ne0\}.
  \end{equation}
  Write 
  $\iota(c_\beta)=\sum_\gamma c_{\beta\gamma}y^\gamma$,
  so that $a_\beta=\sum_{\gamma\le\beta}c_{\gamma,\beta-\gamma}$.
  On the one hand, if $a_\beta\ne0$ then there exists 
  $\gamma\le\beta$ such that $c_\gamma\ne0$.
  On the other hand, if $a_\beta=0$ and $c_\gamma=0$
  for all $\gamma\le\beta$, $\gamma\ne\beta$, then 
  $c_{\beta0}=0$ and hence $c_\beta=0$. 
  These two remarks imply~\eqref{e301}.
  
  Now, it is a standard fact that 
  $\sum_\beta a_\beta y^\beta\to
  \min\{\langle\alpha,\beta\rangle\mid a_\beta\ne0\}$
  defines a (monomial) valuation on the formal power series ring
  $k(\eta)\llbracket y_1,\ldots,y_r\rrbracket$.
  We thus see that $\val_\alpha$ is a well defined valuation 
  on $\widehat{\cO_{Y,\eta}}$. The uniqueness statement of
  the lemma is clear since every $f\in\cO_{Y,\zeta}$
  admits an admissible expansion.

  It follows from~\eqref{eq_lem1_val} that
  $\alpha\mapsto\val_\alpha(f)$ is continuous for 
  $f\in\cO_{Y,\eta}$
  and hence also for all $f\in K(X)$.
  It remains to prove~(i)--(iii).

  The valuation $\val_\alpha$ is nonnegative on 
  the local ring $\cO_{Y,\eta'}$ and positive on the maximal ideal.
  By definition, the center of $\val_\alpha$ on $Y$
  is then equal to $\eta'$, proving~(i).
  
  To prove~(ii), we first consider the case when 
  $y'_i=y_i$ for $1\le i\le r'$.
  Write $v:=\val_{(y_1,\dots,y_{r'}),(\alpha_1,\dots,\alpha_{r'})}$.
  Then the valuations $\val_\alpha$ and $v$ in $\Val_X$
  both have center $\eta'$ on $Y$. We must prove that 
  $\val_\alpha(f)=v(f)$ for all $f\in\cO_{Y,\eta'}$.
  By continuity of $\alpha\to\val_\alpha(f)$, we may assume that
  the numbers $\alpha_1,\dots,\alpha_{r'}$ are rationally independent.
  By $\frm_{Y,\eta'}$-adic continuity, $\val_\alpha$ and $v$ 
  extend uniquely as (semi)valuations on $\widehat{\cO_{Y,\eta'}}$ and 
  give value zero to any element not in the maximal ideal of this ring. 
  Now consider an expansion   
  \begin{equation*}
    f=\sum_{\beta\in\ZZ_{\ge0}^{r'}} c_\beta y^\beta
  \end{equation*}
  with $c_\beta\in\widehat{\cO_{Y,\eta'}}$ either zero or a unit.  
  When $c_\beta\ne0$, we have
  $\val_\alpha(c_\beta y^\beta)=v(c_\beta y^\beta)=\langle\alpha,\beta\rangle$.
  Moreover, as $\beta$ varies, these values are all distinct and they tend to 
  infinity as $\sum_{i=1}^{r'}\beta_i\to\infty$. It then follows that 
  $\val_\alpha(f)=v(f)=\min_{c_\beta\ne0}\langle\alpha,\beta\rangle$.

  In order to prove~(iii) we may by the preceding step assume that
  $r'=r$, that is, $\alpha_i>0$ for $1\le i\le r$. 
  Let us further reduce~(iii) to the case $\zeta=\eta$. 
  If $\zeta\ne\eta$, then there exist $s>r$ and $y_{r+1},\dots,y_s\in\cO_{Y,\zeta}$ 
  such that $(y_1,\dots,y_s)$ is a system of algebraic coordinates at $\zeta$.
  Set $\alpha_i=v(y_i)>0$ for $r<i\le s$. By what precedes, 
  $\val_\alpha=\val_{(y_1,\dots,y_s),(\alpha_1,\dots,\alpha_r,0\dots,0)}$, so
  it follows from~\eqref{eq_lem1_val} that 
  $\val_{(y_1,\dots,y_s),(\alpha_1,\dots,\alpha_s)}\ge\val_\alpha$
  on $\cO_{Y,\zeta}$. 
  Hence it suffices to show that 
  $v\ge\val_{(y_1,\dots,y_s),(\alpha_1,\dots,\alpha_s)}$ on $\cO_{Y,\zeta}$.
  In other words, we may assume that $\zeta=\eta$.
  Now, $v\in\Val_X$ having center $\zeta=\eta$ on $Y$ implies that 
  $v$ extends uniquely as a semivaluation
  $v:\widehat{\cO_{Y,\eta}}\to\RR_{\ge0}\cup\{\infty\}$
  and if $f=\sum c_\beta y^\beta$ is an admissible expansion, 
  then $v(c_\beta y^\beta)=\langle\alpha,\beta\rangle$ for 
  all $\beta$ such that $c_\beta\ne0$.
  This easily implies $v(f)\ge\val_\alpha(f)$,
  proving~(iii).

  Let us finally prove~(ii) in general. By the special case proved above, 
  we may again assume that $r'=r$ and $\eta'=\eta$.
  Write $\val'_\alpha=\val_{y'_1,\dots,y'_r,\alpha_1,\dots,\alpha_r}$.
  We have $y'_i=u_iy_i$ with $u_i$ a unit in $\cO_{Y,\eta}$.
  Hence $\val'_\alpha(y_i)=\val'_\alpha(y'_i)=\alpha_i$.
  By~(iii), this implies 
  $\val'_\alpha\ge\val_\alpha$ on $\widehat{\cO_{Y,\eta}}$.
  By symmetry we then get $\val'_\alpha=\val_\alpha$,
  which completes the proof.
\end{proof}

In practice, instead of considering systems of coordinates, it is more convenient to
consider simple normal crossing divisors. 
Borrowing terminology from the Minimal Model Program, 
we introduce 
\begin{definition}
  A \emph{log-smooth pair over X} is a pair $(Y,D)$ 
  with $Y$ regular and $D$ a reduced effective 
  simple normal crossing divisor,
  together with a proper birational morphism 
  $\pi\colon Y\to X$ which is an isomorphism outside the
  support of $D$.
\end{definition}
The set of (isomorphism classes of) 
log-smooth pairs over $X$
admits a partial ordering: we say that 
$(Y',D')\succeq(Y,D)$ if there exists a birational morphism
$\phi\colon Y'\to Y$ over $X$ with 
$\Supp(D')\supseteq\Supp(\phi^*(D))$. Under this
ordering, any two log-smooth pairs can be dominated by a third,
and any log-smooth pair dominates $(X,\emptyset)$.

\begin{remark}\label{Nagata}
Suppose that we have a birational morphism $W\to X$, with $W$ regular, and $D_W$
a reduced, simple normal crossing divisor on $W$.
By Nagata's compactification theorem (see~\cite{Conrad}), there is a proper birational
morphism $\pi\colon Y\to X$ such that $W$ is isomorphic over $X$ to an open subset of $Y$. 
By~\cite{Temkin2}, we may resolve the singularities of $Y$ by a resolution that is an isomorphism over $W$, and therefore assume that $Y$ is regular. Given any such $Y$, we can find a reduced divisor $D$ on $Y$
whose restriction to $W$ is $D_W$, and whose support  contains the exceptional locus 
$\Exc(\pi)$ of $\pi$. Since $D_W$ has simple normal crossings, it follows 
from~\cite{Temkin2} that there is a 
  proper birational morphism $\phi\colon \widetilde{Y}\to Y$ 
  that is an isomorphism over $W\cup (Y\smallsetminus\Supp(D))$ such that 
  $\widetilde{Y}$ is regular and $\widetilde{D}:=\phi^*(D)_{\red}$ has simple normal crossings\footnote{Actually, the statement in~\cite{Temkin2} only asserts that
  $\phi^*(D)$ has normal crossings. However, resolving a normal crossing divisor is standard, so one can obtain the statement that we need.}. In this case $(\widetilde{Y},\widetilde{D})$ is a 
  log-smooth
pair over $X$, extending $(W,D_W)$.
\end{remark}

We denote by $\QM_{\eta}(Y,D)$ the set of all quasi-monomial 
valuations $v$ that can be described at the point $\eta\in Y$ 
with respect to coordinates $y_1,\dots,y_r$ such that each
$y_i$ defines at $\eta$ an irreducible component of 
$D$ (hence $\eta$ is the generic point of a connected 
component of the intersection of some of the $D_i$). 
We put $\QM(Y,D)=\bigcup_{\eta}\QM_{\eta}(Y,D)$.

\begin{remark}
  Every quasi-monomial valuation belongs to some
  $\QM(Y,D)$. Indeed, suppose the valuation
  $v$ is defined in coordinates $y_1,\dots,y_r$ 
  at $\eta$ and let $D_i$ be the closure
  of the divisor defined by $(y_i)$. Since $D=\sum_iD_i$ 
  has simple normal crossings in a neighborhood $W$ of $\eta$, 
  it follows from Remark~\ref{Nagata} 
  that there is a log-smooth pair $(\widetilde{Y},\widetilde{D})$
  over $X$ extending $(W,D\vert_W)$. 
  Then $v\in \QM(\widetilde{Y},\widetilde{D})$.
\end{remark}

\begin{definition}
  A log-smooth  pair $(Y,D)$ is \emph{adapted} to
  a quasi-monomial valuation $v$ if $v\in\QM(Y,D)$.
  It is a \emph{good} pair adapted to $v$ if
  the values $v(D_i)$ that are strictly positive
  are also rationally independent.
\end{definition} 
The following technical lemma ensures the existence
of good pairs.
\begin{lemma}\label{lem2_val} 
  Let $v\in\Val_X$ be quasi-monomial
  and consider a pair $(Y,D)$ adapted to $v$.
  Let $D_1,\dots,D_r$ be the 
  irreducible components
  of $D$ containing $\eta=c_Y(v)$.
  \begin{itemize}
  \item[(i)]
    If $(Y,D)$ is a good pair adapted to $v$ 
    and $(Y',D')\succeq(Y,D)$, then $(Y',D')$ is 
    also a good pair adapted to $v$. 
    Further, $\eta'=c_{Y'}(v)$ is the generic point of a connected component of
    the intersection of exactly $r$ 
    irreducible components $D'_j$, $1\le j\le r$ of $D'$, 
    and if $\phi\colon Y'\to Y$ is the corresponding morphism, then we can write
    \begin{equation}\label{eq_prop_val0}
      \phi^*(D_i)=\sum_{j=1}^rb_{ij}D'_j+E'_i,\quad i=1,\dots,r.
    \end{equation}
    Here $E'_i$ is an effective divisor on $Y'$ 
    whose support  does not contain $\eta'$
    and the $r\times r$ matrix $(b_{ij})$ 
    has nonnegative integer entries and nonzero determinant.
  \item[(ii)]
    There always exist a good pair $(Y',D')\succeq(Y,D)$ 
    adapted to $v$ and 
    irreducible components $D'_1,\dots,D'_r$ of $D'$
    such that the representation~\eqref{eq_prop_val0}  holds.
    More precisely, $v\in\QM_{\eta'}(Y',D')$, where $\eta'$, 
    lying over $\eta$, is the generic point
    of a connected component of $D'_1\cap\dots\cap D'_r$, 
    each $E'_i$ is an effective divisor 
    whose support  does not contain $\eta'$, 
    and the $r\times r$ matrix $(b_{ij})$ 
    has nonnegative integer entries and nonzero determinant.
    Further, there exists $s\le r$ such that 
    $c_{Y'}(v)$ is the generic point of a connected component
    of $D'_1\cap\dots\cap D'_s$. 
  \end{itemize}
\end{lemma}
The construction of the morphism $\phi\colon Y'\to Y$ in~(ii)
is toric in nature. 
The number $s$ is the rational rank of $v$; see~\S\ref{S201}.
\begin{proof}
  In~(i), let $D_i$, $1\le i\le M$ and $D'_j$, $1\le j\le N$
  be all the irreducible components of $D$ and $D'$, respectively.
  We have $\varphi^*D_i=\sum_jb_{ij}D'_j$ for nonnegative
  integers $b_{ij}$. After re-indexing, we may suppose that 
  $v(D_i)>0$ and $v(D'_j)>0$ if and only if $i\le r$ and $j\le s$, 
  respectively. Note that $c_Y(v)$ is the generic point of 
  a component of $\bigcap_{i\le r}D_i$
  and $c_{Y'}(v)\in\bigcap_{j\le s}D'_j$.
  Since $\varphi(c_{Y'}(v))=c_Y(v)$, 
  we have $\dim(\cO_{Y,c_Y(v)})\geq\dim(\cO_{Y',c_{Y'}(v)})$ by the Dimension Formula 
  (see~\cite[Theorem~15.6]{Matsumura}),
  hence $s\le r$.
  But, by assumption, the values $v(D_i)=\sum_{j=1}^sb_{ij}v(D'_j)$, $i\le r$
  are rationally independent. This implies that $s=r$,
  that $c_{Y'}(v)$ is the generic point of a component of $\bigcap_{j\le r}D'_j$, 
  that the matrix $(b_{ij})_{i,j=1}^r$ has maximal rank $r$, 
  and that the values $v(D'_j)$, $1\le j\le r$ are 
  rationally independent. This completes the proof of~(i).
  
  We now turn to~(ii).
  Given a system of coordinates $\underline{y}=(y_1,\dots,y_r)$ at $\eta=c_Y(v)$
  such that $D_i=V(y_i)$, 
  we get a morphism $h\colon\Spec(\cO_{Y,\eta})\to\Spec(\cO_{\AAA_{\QQ}^r,0})$. 
  Note that $h$ is formally smooth, and since $\cO_{\AAA_{\QQ}^r,0}$ is excellent, 
  it follows by the main theorem in~\cite{Andre} that 
  $h$ is a regular morphism. 
   We call a proper birational morphism $\phi\colon Y'\to Y$ 
   \emph{toroidal} (with respect to $\uy$)\footnote{This is an ad-hoc definition, although 
   related to the usual notion of \emph{toroidal morphism}, see~\cite{KKMS}.}
   if there is a proper birational morphism of toric varieties 
   $\psi\colon Z=Z(\Delta)\to\AAA_{\QQ}^r$, with
   $Z$ regular, such that
   $\phi$ and $\psi$ induce isomorphic schemes over 
   $\Spec(\cO_{Y,\eta})$ via base-change. 
   The morphism $\psi$ is defined by a fan $\Delta$ 
   refining the standard cone defining $\AAA_{\QQ}^r$,
   and the fact that $Z$ is regular is equivalent 
   with $\Delta$ being regular, which means that each cone of $\Delta$ is generated
   by part of a basis for $\ZZ^r$
   (we refer to~\cite{Fulton} for basic facts on toric 
   varieties and toric morphisms\footnote{While in \cite{Fulton} one works with toric varieties over $\CC$, all basic constructions extend to arbitrary ground fields.}). Note that since $h$ is a regular morphism
   and $Z$ is a regular scheme, 
   $Y'$ is regular in a neighborhood of $\phi^{-1}(\eta)$.
   On $Y'$ we have finitely many distinguished 
   points lying over $\eta$ (corresponding to the torus-fixed 
   closed points on $Z$). At each of these points we have
   a system of toroidal coordinates $\uy'=(y'_1,\dots,y'_r)$ 
   induced by the toric coordinates at the corresponding point on $Z$ (we use again the fact that
   $h$ is regular). 
   These are uniquely determined up to reordering. 
   One can write $y_i=\prod_j(y'_j)^{b_{i,j}}$, with 
   $b_{i,j}\in\ZZ_{\geq 0}$, and $\det(b_{i,j})=\pm 1$. 
 
   Starting with a toric proper birational morphism $Z\to\AAA_{\QQ}^r$,
   with $Z$ regular, there exists a log-smooth pair $(Y',D')$ dominating 
   $(Y,D)$,  such that we have $Y'\times_Y\Spec\,\cO_{Y,\eta}\simeq Z\times_{\AAA_{\QQ}^r}
   \Spec\,\cO_{Y,\eta}$, and such that the toroidal coordinates on $Y'$ define irreducible 
   components of $D'$.
   This is a consequence of Remark~\ref{Nagata}.

   Given a toroidal morphism $\phi\colon Y'\to Y$ 
   corresponding to $Z=Z(\Delta)\to\AAA_{\QQ}^r$, we have an affine open cover 
   of $Y'_{\eta}=\Spec\,\cO_{Y,\eta}\times_Y Y'$ by subsets $U_i$, 
   induced by the toric affine open subsets on $Z$. 
   If $\eta'=c_{Y'}(v)$, then $\phi(\eta')=\eta$, hence there is $i$ such that $\eta'\in U_i$.
   We have toroidal coordinates $\uy'=(y'_1,\dots,y'_r)$ on $U_i$ such that
   $y_i=\prod_j(y'_j)^{b_{i,j}}$, with $b_{i,j}\in\ZZ_{\geq 0}$, and $\det(b_{i,j})=\pm 1$. 
   Since $\eta'\in U_i$, it follows that $\alpha'_i:=v(y'_i)\geq 0$, and we have
   $\alpha_i=\sum_jb_{i,j}\alpha'_i$. 
   Since the matrix $(b_{i,j})_{i,j=1}^r$ induces a bijection between
   the monomials in $\uy'$ and the monomials in 
   $\uy$, it is clear that in terms of the coordinates on $Y'$ we have
   $v=\val_{\uy',\alpha'}$. 
   In particular, if $(Y',D')$ is a log-smooth pair such that the closure of each                
   $V(y'_i)$ is a component of $D'$, then $(Y',D')$  is adapted to $v$.
   
   To complete the proof of~(ii) it therefore suffices to prove the following
   statement. Let $\alpha=(\alpha_1,\dots,\alpha_r)\in\RR_{>0}^r$ be
   any vector and set $s:=\dim_\QQ\sum_i\QQ\alpha_i$. Then there exists 
   a regular fan $\Delta$ in $\ZZ^r$ refining the standard fan 
   $\Delta_0$ defining $\AAA_\QQ^r$ such that 
   $\alpha$ belongs to the relative interior
   of a cone of dimension $s$.
   To construct $\Delta$, first pick a vector space $W_\QQ\subseteq\QQ^r$ of
   dimension $s$ such that $\alpha\in W:=W_\QQ\otimes_\QQ\RR$.
   Let $\sigma_1$ be any rational simplicial $s$-dimensional cone 
   $\sigma_1\subseteq\RR_{\ge0}^r\cap W$ containing $\alpha$ in its interior.
   Let $\Delta_1$ be any simplicial fan refining $\Delta_0$ and having 
   $\sigma_1$ as one of its cones. Now refine $\Delta_1$ to a regular
   fan $\Delta$ using barycentric subdivision as in~\cite[\S2.6]{Fulton}.
   Then $\Delta$ will contain a cone $\sigma\subseteq\sigma_1$ 
   containing $\alpha$ in its interior.

   Alternatively, the toric birational morphism $Z(\Delta)\to\AAA_\QQ^r$
   can be constructed explicitly using Perron transformations 
   as in~\cite[Theorem 1]{Zariski}.
  \end{proof}
  
It follows from Lemma~\ref{lem2_val} that given 
finitely many quasi-monomial valuations $v_1,\dots,v_m$ in $\Val_X$,
there exists a pair $(Y,D)$ which is good and adapted to all the $v_i$.
Furthermore, given finitely many ideals $\fra_1,\dots,\fra_p$ on $X$, we may assume that
$(Y,D)$ gives a log resolution of the product 
$\fra=\fra_1\cdot\ldots\cdot\fra_p$:
this means that
$Y\to X$ is a log resolution of $\fra$ with the inverse image
of $V(\fra)$ being contained in the support on $D$.

\subsection{Abhyankar valuations}\label{S201}
Next we recall how to recognize a quasi-monomial 
valuation algebraically, in terms of its numerical
invariants. This will be very useful in the sequel.
The \emph{rational rank} $\ratrk(v)$ of
a valuation $v\in\Val_X$ is equal to $\dim_{\QQ}(\Gamma_v\otimes_{\ZZ}\QQ)$,
where $\Gamma_v:=v(K(X)^*)$ is the value group of $v$.
If $k_v$ and $k(\xi)$ are the residue fields
of the valuation ring $\cO_v$ and of $\cO_{X,\xi}$, respectively, where $\xi=c_X(v)$, 
then the \emph{transcendence degree} of $v$ is
defined as $\trdeg_X(v)=\trdeg(k_v/k(\xi))$. 
Note that if $\pi\colon Y\to X$ is proper and birational,
with $Y$ regular, and $\eta=c_Y(v)$, then $\dim(\cO_{Y,\eta})=\dim(\cO_{X,\xi})-\trdeg(k(\eta)/k(\xi))$ (this follows from the Dimension Formula since $\pi$ is birational, see~\cite[Theorem~15.6]{Matsumura}). This formula can be used to deduce that 
$\trdeg_X(v)$
is the maximum of $\dim(\cO_{X,c_X(v)})-\dim(\cO_{Y,c_Y(v)})$, where the maximum is over all
morphisms $Y\to X$ as above. 

In this setting, the Abhyankar inequality holds (see~\cite{Vaquie}):
\begin{equation}\label{ineq_Abhyankar}
  \ratrk(v)+\trdeg_X(v)\leq\dim(\cO_{X,\xi}).
\end{equation}
A valuation for which equality is achieved is an \emph{Abhyankar valuation}.
Another application of the Dimension formula implies 
that if $\pi\colon Y\to X$ is proper and birational, with $Y$ regular, then $v$ is an Abhyankar valuation over $X$ if and only if it is an Abhyankar valuation over $Y$.

\begin{proposition}\label{Abhyankar}
  A valuation $v\in\Val_X$ is an Abhyankar valuation if and only if
  it is quasi-monomial. Moreover, in this case there exists
  a good log-smooth pair $(Y,D)$ adapted to $v$ such that 
  $\pi(D)\subseteq \overline{c_X(v)}$, where $\pi\colon Y\to X$ is the 
  associated birational morphism.
\end{proposition}
\begin{proof}
  Let $\xi=c_X(v)$.
  First suppose $v$ is quasi-monomial and pick a 
  good pair $(Y,D)$ adapted to $v$. Let $D_1,\dots,D_r$ be the irreducible
  components of $D$ containing the center $\eta=c_Y(v)$.
  Then $\dim\cO_{Y,\eta}=r$. 
  By assumption, the values $v(D_i)$, $1\le i\le r$, are rationally independent,
  so $\ratrk(v)=r$. On the other hand,
  $\trdeg_Y(v)=0$. Thus $v$ is an Abhyankar valuation.
  
  Conversely, it was shown in~\cite{ELS} that 
  every Abhyankar valuation $v$ is 
  quasi-monomial.\footnote{While in~\cite{ELS} one considers 
    an algebraic variety over a field, the proof therein
    also works in our more general framework.} 
  We sketch the main idea in the proof, slightly modified in order to guarantee
  $\pi(D)\subseteq\overline{\xi}$. 
  Note that we may  blow-up any 
  closed subset of
  $\overline{\xi}$: the resulting $W$ over $X$ might be singular, but 
  we may replace $W$ by
  $W'\to W$ that is an isomorphism over
  $X\smallsetminus \overline{\xi}$, with $W'$ regular.
  
  Let $J$ denote the ideal defining $\overline{\xi}$.
  One knows that if $v$ is an Abhyankar valuation of $K(X)$, 
  then the value group $\Gamma_v$
  is a finitely generated free abelian group. 
  We first find a 
  proper birational morphism $Y\to X$ that is an isomorphism over 
  $X\smallsetminus\overline{\xi}$, with $Y$ 
  regular, 
  such that $\dim\cO_{X,\xi}-\dim\cO_{Y,\eta}=\trdeg_X(v)$, where $\eta=c_Y(v)$,
  and there are $f_1,\dots,f_r\in\cO_{Y,\eta}$
  such that $v(f_1),\dots,v(f_r)$ give a basis of $\Gamma_v$. 
  Indeed, in order to obtain both conditions, 
  it is enough to perform finitely many times the following
  operation: given $g,h\in\cO_{X,\xi}$, we blow-up a 
  closed subset in $\overline{\xi}$
  to get $W\to X$ such that there is 
  $Q\in\cO_{W,c_W(v)}$ with $v(Q-\frac{g}{h})>0$ 
  or $v(Q-\frac{h}{g})>0$. For this it is enough to
  blow-up the subscheme defined by $(g,h)+J^N$, 
  where $N\cdot v(J)>\max\{v(g), v(h)\}$. 
    
  Suppose now that $Y$ is as above, and the $f_i$ are defined
  in a neighborhood $U$ of $\xi$, and consider any 
  regular $Y'$ with $\phi\colon Y'\to Y$
  proper and birational such that $\phi^{-1}(U)\to U$ 
  is a log resolution of $\prod_{i=1}^r((f_i)+J^N)$, 
  where $N\cdot v(J)>\max_i\{v(f_i)\}$.
  One can easily see that if $\eta'=c_{Y'}(v)$, 
  then we have coordinates
  $y'_1,\dots,y'_r$ at $\eta'$ such that 
  \begin{equation}
    \left((f_i)+J^N\right)\cdot\cO_{Y',\eta'}
    =\left(\prod_{j=1}^r(y'_j)^{b_{ij}}\right),
    \ \text{with $b_{ij}\in\ZZ_{\geq 0}$ and $\det(b_{ij})=\pm 1$},
  \end{equation}
  and $v(y'_1),\dots,v(y'_s)$ are linearly independent over $\QQ$. 
  It is then clear that
  $v$ is equal to the quasi-monomial valuation 
  attached to $(v(y'_1),\dots,v(y'_r))$ 
  in this system of coordinates. 
  One more application of Remark~\ref{Nagata} 
  gives the conclusion of the proposition. 
\end{proof}
\begin{remark}
  The trivial valuation is quasi-monomial with rational rank zero.
\end{remark}
\begin{remark}\label{R401}
  A valuation is \emph{divisorial},
  that is, a positive multiple of a valuation $\ord_E$,
  if and only if it is quasi-monomial with rational rank one. 
  In particular, a nontrivial  valuation 
  $v\in\QM(Y,D)$ is divisorial 
  if and only if there exists $t\in\RR_{>0}$ such that 
  $v(D_i)\in t\QQ$ for all~$i$.
  Thus the divisorial valuations are dense in $\QM(Y,D)$.
\end{remark}

\subsection{Completion and field extension}\label{S311}
Using the numerical invariants, we now show that the 
set of quasi-monomial valuations is preserved under two
important operations: localization followed by completion; 
and algebraic field extensions.
\begin{lemma}\label{lemma_valuations}
  Let $\xi$ be a point on $X$, and consider the canonical morphism
  $\phi\colon X'= \Spec\,R\to X$, where $R=\widehat{\cO_{X,\xi}}$. 
  If $v'\in \Val_{X'}$ has center the closed point, 
  and if $v\in\Val_X$ is induced 
  from $v'$ by restriction, then $\trdeg_{X'}(v')=\trdeg_X(v)$ and 
  $\ratrk(v')=\ratrk(v)$. In particular, $v$ is quasi-monomial 
  if and only if $v'$ is quasi-monomial. 
\end{lemma}
\begin{proof}
  If $\frm$ is the maximal ideal in $R$, then $\alpha:=v'(\frm)>0$. Given $f\in R$, let $g\in
  \cO_{X,\xi}$ be such that $(f-g)\in\frm^n$, where $n\alpha>v'(f)$. In this case
  $v'(f-g)>v'(f)$, hence 
  $v'(f)=v'(g)$. This shows that $v'$ and $v$ have the same value groups. In particular,
  $\ratrk(v')=\ratrk(v)$. 
  
  Denote by $(\cO_{v'},\frm_{v'})$ and $(\cO_{v},\frm_{v})$ 
  the valuation rings corresponding to $v'$ and $v$, respectively. 
  The equality $\trdeg_{X'}(v')=\trdeg_X(v)$
  is equivalent to the field extension 
  $\cO_{v}/\frm_{v}\hookrightarrow\cO_{v'}/\frm_{v'}$ being algebraic. 
  In fact, we will show that $\cO_{v}/\frm_{v}=
  \cO_{v'}/\frm_{v'}$. Given a nonzero $u\in\cO_{v'}$, write $u=\frac{f}{f_1}$, with
  $f$, $f_1\in R$. As above, let us consider $g$, $g_1\in\cO_{X,\xi}$ with
  $v'(f-g)>v'(f)$ and $v'(f_1-g_1)>v'(f_1)$. In particular, we have $v'(f)=v'(g)$
  and $v'(f_1)=v'(g_1)$. Since $\frac{f}{f_1}-\frac{g}{g_1}=\frac{fg_1-f_1g}{f_1g_1}$
  and $v'(fg_1-f_1g)=v'((f-g)g_1+g(g_1-f_1))>v'(f_1g_1)$, it follows that
  the class of $\frac{f}{f_1}$ in $\cO_{v'}/\frm_{v'}$ lies in $\cO_{v}/\frm_{v}$.
  This completes the proof.
\end{proof}
\begin{lemma}\label{lemma_valuations2}
  Let $k\subset K$ be an algebraic field extension, and $\phi\colon \AAA_K^n\to\AAA_k^n$
  the corresponding morphism of affine spaces. Suppose that $v'$ is a valuation of
  $K(x_1,\dots,x_n)$ with center on $\AAA_K^n$, and let $v$ be its restriction
  to $k(x_1,\dots,x_n)$. Then $\trdeg_{\AAA_k^n}(v)=\trdeg_{\AAA_K^n}(v')$ and 
  $\ratrk(v)=\ratrk(v')$. In particular, $v$ is quasi-monomial if and only if
  $v'$ is quasi-monomial.
\end{lemma}
\begin{proof}
  Let $(\cO_{v},\frm_{v})$ and $(\cO_{v'},\frm_{v'})$ be the valuation rings of $v$ and $v'$,
  respectively. Note that we have a local homomorphism $\cO_{v}\hookrightarrow\cO_{v'}$.
  Since the extension $k[x_1,\dots,x_n]\hookrightarrow K[x_1,\dots,x_n]$
  is integral, in order to show that $\trdeg_{\AAA_k^n}(v)=\trdeg_{\AAA_K^n}(v')$ it is enough to show that
  the field extension $\cO_{v}/\frm_{v}\hookrightarrow\cO_{v'}/\frm_{v'}$ is algebraic. 
  Given $f\in \cO_{v'}$, there is an equation 
  \begin{equation}\label{eq_lemma_valuation2}
    \sum_{i=0}^mc_if^i=0,
  \end{equation}
  with $c_i\in k(x_1,\dots,x_n)$ not all zero. If $v(c_j)=\min_iv(c_i)$, then 
  $c_i/c_j\in\cO_{v}$ for all $i$. Dividing by $c_j$ in~\eqref{eq_lemma_valuation2}, 
  we see that $\overline{f}\in  \cO_{v'}/\frm_{v'}$
  is algebraic over $\cO_{v}/\frm_{v}$.
  
  Since $k(x_1,\dots,x_n)\subseteq K(x_1,\dots,x_n)$, in order to show that 
  $\ratrk(v)=\ratrk(v')$ it is enough to show that for every $f\in K[x_1,\dots,x_n]$,
  some integer multiple of $v'(f)$ lies in the value group of $v$. 
  Consider an equation~\eqref{eq_lemma_valuation2} satisfied by $f$. 
  We can find $i\neq j$ such that $v'(c_if^i)=v'(c_jf^j)$. 
  Hence $(j-i)v'(f)=v(c_i)-v(c_j)$ lies in the value group of $v$.
\end{proof}

\section{Structure of valuation space}\label{S305}
Next we investigate the structure of the valuation space $\Val_X$. 
We show that it is a projective limit of simplicial cone complexes
and endowed with a natural integral affine structure.
This gives a way of approximating a valuation by 
quasi-monomial valuations.
Our discussion largely follows~\cite{BFJ1}, with some details added
and some modifications made due to the fact that our setting here is 
slightly different. See also~\cite{KoSo,Thu2,Payne,BFJ2}.

\subsection{Topology and ordering}
Recall from~\S\ref{S301} that we can view the elements of 
$\Val_X$ either as real valuations of the function field of $X$
or as $\RR_{\ge0}$-valued homomorphisms of the semiring
of ideals on $X$. 
This leads to two natural topologies $\tau$ and $\sigma$ 
on $\Val_X$. 
Namely, $\sigma$ is the weakest topology for which the evaluation 
map $\Val_X\ni v\to\phi_f(v):=v(f)$
is continuous for all nonzero rational functions $f$ on $X$.
Similarly, $\tau$ is the weakest topology
for which the evaluation map $\Val_X\ni v\to\phi_\fra(v):=v(\fra)$
is continuous for all nonzero ideals $\fra$ on $X$.
\begin{lemma}\label{L204}
  The two topologies $\sigma$ and $\tau$ defined above coincide.
\end{lemma}
\begin{proof}
  First suppose that $X$ is affine. 
  Since $v(f/g)=v(f)-v(g)$, we see that $\sigma$ 
  is the weakest topology that makes all maps 
  $\phi_f$, with $f\in\cO(X)$, continuous. 
  In particular, $\tau$ is finer than $\sigma$. 
  On the other hand, if an ideal $\fra$ is generated by 
  $f_1,\dots,f_r$, then $\phi_{\fra}=\min_i\phi_{f_i}$. 
  Therefore $\sigma$ is finer than $\tau$, which completes the 
  proof in the affine case. 
  
  Next, note that if $U$ is an open subset of $X$, then the 
  two topologies on $\Val_U\subseteq\Val_X$ 
  are just the subspace topologies with respect to $\sigma$ 
  and $\tau$ on $\Val_X$. 
  For $\sigma$ this is clear, while for $\tau$ this follows 
  from the fact that every coherent ideal sheaf on $U$ is the 
  restriction of a coherent ideal sheaf on $X$. 

  Now, if $U\subseteq X$ is open and affine, $\Val_U\subseteq\Val_X$ 
  is closed in $\Val_X$ in both the $\sigma$ and $\tau$ topologies.
  Indeed, if $J$ is the ideal defining $X\setminus U$,
  with the reduced scheme structure,
  then $\Val_U=\{v\in\Val_X\mid v(J)=0\}$,
  hence is $\tau$-closed. On the other hand, we also have 
  $\Val_U=\bigcap_{h\in\cO(U)}\{v\in\Val_X\mid v(h)\geq 0\}$, 
  hence $\Val_U$ is also $\sigma$-closed. 
  If we cover $X$ by finitely many affine open subsets $U_i$, 
  we now deduce the assertion in the lemma  for $X$
  from the assertion for the $U_i$.
\end{proof}

\begin{remark}\label{rem_L204}
  It follows from the above proof that the map 
  \begin{equation*}
    \Val_X\ni v\overset{c_X}\to c_X(v)\in X
  \end{equation*}
  is ``anticontinuous'' in the sense that the inverse image of any open subset is closed.
\end{remark}

\begin{definition}\label{D201}
  If $v,w\in\Val_X$, then we say that $v\le w$ if $v(\fra)\le w(\fra)$
  for all (nonzero) ideals $\fra$ on $X$.
\end{definition}
This clearly defines a partial ordering under which the trivial valuation is
the unique minimal element. Note that this order relation depends 
on the model $X$. 
\begin{lemma}\label{L205}
  We have $v\le w$ if and only if $\eta:=c_X(w)\in\overline{c_X(v)}$ and 
  $w(f)\ge v(f)$ for any $f\in\cO_{X,\eta}$.
\end{lemma}
\begin{proof}
  Let $\xi:=c_X(v)$.
  First suppose $v\le w$. If $J$ is the ideal defining $\overline{\xi}$
  with the reduced scheme structure, then 
  $w(J)\ge v(J)>0$, so $\eta\in\overline\xi$.
  Pick $f\in\cO_{X,\eta}\subseteq\cO_{X,\xi}$
  and let $\fra$ be an ideal on $X$ for which 
  $\fra\cdot\cO_{X,\eta}$ is principal, generated by $f$.
  Then $v(f)=v(\fra)\le w(\fra)=w(f)$.

  Conversely, suppose $\eta\in\overline\xi$ and 
  that $v(f)\le w(f)$ for $f\in\cO_{X,\eta}$.
  For any ideal $\fra$ we then have
$v(\fra)=\min_{f\in\fra\cdot\cO_{X,\xi}}v(f)\le\min_{f\in\fra\cdot\cO_{X,\eta}}v(f)\le\min_{f\in\fra\cdot\cO_{X,\eta}}w(f)=w(\fra)$.
\end{proof}

\subsection{Simplicial cone complexes and integral affine structure}
Next we investigate the structure of the subset $\QM(Y,D)\subseteq\Val_X$
for a given log-smooth pair $(Y,D)$ over $X$.
\begin{lemma}\label{lem_QM}
  If $(Y,D)$ is a log-smooth pair over $X$, 
  and if $\eta$ is the generic point of a connected component 
  of the intersection of $r$ irreducible components $D_1,\ldots,D_r$ of $D$, then
  the map $\QM_{\eta}(Y,D)\to\RR^r$ defined by
  $v\to (v(D_1),\ldots,v(D_r))$ gives a homeomorphism 
  onto the cone $\RR_{\geq 0}^r$.
\end{lemma}
\begin{proof}
  It is clear that this map gives a bijection of $\QM_{\eta}(Y,D)$ onto 
  $\RR^r_{\geq 0}$. The map is continuous since by definition of 
  the topology, $v\to v(D_i)$ is continuous for each $i$. 
  The continuity of the inverse map follows from
  Proposition~\ref{lem1_val}.
\end{proof}
Thus $\QM(Y,D)$ is the union of finitely many simplicial
cones $\QM_{\eta}(Y,D)$. Each of these cones is closed in 
$\QM(Y,D)$. Indeed, $\QM_{\eta}(Y,D)$ consists of those 
$v\in\QM(Y,D)$ such that $v(D_j)=0$ for $D_j\not\ni\eta$, 
and such that $c_X(v)$ does not lie on any of the 
connected components of $\bigcap_{D_j\ni\eta}D_j$ not containing $\eta$ 
(for the fact that these are closed conditions, see
Lemma~\ref{L204} and Remark~\ref{rem_L204}). 
This allows us to view $\QM(Y,D)$ 
as a \emph{simplicial cone complex}.

Following~\cite{KKMS} one can equip $\QM(Y,D)$ with an integral
affine structure. We shall not discuss this in detail here, 
but simply define an \emph{integral linear function} on $\QM(Y,D)$ 
to be a map $\QM(Y,D)\to\RR$ whose restriction to each $\QM_{\eta}(Y,D)$
is integral linear under the homeomorphism in Lemma~\ref{lem_QM}. 
We can similarly define \emph{integral linear maps} $\QM(Y',D')\to\QM(Y,D)$
(in this case we require that each $\QM_{\eta'}(Y',D')$ 
is mapped to some $\QM_{\eta}(Y,D)$).
Every such map is continuous. 

\subsection{Retraction}\label{S402}
Given a log-smooth pair $(Y,D)$ over $X$, we define a
\emph{retraction} map
\begin{equation*}
  r_{Y,D}\colon\Val_X\to\QM(Y,D). 
\end{equation*}
This maps a valuation $v$ to
the unique quasi-monomial valuation 
$w:=r_{Y,D}(v)\in\QM(Y,D)$ such that $w(D_i)=v(D_i)$ 
for every irreducible component $D_i$ of $D$.
Note that $c_Y(v)\in\overline{\{c_Y(w)\}}$. 
Clearly $r_{Y,D}$ is the identity on $\QM(Y,D)$ and 
it is not hard to see that $r_{Y,D}$ is continuous.
This justifies the terminology ``retraction''.
\begin{lemma}\label{L201}
  If $(Y',D')\succeq(Y,D)$ are log-smooth pairs, then 
  $r_{Y,D}\circ r_{Y',D'}=r_{Y,D}$.
  Furthermore, $r_{Y,D}\colon \QM(Y',D')\to\QM(Y,D)$ is integral linear.
\end{lemma}
\begin{proof}
  Let $D_1,\dots,D_M$ and $D'_1,\dots,D'_N$ be the 
  irreducible components of $D$ and $D'$, respectively.
  For the first assertion, it suffices to 
  show that $v$ and $v':=r_{Y',D'}(v)$ take
  the same values on $D_i$, $1\le i\le M$.
  By assumption we have a birational morphism 
  $\varphi\colon Y'\to Y$ over $X$ and 
  $\varphi^*(D_i)=\sum_{j=1}^Nb_{ij}D'_j$ for $1\le i\le M$,
  where $b_{ij}\ge 0$.
  Thus 
  \begin{equation*}
    v(D_i)
    =\sum_jb_{ij}v(D'_j)
    =\sum_jb_{ij}v'(D'_j)
    =v'(D_i).
  \end{equation*}
  For the second assertion, let $\eta'$ be the generic point 
  of a connected component
  of $s$ of the $D'_j$, say $D'_1,\ldots,D'_s$. 
  Suppose that $D_1,\ldots,D_r$ are the
  irreducible components of $D$ that contain $\phi(\eta')$, 
  and let $\eta$ be the generic point
  of the connected component of $D_1\cap\dots\cap D_r$ 
  that contains $\phi(\eta')$.
  In this case $r_{Y,D}$ induces a map $\QM_{\eta'}(Y',D')\to\QM_{\eta}(Y,D)$, 
  that under the identifications $\QM_{\eta'}(Y',D')\simeq\RR_{\geq 0}^s$ and 
  $\QM_{\eta}(Y,D)\simeq\RR_{\geq 0}^r$ 
  provided by Lemma~\ref{lem_QM} is given by the 
  matrix $(b_{ij})$, with $1\leq i\leq r$ and $1\leq j\leq s$.  
\end{proof}
As a consequence of Proposition~\ref{lem1_val}, the retraction map
is order-reversing: 
\begin{lemma}\label{L202}
  Let $(Y,D)$ be a log-smooth pair and $v\in\Val_X$.
  If $w:=r_{Y,D}(v)$ and $\eta=c_Y(v)$, 
  then $w(f)\le v(f)$ for any $f\in\cO_{Y,\eta}$.
  Equality holds if the support of $V(f)$ is locally contained in 
  the support of $D$ at $\eta$.
\end{lemma}
\begin{corollary}\label{C202}
  For every $v\in\Val_X$, we have
  $r_{Y,D}(v)\le v$ in the sense of Definition~\ref{D201}.
  More precisely, for any ideal $\fra$ on $X$ we have
  $r_{Y,D}(v)(\fra)\le v(\fra)$, with equality if 
  $(Y,D)$ gives a log resolution of $\fra$.
\end{corollary}

\subsection{Structure theorem}
We are now in position to exhibit $\Val_X$ as a projective
limit of simplicial cone complexes.
\begin{theorem}\label{T301}
  The retraction maps induce a homeomorphism
  \begin{equation*}
    r\colon\Val_X\to\varprojlim_{(Y,D)}\QM(Y,D).
  \end{equation*}
\end{theorem}
\begin{proof}
  The map $r$ is continuous since each $r_{Y,D}$ is.
  Let us construct its inverse. An element of the 
  projective limit is a compatible family of valuations
  $(v_{Y,D})$. To such a family we associate the function 
  $v$ that on an ideal $\fra$ on $X$ takes the value 
  $v(\fra):=\sup_{(Y,D)}v_{Y,D}(\fra)$. By Corollary~\ref{C202}
  the supremum is attained whenever $(Y,D)$
  defines a log resolution of $\fra$.
  It is easy to check that $v$ defines a valuation 
  in $\Val_X$ whose center on $X$ is the unique minimal element
  among the centers of all the $v_{Y,D}$.
  We see that $r$ is a continuous bijection. The 
  continuity of $r^{-1}$ follows from 
  Lemma~\ref{L204} and Corollary~\ref{C202}.
\end{proof}
\begin{corollary}\label{C401}
  The set of quasimonomial valuations is dense in $\Val_X$.
  Moreover, if $v\in\Val_X$, then given any 
  neighborhood  $U$ of $v$ in $\Val_X$ there exists
  a log-smooth pair $(Y,D)$ adapted to $v$ 
  such that $r_{Y,D}(v)\in U$ and
  such that $\pi(D)\subseteq \overline{c_X(v)}$,
  where $\pi\colon Y\to X$ is the induced morphism.
\end{corollary}
\begin{proof}
  The result is an immediate consequence of 
  Theorem~\ref{T301}, except for the requirement
  that $\pi(D)\subseteq c_X(v)$.
  To have this last property, it suffices to show that for
  any ideal $\fra$ on $X$ there exists a log-smooth pair 
  $(Y,D)$ above $X$ such that $\pi(D)\subseteq c_X(v)$
  and $r_{Y,D}(v)(\fra)=v(\fra)$. Let $\frm$ be the ideal
  defining $\overline{c_X(v)}$ with the reduced structure, 
  and pick $n>v(\fra)/v(\frm)$.
  Then $v(\fra+\frm^n)=v(\fra)$. 
  Now $V(\fra+\frm^n)\subseteq\overline{c_X(v)}$, so
  there exists a log resolution $(Y,D)$ of $\fra+\frm^n$
  such that $\pi(D)\subseteq\overline{c_X(v)}$.
  Then 
  \begin{equation*}
    v(\fra)=v(\fra+\frm^n)=r_{Y,D}(v)(\fra+\frm^n)\leq r_{Y,D}(v)(\fra),
  \end{equation*}
  hence $v(\fra)=r_{Y,D}(\fra)$ by Corollary~\ref{C202}.
\end{proof}
\begin{remark}
  The set of divisorial valuations is also dense in $\Val_X$.
  Indeed, divisorial valuations are dense in every $\QM(Y,D)$ 
  as a consequence of Remark~\ref{R401}.
\end{remark}

\section{Log discrepancy}\label{S306}
Our next goal is to define the log discrepancy  
of quasi-monomial, and more general valuations. 

\subsection{Log discrepancy of quasi-monomial valuations}
\begin{proposition}\label{log_discrep}
  One can associate to every quasi-monomial valuation $v\in\Val_X$ 
  a nonnegative real number $A_X(v)$, its \emph{log discrepancy}, such that
  \begin{itemize}
  \item[(i)] $A_X$ coincides with our old definition for divisorial valuations;
  \item[(ii)] for any log-smooth pair $(Y,D)$ over $X$, $A_X$ is integral 
    linear on $\QM(Y,D)$;
  \item[(iii)] for any proper birational morphism $X'\to X$, 
    with $X'$ regular,
    and any quasi-monomial valuation $v\in\Val_X$, we have
    $A_X(v)=A_{X'}(v)+v(K_{X'/X})$. 
  \end{itemize}
\end{proposition}
Conditions~(i) and (ii) together can be rephrased by saying that if
$v\in\QM(Y,D)$, then 
\begin{equation}\label{eq_log_discrep}
  A_X(v)=\sum_{i=1}^Nv(D_i)\cdot A_X(\ord_{D_i})=\sum_{i=1}^Nv(D_i)\cdot
  (1+\ord_{D_i}(K_{Y/X})),
\end{equation}
where $D_1,\ldots,D_N$ are the irreducible components of $D$.
Whenever $X$ is understood, we write $A(v)$ instead of $A_X(v)$.
\begin{proof}
  It is clear that the formula (\ref{eq_log_discrep}) 
  uniquely determines $A_X(v)$ for $v\in\QM(Y,D)$. 
  Let us temporarily denote the expression in (\ref{eq_log_discrep}) by
  $A_{X,Y,D}(v)$. We need to show that this is independent of the choice
  of pair $(Y,D)$. 
  
    Let us first reduce to the case when $(Y,D)$
  is a good log-smooth pair adapted to $v$. 
  To do so, we use Lemma~\ref{lem2_val}~(ii) to find 
  a good log-smooth pair $(Y',D')\succeq(Y,D)$ adapted to $v$.
  Furthermore, we may assume that $v\in\QM_{\eta'}(Y',D')$, 
  where $\eta'$ lies over $\eta=c_X(v)$, and that the components of $D$ (resp. $D'$)
  through $\eta$ (resp. $\eta'$) are $D_1,\ldots,D_r$ (resp. $D'_1,\ldots,D'_r$).
  We can also assume that we have the formulas (\ref{eq_prop_val0}), where 
  the $E'_i$ do not contain $\eta'$, and $\det(b_{i,j})\neq 0$. 
  
  We claim that $A_{X,Y,D}(v)=A_{X,Y',D'}(v)$.
  Note first that Lemma~\ref{lem3_val} (ii) gives $1+\ord_{D'_j}(K_{Y'/Y})=\sum_{i=1}^rb_{ij}$.
  On the other hand, we have 
  $v(E'_i)=0$ for every $i$, hence (\ref{eq_prop_val0}) implies
  $v(D_i)=\sum_{j=1}^rb_{ij}v(D'_j)$ 
  for $i=1,\dots,r$. We also have
  $\ord_{D'_j}(K_{Y/X})=\sum_{i=1}^rb_{ij}\ord_{D_i}(K_{Y/X})$
  for $j=1,\dots,r$.
  Putting these together, we get
  \begin{multline*}
    A_{X,Y,D}(v)=\sum_iv(D_i)(1+\ord_{D_i}(K_{Y/X}))
    =\sum_{i,j}b_{ij}v(D'_j)(1+\ord_{D_i}(K_{Y/X})\\
    =\sum_jv(D'_j)(1+\ord_{D'_j}(K_{Y'/Y})+\ord_{D'_j}(K_{Y/X}))\\
    =\sum_jv(D'_j)(1+\ord_{D'_j}(K_{Y'/X}))
    =A_{X,Y',D'}(v).
  \end{multline*}

   After this reduction, it suffices to show that 
   $A_{X,Y,D}(v)$ is independent of $(Y,D)$
   as long as $(Y,D)$ is good for $v$. 
   Since any two such pairs can be dominated by a third, 
   it suffices to prove that 
   $A_{X,Y,D}(v)=A_{X,Y',D'}(v)$ whenever 
   $(Y,D)$ is good for $v$ and $(Y',D')\succeq(Y,D)$.
   By Lemma~\ref{lem2_val}~(i), $(Y',D')$
   is automatically good for $v$. We can now
   proceed exactly as 
   above, using Lemma~\ref{lem2_val}~(i) and
   Lemma~\ref{lem3_val}~(ii), to show that 
   $A_{X,Y,D}(v)=A_{X,Y',D'}(v)$.
   
   It remains to prove assertion~(iii). 
   Pick any log-smooth pair $(Y,D)$ over $X'$.
   The function 
   $v\to A_{X'}(v)+v(K_{X'/X})-A_X(v)$ is
   integral linear on $\QM(Y,D)$ and vanishes
   when $v=\ord_{D_i}$ for any irreducible component
   $D_i$ of $D$. Thus this function vanishes identically,
   which proves~(iii) since $(Y,D)$ was arbitrary.
\end{proof}

\begin{remark}\label{open_immersion}
It is clear from definition that if $v\in\Val_X$ is a quasi-monomial valuation, and 
if $U\subseteq X$ is an open subset such that $c_X(v)\in U$, then
$v$ is quasi-monomial also as an element in $\Val_U$, and 
$A_U(v)=A_X(v)$.
\end{remark}

\begin{lemma}\label{lem_compatibility}
  Let $(Y',D')\succeq(Y,D)$ be log-smooth pairs over $X$
  with associated retractions $r=r_{Y,D}$ and $r'=r_{Y',D'}$,
  respectively. Then $A(r(v))\le A(r'(v))$
  for all $v\in\Val(X)$, with equality 
  if and only if $r'(v)\in\QM(Y,D)$.
\end{lemma}
\begin{proof}
  By Lemma~\ref{L201} we have $r(v)=r(r'(v))$, hence after replacing $v$ by $r'(v)$,
  we may assume that $v=r'(v)\in\QM(Y',D')$.
  Write $w:=r(v)$. 
  
  We first prove that $A(w)\le A(v)$. 
  This inequality follows from Lemma~\ref{lem3_val}~(i)
  when $v=\ord_{D'_j}$ for some irreducible
  component $D'_j$ of $D'$. It must then hold on
  all of $\QM(Y',D')$. Indeed, by 
  Lemma~\ref{L201} and by Proposition~\ref{log_discrep},
  the function $A(v)-A(r_{Y,D}(v))$
  is (integral) linear on $\QM(Y',D')$. 

  Now suppose that $v\not\in\QM(Y,D)$, that is, $v\ne w$.
  We will show that $A(v)>A(w)$. By condition (iii) in Proposition~\ref{log_discrep},
  it is enough to show that $A_Y(v)>A_Y(w)$, and therefore we may and will assume that $X=Y$.
  Furthermore, by Remark~\ref{open_immersion} we may replace $Y$ by an open neighborhood
  of $c_Y(v)$. If $c_Y(v)\neq c_Y(w)$, then the inequality $A(v)>A(w)$ follows from the first part: after replacing
  $Y$ by an open neighborhood of $c_Y(v)$, we can find a prime divisor $E$ containing $c_Y(v)$
  such that $(Y,D+E)$ is a log-smooth pair. Hence $A(v)\geq A(r_{Y,D+E}(v))>
  A(w)$.
  Therefore we may assume that $c_Y(v)=c_Y(w)$.
  Let $(Y',D')\succeq (Y,D)$ be induced by a 
  suitable toroidal blowup as in Lemma~\ref{lem2_val} (ii), centered at $c_Y(v)$,
  such that $(Y',D')$ is a good
 pair adapted to $w$. Note that $r_{Y',D'}(v)=w$. 
 Since $(X,D)$ is a good pair adapted to $w$, and $v\ne w$, we have
 $c_{Y'}(v)\neq c_{Y'}(w)$. As we have seen, this implies $A_{Y'}(v)>A_{Y'}(w)$,
 hence $A_Y(v)>A_Y(w)$.
\end{proof}

\subsection{Log discrepancy of general valuations}\label{S401}
We now extend the log discrepancy to arbitrary valuations in $\Val_X$.  
If $v$ is any valuation in $\Val_X$, then we set
\begin{equation}\label{e205}
  A(v)=A_X(v):=\sup_{(Y,D)}A(r_{Y,D}(v))\in\RR_{\ge0}\cup\{\infty\},
\end{equation}
where the supremum is over all log-smooth pairs $(Y,D)$ over $X$.
As a consequence of Lemma~\ref{L201} and 
Lemma~\ref{lem_compatibility}
we may, in the definition of $A(v)$,
take the supremum over sufficiently high pairs $(Y,D)$.
This in particular implies that for a
quasi-monomial valuation $v$, the new definition of $A(v)$ 
is equivalent to the old one.
Note that $A(v)>0$ when $v$ is nontrivial.
We also obtain
\begin{corollary}\label{cor_compatibility}
  For any log-smooth pair $(Y,D)$ over $X$
  and any valuation $v\in\Val_X$ we have 
  $A(r(v))\le A(v)$ with equality if and only 
  if $v\in\QM(Y,D)$.
\end{corollary}
\begin{remark}\label{open_immersion2}
  Suppose that $v\in\Val_X$ is an arbitrary valuation, 
  and $U$ is an open subset of $X$
  containing $c_X(v)$. It follows from Remarks~\ref{Nagata} 
  and~\ref{open_immersion} that $A_X(v)=A_U(v)$.
\end{remark}
\begin{remark}\label{R301}
  Let $\mu\colon X'\to X$ be a proper birational morphism, with $X'$
  regular. 
  It follows from Proposition~\ref{log_discrep}~(iii) that
  $A_X(v)=A_{X'}(v)+v(K_{X'/X})$
  for any valuation $v\in\Val_X=\Val_{X'}$.
\end{remark}
\begin{lemma}\label{lem_semicontinuity}
  The log discrepancy function $A\colon\Val_X\to\RR_{\ge0}\cup\{\infty\}$ 
  is lower semicontinuous. 
\end{lemma}
\begin{proof}
  Since $A$ is continuous on each $\QM(Y,D)$ and 
  $r_{Y,D}\colon \Val_X\to\QM(Y,D)$ is continuous,
  $A$ is a supremum of continuous functions, hence
  lower semicontinuous.
\end{proof}
\begin{corollary}\label{C402}
  Given $v\in\Val_X$ we have $A(v)=\sup_{Y,D}A(r_{Y,D}(v))$,
  where the supremum is taken over log-smooth pairs $(Y,D)$ 
  over $X$ such that $\pi(D)\subseteq \overline{c_X(v)}$, where
  $\pi\colon Y\to X$ is the associated morphism.
\end{corollary}
\begin{proof}
  The inequality $A(v)\ge\sup_{Y,D}A(r_{Y,D}(v))$  is definitional. 
  For the reverse inequality, fix $\epsilon>0$ and 
  first assume $A(v)<\infty$.
  By the lower semicontinuity of $A$, the subset 
  $\{A>A(v)-\epsilon\}\subseteq\Val_X$ is open and hence
  contains a valuation of the desired form $r_{Y,D}(v)$
  by Corollary~\ref{C401}. When $A(v)=\infty$, we instead look
  at the open set $\{A>\epsilon^{-1}\}$.
\end{proof}
We will later make use of the following compactness result. 
\begin{proposition}\label{compact}
  Let $\xi\in X$ be a point and $\frm$ the ideal defining 
  $\overline{\xi}$, with the reduced scheme structure.
  For every $M\in\RR_{\geq 0}$, the set
  \begin{equation*}
    V_M:=\{v\in\Val_X\mid c_X(v)=\xi, v(\frm)=1, A(v)\leq M\}
  \end{equation*}
  is a compact subspace of $\Val_X$.
\end{proposition}
\begin{proof}
  We consider an element of $\Val_X$ as a morphism of semirings
  $\cI\to\RR_{\geq 0}$, where $\cI$ is the semiring of 
  nonzero ideals on $X$ (see \S 1.1). 
  Recall that the condition $c_X(v)=\xi$ 
  simply says that $v(\fra)=0$ if $\fra\not\subseteq\frm$, 
  and $v(\fra)>0$, otherwise.
  Of course, in the presence of $v(\frm)=1$, 
  the second condition is automatically fulfilled.
  
  Note that if $A(v)\leq M$, then for every nonzero $\fra\in{\mathcal I}$ we have
  $v(\fra)\leq M\cdot \Arn(\fra)$. It follows from the definition of the topology 
  on $\Val_X$ that 
  \begin{equation*}
    W_M:=\{v\in\Val_X\mid c_X(v)=\xi, v(\frm)=1, v(\fra)\leq M\cdot \Arn(\fra)\ 
    \text{for all $\fra\in\cI$}\}
  \end{equation*}
  is a closed subset of
  $\prod_{\fra\in\cI,\fra\subseteq\frm}[1, M\cdot \Arn(\fra)]$, 
  hence a compact topological space 
  by Tychonoff's theorem. Moreover, $V_M$ is a closed subset of $W_M$, 
  since $A$ is lower semicontinuous on $\Val_X$ 
  by Proposition~\ref{lem_semicontinuity}. 
  Thus $V_M$ is compact.
\end{proof}

When $M=+\infty$, the space $V_M$ above is not compact but 
can be compactified as follows.
For simplicity assume that $\xi\in X$ is a closed point.
Let $\cV_{X,\xi}$ denote the set of all semivaluations
$v$ on $\cO_{X,\xi}$ for which $v(\frm)=1$. 
Note that we allow $v(f)=+\infty$ for nonzero $f$. 
This space $\cV_{X,\xi}$ is the valuation space considered in~\cite{BFJ1}, 
see~\S\ref{comparison} below. 
As with $\Val_X$, we equip it with the topology of 
pointwise convergence, turning it into a 
closed subset of $\prod_{f\in \frm\cdot\cO_{X\xi}}[1,\infty]$
and hence compact by Tychonoff's Theorem.
One can show that $\Val_X\cap\cV_{X,\xi}$ is dense in $\cV_{X,\xi}$,
see~\S\ref{comparison}.
We will make use of the space $\cV_{X,\xi}$ in the proof of Proposition~\ref{P303}.

\subsection{Izumi's inequality}
If $\xi$ is a point on $X$, we denote by $\frm_{\xi}$ the ideal defining the 
closure of $\xi$. Let 
$\ord_{\xi}$ be the valuation with center $\xi$, such that 
$\ord_{\xi}(f)=\sup\{r\mid f\in \frm^r\cdot\cO_{X,\xi}\}$ for every $f\in\cO_{X,\xi}$.
Note that this is a divisorial valuation: it is equal to $\ord_{E_{\xi}}$, where $E_{\xi}$ is the component
of the exceptional divisor on $\Bl_{\frm_{\xi}}(X)$ whose image contains $\xi$.
We shall later make use of the following 
Izumi-type estimate~\cite{izumi,ELS}:
\begin{proposition}\label{prop_Izumi}
  For any $v\in\Val_X$ we have 
  \begin{equation}\label{e209}
    v(\frm_\xi)\ord_\xi
    \le v
    \le A(v)\ord_\xi,
  \end{equation}
  in the sense of Definition~\ref{D201}, where $\xi=c_X(v)$.
\end{proposition}
\begin{proof}
  The first inequality follows from the definitions.
  By approximation, the second inequality can be reduced 
  to the case when
  $v$ is divisorial, and then it goes
  back at least to Tougeron~\cite[Lemma~1.3, p.178]{tougeron}.
  Alternatively, it comes from the fact that for 
  $f\in\cO_{X,\xi}$ with $\ord_\xi(f)=m$, after replacing $X$ by some open neighborhood
  of $\xi$
  we have $\frac{A(v)}{v(f)}\geq\lct(f)\geq\frac{1}{m}$. For the last inequality, 
  see~\cite[Lemma~8.10]{Kol} (this treats the case when $X$ is of finite type over a field, but the general case can be easily reduced to this one, arguing as in~\cite[Corollary~2.10]{dFM}).
\end{proof}
\begin{corollary}\label{C403}
  If $v\in\Val_X$ satisfies $A(v)<\infty$, then $v$ has a unique
  extension as a valuation $v'\in\Val_{X'}$, where
  $X'=\Spec\widehat{\cO_{X,\xi}}$, $\xi=c_X(v)$.
\end{corollary}
\begin{proof}
  Let $\frm=\frm_{\xi}$ be as above and set 
  $\frm'=\frm\cdot\widehat{\cO_{X,\xi}}$.
  One can always uniquely extend 
  $v\colon\cO_{X,\xi}\smallsetminus\{0\}\to\RR_{\geq 0}$ 
  by $\frm_{\xi}$-adic continuity 
  to $v'\colon\widehat{\cO_{X,\xi}}\smallsetminus\{0\}\to\RR_{\geq 0}\cup\{\infty\}$.
  Indeed, for $f\in\widehat{\cO_{X,\xi}}\smallsetminus\{0\}$ and $n\ge 1$, 
  we can write $\fra'_n:=(f)+\frm'^n=\fra_n\cdot\widehat{\cO_{X,\xi}}$
  for some ideal $\fra_n$ on $X$. Then 
  $v'(f)=\lim_{n\to\infty}v(\fra_n)$, where the limit is increasing.
  
  If $\xi'$ denotes the closed point of $X'$, then 
  $\ord_{\xi'}(\fra'_n)=\ord_{\xi'}(f)$ for $n>\ord_{\xi'}(f)$.
  Now, if $A(v)<\infty$, then~\eqref{e209} shows that 
  $v(\fra_n)\le A(v)\ord_\xi(\fra_n)=A(v)\ord_{\xi'}(f)$
  for $n>\ord_{\xi'}(f)$.
  This shows that $v'(f)\le A(v)\ord_{\xi'}(f)<\infty$.
\end{proof}
\begin{remark}\label{extension}
  When $A(v)=\infty$, the extension $v'$ of $v$ to 
  $\widehat{\cO_{X,\xi}}$ may not be a (finite-valued) 
  valuation.  For example,  if $v$ is defined on $k[x,y]$
  by $v(f(x,y))=\ord_t f(t, g(t))$, where
  $g(t)=\sum_{m\geq 1}t^m/{m!}$, then $v$ is a valuation 
  on $\AAA^2_k$ with center at the origin, but
  $v'(y-g(x))=\infty$.
\end{remark}

\subsection{Completion and field extension}
The following technical proposition will play an important role in 
the proofs of our main results. Note that if $\phi\colon X'\to X$ is a flat morphism
of integral schemes, then $\phi$ is dominant, hence induces a field extension
$K(X)\hookrightarrow K(X')$, that in turn induces a continuous map
$\Val_{X'}\to\Val_X$ given by restriction of valuations.

\begin{proposition}\label{P303}
  Let $\phi\colon X'\to X$ be a regular morphism, with $X$ and $X'$ schemes as before.
  Consider $v'\in \Val^*_{X'}$ and let $v\in\Val^*_X$ be the 
  valuation induced by restriction of $v'$.
  Then $A(v')\ge A(v)$.
  Furthermore, equality holds in the following cases:
  \begin{enumerate}
  \item[(i)]  
    $X'=\Spec\,\widehat{\cO_{X,\xi}}\to X$, where $\xi\in X$ and
    $v'$ is centered at the closed point on $X'$;
  \item[(ii)] 
    $X'=\AAA_K^n$ and $X=\AAA_k^n$, 
    where $K/k$ is an algebraic field extension, and $v'$ is centered at $0\in X'$.
  \end{enumerate}
\end{proposition}
To be perfectly precise we should really write  $A(v)=A_X(v)$ and
$A(v')=A_{X'}(v')$
here. Recall that we have seen in Remark~\ref{open_immersion2}
that the equality $A(v')=A(v)$ holds also in the case when $\phi$
is an open immersion.

\begin{proof}
  To prove the inequality $A(v')\ge A(v)$
  we argue as in the proof of Proposition~\ref{P302}.
  Given any log-smooth pair $(Y,D)$ over $X$, we get a 
  log-smooth pair $(Y',D')$ over $X'$, where 
  $Y'=Y\times_XX'\overset{\psi}\to Y$, with $D'=\psi^*(D)$. 
  Let $\eta=c_Y(v)$ and $\eta'=c_{Y'}(v')$. 
  If $E$ is an irreducible component of $D$ containing
  $\eta$, and if $E'$ is the connected component of $\psi^*(E)$ 
  containing $\eta'$, then $v'(E')=v(E)$, 
  hence $A(v')\geq A(r_{Y',D'}(v'))=A(r_{Y,D}(v))$. 
  Therefore $A(v')\geq A(v)$.
  
  We note that in case (ii), in order to prove that $A(v')=A(v)$, it is enough to consider the
  case when $K/k$ is a finite Galois extension. 
  Indeed, note first that $v'$ can be extended to an element $\overline{v}$ 
  of $\Val_{\AAA_{\overline{k}}^n}$, where $\overline{k}$ 
  is an algebraic closure of $k$ containing $K$. By what we have seen so far,
  $A(\overline{v})\geq A(v')\geq A(v)$, hence it is enough to 
  show that $A(\overline{v})=A(v)$. On the other hand, it is easy to see
  that the inequality we have already proved gives $A(\overline{v})=\sup_{L/k}A(v_L)$, 
  where $L$ varies over the finite Galois extensions of $k$ contained in 
  $\overline{k}$, and $v_L$ is the restriction of
  $\overline{v}$ to $\AAA_L^n$ (this follows from the fact that every log-smooth pair over 
  $\AAA_{\overline{k}}^n$ is defined over some $L$ as above). Therefore it is enough to show that
  $A(v_L)=A(v)$ for all such $L$, hence whenever considering case (ii), we will assume that
  $K/k$ is finite and Galois, with Galois group $G$.
  
  Recall that by Lemmas~\ref{lemma_valuations} and~\ref{lemma_valuations2},
  $v$ is quasi-monomial if and only if $v'$ is quasi-monomial. 
  We first prove the equality in cases~(i) and~(ii) under this assumption. 
  Let $(Y,D)$ be a good pair adapted to $v$. If $(Y',D')$ is defined as before,
  then the same argument as above  shows that it is enough to prove 
  that $(Y',D')$ is a good pair adapted to $v'$. 
  Recall that we put $\eta=c_Y(v)$ and $\eta'=c_{Y'}(v')$. 
  By assumption, we have $\dim(\cO_{Y,\eta})=\ratrk(v)$, 
  and it is enough to show that $\dim(\cO_{Y',\eta'})=\ratrk(v')$. 
  By Lemmas~\ref{lemma_valuations}
  and~\ref{lemma_valuations2} we have $\ratrk(v)=\ratrk(v')$, 
  hence it suffices to show that
  $\dim(\cO_{Y',\eta'})=\dim(\cO_{Y,\eta})$. 
  In case~(ii) this follows from the fact that the morphism
  $X'\to X$, hence also $Y'\to Y$, is finite.
  In case~(i), this follows from Corollary~\ref{cor_regular_morphism}. This completes
  the proof of $A(v')=A(v)$ when $v$ is quasi-monomial. 
 
  Before proving the general case, let us make some preparations. Let $\xi=c_X(v)$ and
  $\xi'=c_X(v')$, and 
  let $\frm$ and $\frm'$ denote the ideals on $X$ (resp. $X'$) defining  $\overline{\xi}$
  (resp. $\overline{\xi'}$), so that $\frm'=\frm\cdot\cO_{X'}$.
  Let $V\subseteq\Val_X$ (resp.\ $V'\subseteq\Val_{X'}$) be the 
  subset of valuations $w$ (resp. $w'$) for which $w(\frm)=1$
  (resp.\ $w'(\frm')=1$). The restriction map $\rho\colon V'\to V$
  is continuous but not surjective in general. 
  We claim that $\rho$ is \emph{open onto its image},
  that is, for any open subset $U'\subseteq V'$ there exists
  an open subset $U\subseteq V$ such that 
  $\rho(U')=\rho(V')\cap U$.
  First consider case~(i), when $\rho$ is injective.
  We may assume $U'$ is of the form
  \begin{equation*}
    U'=\{w'\in V'\mid s_j<w'(\fra'_j)<t_j, \ j=1,2,\dots,n\},
  \end{equation*}
  where $\fra'_j$ is an ideal on $X'$ and 
  $s_j<t_j\le+\infty$ for $1\le j\le n$.
  Pick $p$ large enough so that $p>t_j$ whenever $t_j<\infty$. 
  As $w'(\frm')=1$, replacing $\fra'_j$ by $\fra'_j+\frm'^p$
  does not change $U'$. We may therefore assume that 
  $\fra'_j=\fra_j\cdot\cO_{X'}$, for some ideal $\fra_j$ on $X$.
  But then $\rho(U')=\rho(V')\cap U$, where
  $U=\bigcap_j\{w\in V\mid s_j<w(\fra_j)<t_j\}$ is open.
  
  Suppose now that we are in case (ii), when $\rho$ is surjective.  
  It is convenient to consider the extension of $\rho$ to the spaces of
  semivaluations introduced at the end of~\S\ref{S401}.
  More precisely, we have a continuous surjective map induced by restriction
  $\widetilde{\rho}\colon \cV_{X',\xi'}\to\cV_{X,\xi}$. 
  We can identify $V'$ and $V$ with subspaces of $\cV_{X',\xi}$ and
  $\cV_{X,x}$, respectively,
  such that $\rho$ is the restriction of $\widetilde{\rho}$ over $V$. 
  Therefore it is enough to show that $\widetilde{\rho}$ is open. 

Note that the Galois group $G=G(K/k)$ has a natural 
action on $X'$, which induces an action on $\cV_{X',\xi'}$ such that the fibers of 
$\widetilde{\rho}$ coincide with the $G$-orbits of this action. If $U'$ is an open subset of
$\cV_{X',\xi'}$, then $\widetilde{\rho}^{-1}(\widetilde{\rho}(U'))=\bigcup_{g\in G}gU'$ is open in
$\cV_{X',\xi'}$, hence its complement is closed. Since $\cV_{X',\xi'}$ is compact, it follows that
$\widetilde{\rho}(F)$ is compact, hence closed in $\cV_{X,\xi}$. Therefore
$\cV_{X,\xi}\smallsetminus\widetilde{\rho}(F)=\widetilde{\rho}(U')$ is open in $\cV_{X,\xi}$.
  
    We can now prove that $A(v)\ge A(v')$.
  First suppose $A(v')<\infty$.
 
  Given $\epsilon>0$, the set 
  $U'=U'_\epsilon:=\{w'\in V'\mid A(w')>A(v')-\epsilon\}$ is an
  open subset of $V'$ by the lower
  semicontinuity of $A$.
  By what precedes, there exists an open subset $U\subseteq V$ 
  such that $\rho(U')=\rho(V')\cap U$. 
  Clearly $v\in U$, so by Corollary~\ref{C401} we can find
  a log-smooth pair $(Y,D)$ above $X$
  such that $w:=r(v)\in U$, where 
  $r=r_{Y,D}\colon\Val_X\to\QM(Y,D)$ is the corresponding retraction.
  Since $w$ is quasi-monomial, $w$ is in the image of $\rho$,
  so there exists $w'\in U'$ quasi-monomial such that 
  $\rho(w')=w$.  We have seen that $A(w')=A(w)$.
  Thus $A(v)\ge A(w)=A(w')>A(v')-\epsilon$.
  As $\epsilon\to0$ we get $A(v)\ge A(v')$.
  The case when $A(v')=\infty$ is treated similarly,
  by setting $U'=\{w'\in V'\mid A(w')>\epsilon^{-1}\}$.
\end{proof}

\section{Graded and subadditive systems revisited}\label{S307}
We now extend to arbitrary valuations some of the results 
proved in \S\ref{S301}--\S\ref{S302} for divisorial valuations.

\subsection{Induced functions on valuation space}
\begin{lemma}\label{C201}
  If $\fra_\bullet$ is a graded sequence of ideals, then the function
  $v\mapsto v(\fra_\bullet)$ is upper semicontinuous on $\Val_X$.
  Similarly, if $\frb_\bullet$ is any subadditive system of ideals,
  then the function $v\mapsto v(\frb_\bullet)$ is lower semicontinuous. 
\end{lemma}
\begin{proof}
  For each $t$, the function $v\mapsto v(\frb_t)$ is
  continuous on $\Val_X$ by Lemma~\ref{L204}.
  Hence $v\mapsto v(\frb_\bullet)=\sup_t\frac1tv(\frb_t)$
  is lower semicontinuous. The argument for 
  $\fra_\bullet$ is analogous.
\end{proof}
\begin{proposition}\label{prop2_arbitrary_val}
  If $\fra_{\bullet}$ is a graded sequence of ideals, 
  and $\frb_{\bullet}$ is the corresponding
  system of asymptotic multiplier ideals, 
  then for every $m$ such that $\fra_m$ is nonzero we have
  \begin{equation}\label{eq_prop2_arbitrary_val}
    v(\fra_{\bullet})-\frac{A(v)}{m}
    \le\frac{v(\fra_m)}{m}-\frac{A(v)}{m}
    \leq \frac{v(\frb_m)}{m}
    \le\frac{v(\fra_m)}{m}
  \end{equation}
  for all
  $v\in\Val_X$, with the second inequality being strict when $v$ is nontrivial.
  In particular,
  $v(\fra_{\bullet})=v(\frb_{\bullet})$ whenever $A(v)<\infty$.
\end{proposition}
\begin{proof}
  The first inequality is definitional and the last inequality
  follows from the inclusion $\fra_m\subseteq\frb_m$.
  To prove~\eqref{eq_prop2_arbitrary_val} it therefore
  suffices to show that the function 
  $h_m(v):=v(\frb_m)-v(\fra_m)+A(v)$ is positive
  on $\Val^*_X$. 
  Let $(Y,D)$ be a log-smooth pair that 
  defines a log resolution of $\fra_m\cdot\frb_m$.
  Then $v(\fra_m)=r_{Y,D}(v)(\fra_m)$, 
  $v(\frb_m)=r_{Y,D}(v)(\frb_m)$ and $A(v)\ge A(r_{Y,D}(v))$,
  so it suffices to show that $h_m$ is positive on the nontrivial valuations in $\QM(Y,D)$.
  But $h_m$ is linear on $\QM(Y,D)$ and we already know
  that $h_m(\ord_{D_i})> 0$ for any irreducible component
  $D_i$ of $D$. Hence $h_m>0$ on $\Val^*_X$, as claimed.

  The last assertion follows by letting $m\to\infty$ along a suitable subsequence.
\end{proof}
\begin{remark}\label{all_t}
The same argument shows that if $\fra_{\bullet}$ and $\frb_{\bullet}$ are as above, then
$\frac{v(\frb_t)}{t}>v(\fra_{\bullet})-\frac{A(v)}{t}$ for every $v\in\Val_X^*$ with $A(v)<\infty$, and for every $t>0$. 
\end{remark}

\begin{corollary}\label{cor_prop2_arbitrary_val}
If $\fra_{\bullet}$ is a graded sequence of ideals on $X$,
  then the function 
  $v\mapsto v(\fra_{\bullet})$ is continuous on 
  $\{v\in\Val_X\mid A(v)<\infty\}$.
\end{corollary}

\begin{proof}
Let $W=\{v\in\Val_X\mid A(v)<\infty\}$.
 Proposition~\ref{prop2_arbitrary_val} gives 
$v(\fra_{\bullet})=v(\frb_{\bullet})$ for $v\in W$. Therefore Lemma~\ref{C201} implies that the map 
$v\to v(\fra_{\bullet})$
is both lower and upper semicontinuous, hence continuous,
on W.
\end{proof}

\begin{proposition}\label{prop1_arbitrary_val}
  If $\frb_{\bullet}$ is a subadditive sequence of ideals on $X$, of controlled growth, then 
  \begin{equation}\label{eq_prop1_arbitrary_val}
    v(\frb_{\bullet})-\frac{A(v)}{t}\leq \frac{v(\frb_t)}{t}\le v(\frb_\bullet)
  \end{equation}
  for every $t$ and every $v\in\Val_X$.
\end{proposition}
\begin{proof}
  The second inequality is definitional.
  For the first inequality, 
  it is enough to show that for every $s$ and every $v\in\Val_X$, we have 
  \begin{equation}\label{eq2_prop1_arbitrary_val}
    h(v):=\frac{v(\frb_t)}{t}-\frac{v(\frb_s)}{s}+\frac{A(v)}{t}\geq 0.
  \end{equation}
  Pick a log-smooth pair $(Y,D)$ that defines a 
  log resolution of $\frb_s\cdot\frb_t$. 
  Then $h\ge h\circ r_{Y,D}$ so it suffices to prove
  $h(v)\geq 0$ when $v\in\QM(Y,D)$.
  But this follows since $h$ is linear on $\QM(Y,D)$ and 
  $h(\ord_{D_i})>0$ for every irreducible component $D_i$ of $D$.
\end{proof}
\begin{corollary}\label{cor_prop1_arbitrary_val}
  If $\frb_{\bullet}$ is a subadditive sequence of ideals on $X$,
  of controlled growth, then the function 
  $v\mapsto v(\frb_{\bullet})$ is continuous on any subset of 
  $\Val_X$ on which $A$ is bounded. 
  In particular,  $v\mapsto v(\frb_\bullet)$ is 
  continuous on $\QM(Y,D)$ for any log-smooth  pair $(Y,D)$.
\end{corollary}
\begin{proof}
  The function $v\mapsto h(v):=v(\frb_\bullet)$
  is the pointwise limit of the continuous functions $v\mapsto \frac1mv(\frb_m)$
  and by Proposition~\ref{prop1_arbitrary_val} the convergence
  is uniform on subsets where $A$ is bounded. 
  This proves the first assertion. For the second assertion, note that 
  $h$ is continuous on $\QM(Y,D)\cap\{A\le 1\}$,
  hence on all of $\QM(Y,D)$ by homogeneity.
\end{proof}

\subsection{Jumping numbers}
\begin{lemma}\label{lem1_arbitrary_val}
  If $\fra$ and $\frq$ are nonzero ideals on $X$, then
  \begin{equation}\label{e206}
    \Arn^{\frq}(\fra)=\max_{v\in\Val_X^*}\frac{v(\fra)}{A(v)+v(\frq)}.
  \end{equation}
  Suppose that $\fra\neq\cO_X$ and $(Y,D)$ is a log-smooth pair over 
  $X$ giving a log resolution of $\fra\cdot\frq$.
  Then equality in~\eqref{e206} is achieved for $v$ 
  if and only if $v\in\QM(Y,D)$ and
  $\ord_{D_i}$ computes $\Arn^{\frq}(\fra)$ for every irreducible component 
  $D_i$ of $D$ for which $v(D_i)>0$.
  In particular, $v$ must be quasi-monomial.
\end{lemma}
\begin{proof}
  Let $\chi(v)=\frac{v(\fra)}{A(v)+v(\frq)}$.
  By Corollary~\ref{C202} and Corollary~\ref{cor_compatibility}, 
  $\chi\circ r_{Y,D}\ge\chi$ with strict inequality if $v\not\in\QM(Y,D)$ and $v(\fra)>0$.
  Thus $v$ achieves the maximum in (\ref{e206}) if and only if $v\in\QM(Y,D)$
  and $v$ belongs to the zero locus of the function 
  $v\mapsto v(\fra)-\Arn^\frq(\fra)(A(v)+v(\frq))$.
  But this function is linear on $\QM(Y,D)$. The result follows.
\end{proof}
\begin{corollary}\label{cor1_arbitrary_val}
  If $\frb_{\bullet}$ is a subadditive system of ideals, 
  and $\frq$ is a nonzero ideal, then
 \begin{equation}\label{e207}
   \Arn^{\frq}(\frb_{\bullet})
   =\sup_{v\in\Val_X^*, A(v)<\infty}\frac{v(\frb_{\bullet})}{A(v)+v(\frq)}.
 \end{equation}
\end{corollary}
\begin{proof}
  By Proposition~\ref{prop1}, we only need to show that
  \begin{equation}\label{eq_cor1_arbitrary_val}
    \Arn^{\frq}(\frb_{\bullet})\geq\frac{v(\frb_{\bullet})}{A(v)+v(\frq)}
  \end{equation}
  when $A(v)<\infty$. But Lemma~\ref{lem1_arbitrary_val} gives
  $\Arn^{\frq}(\frb_t)\cdot (A(v)+v(\frq))\geq v(\frb_t)$
  for every $t>0$. Dividing by $t$ and letting $t\to\infty$ 
  gives~\eqref{eq_cor1_arbitrary_val}.
\end{proof}
\begin{corollary}\label{cor_Arnold_graded_gen_val}
  If $\fra_{\bullet}$ is a graded sequence of ideals, 
 then for every nonzero ideal $\frq$ we have 
  \begin{equation}\label{e208}
    \Arn^{\frq}(\fra_{\bullet})=\sup_{v\in\Val_X^*}\frac{v(\fra_{\bullet})}{A(v)+v(\frq)}.
  \end{equation}
\end{corollary}
\begin{proof}
  This follows by combining Corollary~\ref{cor1_arbitrary_val} and 
  Propositions~\ref{prop2_arbitrary_val} and~\ref{prop2_2}. 
\end{proof}
As a consequence of Corollary~\ref{cor_Arnold_graded}
and Corollary~\ref{cor_Arnold_graded_gen_val} we get
\begin{corollary}\label{C301}
  For graded sequence $\fra_{\bullet}$ of ideals, the following
  conditions are equivalent:
  \begin{itemize}
  \item[(i)]
    $\Arn(\fra_\bullet)=0$;
  \item[(ii)]
    $\ord_E(\fra_\bullet)=0$ for all divisors $E$ over $X$;
  \item[(iii)]
    $v(\fra_\bullet)=0$ for all $v\in\Val_X$ with $A(v)<\infty$;
  \item[(iv)]
    $\Arn^\frq(\fra_\bullet)=0$ for every nonzero ideal $\frq$ on $X$.
 \end{itemize}   
\end{corollary}
\begin{remark}
  The right-hand side of~\eqref{e207} is a priori undefined 
  when $A(v)=\infty$ as in this case we could also have $v(\frb_\bullet)=\infty$.
  On the other hand, for a graded sequence $\fra_\bullet$ we always have 
  $v(\fra_\bullet)<\infty$, so the right-hand side of~\eqref{e208}
  is well-defined for any nontrivial valuation $v\in\Val_X^*$.
\end{remark}

\subsection{Comparison with other valuation spaces}\label{comparison}
While our usage of the valuation space $\Val_X$ is, to our knowledge, new, 
it is certainly related to other approaches. For simplicity, suppose
that $X$ is a smooth variety over an algebraically closed field 
$k$ of characteristic zero and equip $k$ with the trivial valuation. 

In this context, the Berkovich space $\Xan$
is defined (as a topological space) 
as follows~\cite{Ber1}. When $X=\Spec A$ is affine,
$\Xan$ is the set of semivaluations 
$v\colon A\to[0,+\infty]$ whose restriction to $k$
is trivial. In general, 
$\Xan$ is obtained by gluing the 
subsets $U^\mathrm{an}$, where 
$U$ ranges over an open affine covering of $X$.
Just as for $\Val_X$, the topology on $\Xan$
is defined in terms of pointwise convergence.
Thus $\Val_X$ embeds in $\Xan$.

In fact, $\Val_X$ is dense in $\Xan$. Let us sketch a 
proof for completeness. We may assume $X=\Spec A$
is affine. Consider any $v\in\Xan$. 
If $v\notin\Val_X$, then the prime ideal $I:=\{v=\infty\}\subseteq A$
is nonzero. 
Now look at the prime ideal $J:=\{v>0\}\supseteq I$
with associated point $\xi\in X$. If $I=J$, then 
the semivaluation $v$ satisfies $v(f)=\infty$ if 
$f\in J$ and $v(f)=0$ otherwise. 
Hence the divisorial valuation $n\ord_\xi\in\Val_X$ 
tends to $v$ as $n\to\infty$.
Now suppose $J\supsetneq I$ so that $0<v(J)<\infty$.
If $(Y,D)$ is a log-smooth pair above $X$ for which the 
associated birational morphism $\phi\colon Y\to X$
satisfies $\phi(D)\subset\overline{\xi}$, then we
can define the retraction $r_{Y,D}(v)\in\Val_X$ 
as in~\S\ref{S402}. We claim that every neighborhood
of $v$ in $\Xan$ contains an element of the form 
$r_{Y,D}(v)$. To see this, fix $f\in A$. It suffices to
find a sequence $(Y_n,D_n)$ such that 
$r_{Y_n,D_n}(v)(f)\to v(f)$ as $n\to\infty$, but for this
we may take $(Y_n,D_n)$ to be a log resolution of $(f)+J^n$. 

When $X$ is projective, $\Xan$ is compact~\cite[Theorem~3.5.3]{Ber1},
hence defines a compactification of $\Val_X$. Note that while
$\Val_X$ is invariant under proper birational morphisms, $\Xan$ is not.

Given any closed point $\xi\in X$ we can also, as in~\S\ref{S401},
consider the compact subset $\cV_{X,\xi}\subseteq \Xan$ consisting
of semivaluations for which $v(\frm_\xi)=1$.
This is the valuation space studied in~\cite{BFJ1},
By~\cite[Theorem~1.16]{BFJ1} (see also~\cite{Ber1,Thu2}),
$\cV_{X,\xi}$ is contractible.
The argument above shows that 
$\Val_X\cap\cV_{X,\xi}$ is dense in $\cV_{X,\xi}$.
In fact, for each log-smooth pair $(Y,D)$ as above, 
$\cV_{X,\xi}\cap\QM(Y,D)$ is a simplicial complex
and by~\cite[Theorem~1.13]{BFJ1}, $\cV_{X,\xi}$ is
homeomorphic to the 
projective limit of these complexes.
It follows from Corollary~\ref{C402} that the
log discrepancy defined in~\cite[Definition~3.4]{BFJ1} 
(called thinness there) coincides with the one defined in this paper.

In~\cite{BFJ1}, a class of plurisubharmonic (psh) functions
on $\cV_{X,\xi}$ was defined. A posteriori, a function 
on $\cV_{X,\xi}$ is plurisubharmonic if and only if 
it is of the form $v\to-v(\frb_\bullet)$ 
where $\frb_\bullet$ is a 
subadditive system of controlled growth
satisfying $\frb_t\supseteq\frm_\xi^{p_t}$,
for each $t$, where $p_t\ge 1$.
This cone of psh functions has good compactness properties and 
is studied in detail in~\cite{BFJ3}.

In dimension $\dim X=2$, $\cV_{X,\xi}$ is naturally an
$\RR$-tree, being a contractible projective limit of 
one-dimensional simplicial complexes.
A function on $\cV_{X,\xi}$ is psh if and only if it 
satisfies certain convexity 
conditions~\cite{FJ1,BR,Thu1}.\footnote{In~\cite{FJ1,FJ2,FJ3}, the negative of a 
  psh function is called a tree potential.}
This allows us
to construct graded and subadditive sequences with 
interesting behavior. 
For example, given coordinates $(x,y)$ at $\xi$
and a strictly increasing sequence 
$1\le\beta_1<\beta_2<\dots$
of rational numbers with unbounded denominators
we can define a valuation $v\in\Val_X\cap\cV_{X,\xi}$ 
by $v(f)=\ord_{x=0}f(x,\sum_jx^{\beta_j})$.
Such a valuation satisfies $\trdeg(v)=0$, $\ratrk(v)=1$
and is called \emph{infinitely singular} in~\cite{FJ1}.
If the $\beta_j$ grow sufficiently fast, then there 
exists a psh function $\varphi$ on $\cV_{X,\xi}$ for which 
$\varphi(\ord_0)=-1$ and $\varphi(v)=-\infty$.
This translates into the existence of a 
subadditive sequence $\frb_\bullet$
of controlled growth such that 
$\ord_0(\frb_\bullet)=1$ but $v(\frb_\bullet)=\infty$.
In particular, $A(v)=\infty$.
One can also show that the associated graded system 
$\fra_\bullet=\fra_\bullet(v)$
satisfies $v(\fra_\bullet)=1$ but $w(\fra_\bullet)=0$
for all $w\in\cV_{X,\xi}\setminus\{v\}$.

\section{Valuations computing asymptotic invariants} \label{S308}
Now we are ready to formulate our main results and conjectures.
We keep our previous setup.
\subsection{Results and conjectures}
\begin{definition}
  A valuation $v\in\Val_X^*$ \emph{computes} 
  $\Arn^{\frq}(\fra_{\bullet})$, for a nonzero
  ideal $\frq$ and a graded sequence of ideals $\fra_{\bullet}$ if 
  $\Arn^{\frq}(\fra_{\bullet})=\frac{v(\fra_{\bullet})}{A(v)+v(\frq)}$.
\end{definition}
Equivalently, $v$ then computes $\lct^\frq(\fra_\bullet)$.
Of course, if $\Arn^\frq(\fra_\bullet)=0$, any valuation computes
$\Arn^\frq(\fra_\bullet)$, so we shall focus on the case 
$\Arn^\frq(\fra_\bullet)>0$ in the sequel. In this case
any $v$ computing $\Arn^\frq(\fra_\bullet)$ satisfies $A(v)<\infty$.
\begin{remark}\label{rem1_conjecture}
  If $v\in\Val^*_X$ computes $\Arn^{\frq}(\fra_{\bullet})>0$,
  then $c_X(v)$ lies in the zero-locus of 
  $(\frb_{\lambda}\colon\frq)$, where $\lambda=\lct^{\frq}(\fra_{\bullet})$ 
  and $\frb_{\lambda}=\cJ(\fra_{\bullet}^{\lambda})$. Indeed,
  if $f$ is a local section of $(\frb_{\lambda}\colon\frq)$ 
  defined in a neighborhood of $c_X(v)$, then 
  $f\cdot\frq\subseteq\frb_{\lambda}$, hence
  \begin{equation*}
    v(f)+v(\frq)\geq v(\frb_{\lambda})>\lambda\cdot v(\fra_{\bullet})-A(v)=v(\frq),
  \end{equation*}
  in view of Remark~\ref{all_t}. It follows that $v(f)>0$, so $f$ vanishes at $c_X(v)$.
\end{remark}

The following result generalizes Theorem~A from the introduction.
\begin{theorem}\label{main_thm1}
  Let $\fra_{\bullet}$ be  a graded sequence of ideals on $X$, 
  and $\frq$ a nonzero ideal.
  If $\Arn^{\frq}(\fra_{\bullet})=\lambda^{-1}>0$,
  then for every generic point $\xi$ of an irreducible
  component of $V(\cJ(\fra_{\bullet}^{\lambda})\colon\frq)$, 
  there is a valuation $v\in\Val^*_X$ that computes 
  $\Arn^{\frq}(\fra_{\bullet})$, with $c_X(v)=\xi$.
\end{theorem}
As we will see in Remark~\ref{rem_divisorial} below, 
the valuation $v$ cannot always be taken divisorial.
However, we state

\begin{conjecture}\label{conj1}
  Let $\fra_{\bullet}$ be a graded sequence of ideals on $X$ and 
  $\frq$ a nonzero ideal on $X$ such that 
  $\Arn^\frq(\fra_\bullet)=\lambda^{-1}>0$. 
  \begin{itemize}
  \item \textbf{Weak version}:
    for any generic point $\xi$ of an irreducible component of the subscheme defined by
    $(\cJ(\fra_{\bullet}^{\lambda})\colon\frq)$,
    there exists a quasi-monomial valuation $v\in\Val^*_X$ that
    computes $\Arn^{\frq}(\fra_{\bullet})$  and with $c_X(v)=\xi$.
  \item \textbf{Strong version}: 
    any valuation $v\in\Val^*_X$ that computes
    $\Arn^{\frq}(\fra_{\bullet})$ must be quasi-monomial.
  \end{itemize}
\end{conjecture}
While we are unable to prove either version of 
Conjecture~\ref{conj1}, we shall reduce them, 
in two ways, to statements that hopefully are easier to prove.
First, we reduce to the case of an affine space over
an algebraically closed field.
\begin{conjecture}\label{conj2}
  Let $X=\AAA_k^n$, where 
  $k$ is an algebraically closed field of characteristic zero
  and where $n\ge 1$.
  Let $\fra_{\bullet}$ be a graded sequence of ideals on $X$ and 
  $\frq$ a nonzero ideal on $X$ such that 
  $\Arn^\frq(\fra_\bullet)>0$ and such that 
  $\fra_1\supseteq\frm^p$, where $p\ge 1$ and 
  $\frm=\frm_\xi$ is the ideal defining a closed point $\xi\in X$.
 \begin{itemize}
  \item \textbf{Weak version}:
    there exists a quasi-monomial valuation $v\in\Val^*_X$ 
    computing $\Arn^\frq(\fra_\bullet)$ and with $c_X(v)=\xi$.
  \item \textbf{Strong version}: 
    any valuation $v\in\Val^*_X$ of transcendence degree 0
    computing $\Arn^\frq(\fra_\bullet)$ must be quasi-monomial.
  \end{itemize} 
\end{conjecture}

The strong version of Conjecture~\ref{conj2} 
is trivially true in dimension one. 
A proof in dimension two will be given in \S\ref{S310}
and the monomial case is treated in \S\ref{S101}.
The following result strengthens Theorem~D in
the introduction.
\begin{theorem}\label{main_thm2}
  If the weak (resp.\ strong) version of 
  Conjecture~\ref{conj2} holds for every $n\leq N$, then 
  the weak (resp.\ strong) version of Conjecture~\ref{conj1} holds 
  for all $X$ with $\dim(X)\leq N$. 
\end{theorem}

Second, we reduce to the case of a graded sequence of
valuation ideals.
\begin{theorem}\label{main_thm3}
  In Conjecture~\ref{conj1} (weak and strong version)
  we may assume that $\fra_\bullet$ is a graded sequence of
  valuation ideals, that is, $\fra_m=\{f\mid w(f)\ge m\}$ 
  for some $w\in\Val^*_X$.
\end{theorem}
We also have a related result.
\begin{theorem}\label{equivalent_description}
  Let $v\in\Val^*_X$ be a nontrivial valuation with $A(v)<\infty$
  and $\frq$ a nonzero ideal on $X$.
  Then the following assertions are equivalent:
  \begin{enumerate}
  \item[(i)] 
    There is a graded sequence of ideals $\fra_{\bullet}$ on $X$ 
    such  that $v$ computes $\Arn^{\frq}(\fra_{\bullet})>0$.
  \item[(ii)] 
    There is a subadditive system of ideals $\frb_{\bullet}$ 
    (which can be assumed of controlled growth) such that
    $v$ computes $\Arn^{\frq}(\frb_{\bullet})>0$. 
  \item[(iii)] 
    For every $w\in\Val_X$ such that $w\geq v$ 
    in the sense of Definition~\ref{D201}, we have
    $A(w)+w(\frq)\geq A(v)+v(\frq)$.
 \item[(iv)] 
   If $\fra'_m=\{f\mid v(f)\geq m\}$, then $v$ computes 
   $\Arn^{\frq}(\fra'_{\bullet})$.
  \end{enumerate}
\end{theorem}
In~(ii), by a valuation $v\in\Val_X^*$ computing 
$\Arn^\frq(\frb_\bullet)$ we mean that $A(v)<\infty$ and 
$v(\frb_\bullet)/(A(v)+v(\frq))=\Arn^\frq(\frb_\bullet)$.

{}From the equivalence of~(i) and~(iii) we obtain
\begin{corollary}
  If $\frq$ is a nonzero ideal on $X$ and $v$ 
  computes $\Arn(\fra_{\bullet})>0$ for 
  some graded sequence $\fra_{\bullet}$, 
  then $v$ also computes $\Arn^{\frq}(\widetilde{\fra}_{\bullet})>0$ 
  for some (other) graded sequence~$\widetilde{\fra}_{\bullet}$.
\end{corollary}

\subsection{Valuation ideals}
Now we give the proofs of the results in Section~\ref{S308}.
We start by the reductions to the case of graded sequences of
valuation ideals,
specifically Theorems~\ref{main_thm3} and~\ref{equivalent_description}.

\begin{proof}[Proof of Theorem~\ref{equivalent_description}]
  We will show that 
  (i)$\Rightarrow$(ii)$\Rightarrow$(iii)$\Rightarrow$(iv)$\Rightarrow$(i)

  The implication (i)$\Rightarrow$(ii) follows from 
  Proposition~\ref{prop2_2} and Proposition~\ref{prop2_arbitrary_val}:
  it is enough to take $\frb_{\bullet}$ to be given by the asymptotic 
  multiplier ideals of $\fra_{\bullet}$.

  In order to show (ii)$\Rightarrow$(iii), 
  suppose that $v$ computes $\Arn^{\frq}(\frb_{\bullet})>0$. 
  If $w\ge v$, then clearly $w(\frb_{\bullet})\geq v(\frb_{\bullet})$. 
  Now Corollary~\ref{cor1_arbitrary_val} gives
  $\frac{w(\frb_{\bullet})}{A(w)+w(\frq)}\leq\frac{v(\frb_{\bullet})}{A(v)+v(\frq)}$, hence
  $\frac{A(w)+w(\frq)}{A(v)+v(\frq)}\geq\frac{w(\frb_{\bullet})}{v(\frb_{\bullet})}\geq 1$.
  Therefore we have~(iii). 
  
  Now suppose~(iii) holds.
  By Lemma~\ref{L301}, $v(\fra'_\bullet)=1$.
  To prove~(iv) it therefore suffices,
  by Corollary~\ref{cor_Arnold_graded_gen_val},
  to show that for every $w\in\Val^*_X$ we have
  \begin{equation}\label{eq2_equivalent_description}
    \frac{w(\fra'_{\bullet})}{A(w)+w(\frq)}\leq \frac{1}{A(v)+v(\frq)}.
  \end{equation}
  If $w(\fra'_\bullet)=0$, then~\eqref{eq2_equivalent_description}
  is trivial, so suppose $w(\fra'_\bullet)>0$.
  Since the left hand side is invariant under scaling of $w$,
  we may assume $w(\fra'_\bullet)=1$.
  By Lemma~\ref{L301} this implies $w\ge v$.
  The assumption~(iii) now gives $A(w)+w(\frq)\geq A(v)+v(\frq)$,
  so that~\eqref{eq2_equivalent_description} holds.

  Finally, the implication~(iv)$\Rightarrow$(i) is trivial: if $v$ computes
  $\Arn^{\frq}(\fra'_{\bullet})$, then $\Arn^{\frq}(\fra'_{\bullet})=(A(v)+v(\frq))^{-1}>0$. 
  This completes the proof.
\end{proof}

Now we turn to Theorem~\ref{main_thm3}. The assertion corresponding
to the strong versions of the conjectures follows from
the implication~(i)$\Rightarrow$(iv) in Theorem~\ref{equivalent_description}.
The assertion concerning the weak statements of the conjectures
is a consequence of Theorem~\ref{main_thm1} and the following result.
\begin{proposition}\label{reduction_special_case}
  Assume that $v\in\Val_X^*$ computes $\Arn^\frq(\fra_\bullet)>0$
  and define $\fra'_\bullet$ by $\fra'_m=\{f\mid v(f)\ge m\}$.
  Then $\Arn^\frq(\fra'_\bullet)=\Arn^\frq(\fra_\bullet)$
  and any $w\in\Val_X^*$ that computes $\Arn^\frq(\fra'_\bullet)$
  also computes $\Arn^\frq(\fra_\bullet)$.
\end{proposition}  
\begin{proof}
  Since $v\in\Val_X$ computes $\Arn^\frq(\fra_\bullet)>0$
  we must have $A(v)<\infty$ and $v(\fra_{\bullet})>0$.
  After rescaling $v$, we may assume $v(\fra_{\bullet})=1$. 
  By Lemma~\ref{L301}  we also have $v(\fra'_\bullet)=1$.
  Since $v$ computes $\Arn^{\frq}(\fra_{\bullet})$, it also computes 
  $\Arn^{\frq}(\fra'_{\bullet})$ by Theorem~\ref{equivalent_description}.
  This yields
  \begin{equation*}
    \Arn^{\frq}(\fra'_{\bullet})
    =\frac{v(\fra'_{\bullet})}{A(v)+v(\frq)}
    =\frac{1}{A(v)+v(\frq)}
    =\frac{v(\fra_\bullet)}{A(v)+v(\frq)}
    =\Arn^{\frq}(\fra_{\bullet}).
  \end{equation*}
  Now $v(\fra_\bullet)=1$ implies 
  $\fra_m\subseteq\fra'_m$ for every $m$. 
  In particular, $w(\fra'_m)\leq w(\fra_m)$ for all $m$
  and all $w\in\Val_X^*$, hence $w(\fra'_{\bullet})\leq w(\fra_{\bullet})$. 
  If $w$ computes $\Arn^\frq(\fra'_\bullet)$, we therefore
  get
 \begin{equation*}
    \Arn^{\frq}(\fra_{\bullet})
    =\Arn^{\frq}(\fra'_{\bullet})=\frac{w(\fra'_{\bullet})}{A(w)+w(\frq)}
    \leq\frac{w(\fra_{\bullet})}{A(w)+w(\frq)},
  \end{equation*}
  so that $w$ computes $\Arn^{\frq}(\fra_{\bullet})$
  (and $w(\fra_\bullet')=v(\fra_\bullet')$).
\end{proof}

\subsection{Birational and regular morphisms}
Throughout this subsection, $\varphi\colon X'\to X$ is 
a morphism that is either proper birational or regular.
Let $\fra_\bullet$ be a graded sequence of ideals on $X$, 
$\frq$ a nonzero ideal on $X$ and 
$\fra'_\bullet$, $\frq'$ their transforms to $X'$,
defined by $\fra'_m:=\fra_m\cdot\cO_{X'}$ and
$\frq':=\frq\cdot\cO_{X'}(-K_{X'/X})$ (in the birational case) or $\frq'=\frq\cdot\cO_{X'}$
(in the regular case).
\begin{lemma}\label{compute_open_subset}
  Suppose that $\varphi\colon X'\to X$ is a proper birational
  morphism with $X'$ regular.
  Then $\Arn^{\frq'}(\fra'_{\bullet})=\Arn^{\frq}(\fra_{\bullet})$.
  Moreover, $v\in\Val_X=\Val_{X'}$  
  computes $\Arn^{\frq'}(\fra'_{\bullet})$ 
  if and only if it computes $\Arn^{\frq'}(\fra'_{\bullet})$.
\end{lemma}
\begin{proof}
  The equality $\Arn^{\frq'}(\fra'_{\bullet})=\Arn^{\frq}(\fra_{\bullet})$ 
  is exactly Proposition~\ref{prop_birational}.
  The last assertion in the lemma follows from
  $v(\fra_{\bullet})=v(\fra'_{\bullet})$, 
  $v(\frq')=v(\frq)+v(K_{X'/X})$ and 
  $A_X(v)=A_{X'}(v)+v(K_{X'/X})$; see Remark~\ref{R301}.
\end{proof}
\begin{proposition}\label{P401}
  Suppose that $\Arn^\frq(\fra_\bullet)=\lambda^{-1}>0$.
  Let $\xi\in X$ be a point in the subscheme defined by
  $(\cJ(\fra_{\bullet}^{\lambda})\colon\frq)$.
  Let $\phi\colon X'\to X$ be the canonical morphism,
  where $X'=\Spec\widehat{\cO_{X,\xi}}$.
  Then $\Arn^{\frq}(\fra_{\bullet})=\Arn^{\frq'}(\fra'_{\bullet})$.
  Moreover, if $v'\in\Val_{X'}$ is a valuation centered at the closed
  point and $v\in\Val_X$ denotes
  its restriction to $X$, then $v$ computes
  $\Arn^{\frq}(\fra_{\bullet})$ if and only if
  $v'$ computes $\Arn^{\frq'}(\fra'_{\bullet})$.
\end{proposition}
\begin{proof}
  Since $\phi$ is regular (recall that $X$ is excellent),
  the equality of Arnold multiplicities follows from 
  Proposition~\ref{P302}. 
  Proposition~\ref{P303} implies $A(v')=A(v)$.
  Since $v'(\frq')=v(\frq)$ and $v'(\fra'_\bullet)=v(\fra_\bullet)$,
  it is now clear that $v'$ computes $\Arn^{\frq'}(\fra'_\bullet)$
  if and only if $v'$ computes $\Arn^{\frq}(\fra_\bullet)$.
\end{proof}
\begin{proposition}\label{P402}
  If $K/k$ is an algebraic field extension and
  $\varphi\colon X'=\AAA_K^n\to\AAA_k^n=X$ is the canonical map,
  then $\Arn^{\frq}(\fra_{\bullet})=\Arn^{\frq'}(\fra'_{\bullet})$.
  Moreover, for $v'\in\Val^*_{X'}$ let $v\in\Val^*_X$ be the 
  restriction of $v'$ to $X$. 
  Then $v$ computes $\Arn^{\frq}(\fra_{\bullet})$ if and only if
  $v'$ computes $\Arn^{\frq'}(\fra'_{\bullet})$.
\end{proposition}
\begin{proof}
  Since $\varphi$ is regular and faithfully flat (see Example~\ref{E401}),
  the equality $\Arn^{\frq}(\fra_{\bullet})=\Arn^{\frq'}(\fra'_{\bullet})$ follows
  from Proposition~\ref{P302}.
  Proposition~\ref{P303} implies $A(v')=A(v)$.
  Since $v'(\frq')=v(\frq)$ and $v'(\fra'_\bullet)=v(\fra_\bullet)$,
  it is now clear that $v'$ computes $\Arn^{\frq'}(\fra'_\bullet)$
  if and only if $v$ computes $\Arn^{\frq}(\fra_\bullet)$.
\end{proof}

\subsection{Enlarging a graded sequence}
Fix a  graded sequence $\fra_{\bullet}$ of ideals,
a nonzero ideal $\frq$ on $X$ and a point $\xi\in X$.
For the proof of Theorem~\ref{main_thm2} it is useful
to enlarge $\frq$ and $\fra_\bullet$ so that they vanish only
at $\xi$. Given an integer $p\ge 1$, define
$\frc_\bullet$ by
\begin{equation}\label{e401}
  \frc_j=\sum_{i=0}^j\fra_i\cdot\frm^{p(j-i)}, j\ge 0
\end{equation}
where $\frm=\frm_\xi$ is the ideal defining $\overline{\xi}$.
Note that $\frc_1\supseteq\frm^p$.

\begin{proposition}\label{enlarge2}
  Assume $\Arn^\frq(\fra_\bullet)=\lambda^{-1}>0$ and 
  let $\xi$ be the generic point of an irreducible component 
 of the subscheme defined by $(\cJ(\fra_{\bullet}^{\lambda})\colon\frq)$.
 Define $\frc_\bullet$ using~\eqref{e401}.
  Then, for $p\gg0$, 
  $\lct^{\frq}(\frc_{\bullet})=\lct^\frq(\fra_\bullet)=\lambda$ 
  and if $v\in\Val^*_X$ computes $\Arn^{\frq}(\frc_{\bullet})$, 
  then $v$ also computes $\Arn^{\frq}(\fra_{\bullet})$. 
\end{proposition}

\begin{proposition}\label{enlarge3}
Suppose that $\Arn^{\frq}(\fra_{\bullet})=\lambda^{-1}>0$ and that $\frm^p\subseteq\fra_1$.
If $N\geq\lambda p$ and $\frr=\frq+\frm^N$, then $\Arn^{\frq}(\fra_{\bullet})=\Arn^{\frr}(\fra_{\bullet})$. Furthermore,
if $v\in\Val_X^*$, then $v$ computes $\Arn^{\frq}(\fra_{\bullet})$ if and only if
$v$ computes $\Arn^{\frr}(\fra_{\bullet})$.
\end{proposition}

\begin{proof}[Proof of Proposition~\ref{enlarge2}]
  In order to prove $\Arn^\frq(\frc_\bullet)=\Arn^\frq(\fra_\bullet)$ for $p\gg0$,
  let us first consider the special case when 
  $\frm\subseteq\sqrt{(\cJ(\fra_{\bullet}^{\lambda})\colon\frq)}$.
  Then there exists a positive integer $n$ such that 
  $\frm^n\cdot\frq\subseteq\cJ(\fra_{\bullet}^{\lambda})$. 
  Set $\lambda':=\lct^{\frm^n\frq}(\fra_\bullet)>\lambda$
  and pick $p>n/(\lambda'-\lambda)$. 
  Fix $0<\epsilon<1$ such that 
  $p>n/((1-\epsilon)\lambda'-\lambda)$. 

  Note that $v(\frc_\bullet)=\min\{v(\fra_\bullet),pv(\frm)\}$ for all $v\in\Val_X^*$.
  Thus 
  \begin{equation*}
    \Arn^\frq(\frc_\bullet)
    =\sup_{v\in\Val_X^*}\frac{\min\{v(\fra_\bullet),pv(\frm)\}}{A(v)+v(\frq)}
    \ge\sup_{v\in V_\epsilon}\frac{\min\{v(\fra_\bullet),pv(\frm)\}}{A(v)+v(\frq)},
  \end{equation*}
  where $V_\epsilon$ is the set of 
  $v\in\mathrm{Val}^*_X$ 
  for which $\frac{v(\fra_\bullet)}{A(v)+v(\frq)}\ge(1-\epsilon)/\lambda$.

  By the definition of $\lambda'$ we have
  \begin{equation*}
    \frac{n\cdot v(\frm)}{v(\fra_\bullet)}\ge\lambda'-\frac{A(v)+v(\frq)}{v(\fra_\bullet)}
  \end{equation*}
  for all $v\in\mathrm{Val}^*_X$. This implies
  \begin{multline*}
    \Arn^\frq(\frc_\bullet)
    \ge\sup_{v\in V_\epsilon}
    \frac{v(\fra_\bullet)}{A(v)+v(\frq)}
    \min\left\{1,\frac{p}{n}\left(
        \lambda'-\frac{A(v)+v(\frq)}{v(\fra_\bullet)}
      \right)\right\}\\
    \ge\sup_{v\in V_\epsilon}
    \frac{v(\fra_\bullet)}{A(v)+v(\frq)}
    \min\left\{1,\frac{p}{n}\left(
        \lambda'-\frac{\lambda}{1-\epsilon}
      \right)\right\}
    =\sup_{v\in V_\epsilon}
    \frac{v(\fra_\bullet)}{A(v)+v(\frq)}
    =\Arn^\frq(\fra_\bullet).
  \end{multline*}
  Therefore
  $\Arn^\frq(\frc_\bullet)\ge\Arn^\frq(\fra_\bullet)$, and 
  the reverse inequality is obvious.

  We now treat the general case.
  Consider the natural morphism $\phi\colon \Spec R=X'\to X$, 
  where $R=\widehat{\cO_{X,\xi}}$.
  Let $\xi'$ denote the closed point of $X'$, and 
  let $\frm'=\frm\cdot R$, 
  $\fra'_j=\fra_j\cdot R$, $\frq'=\frq\cdot R$, 
  and $\frc'_j=\frc_j\cdot R=\sum_{i=0}^j\fra'_i\cdot\frm'^{p(j-i)}$. 
  Note that $\frm'$ is the ideal defining the closed point of $X'$.
  It follows from Proposition~\ref{P302} that
  $\cJ(\fra_{\bullet}'^{\lambda})=\cJ(\fra_{\bullet}^{\lambda})\cdot R$. 
  By construction, 
  $\sqrt{(\cJ(\fra_{\bullet}'^{\lambda})\colon\frq')}=\sqrt{(\cJ(\fra_{\bullet}^{\lambda})\colon
    \frq)}\cdot R=\frm'$, so by the case already treated, we have
  $\lct^{\frq'}(\fra'_{\bullet})=\lct^{\frq'}(\frc'_{\bullet})$ for $p\gg0$.
  Therefore
  \begin{equation}\label{eq_cor_thm5_1}
    \lct^{\frq}(\fra_{\bullet})
    \leq\lct^{\frq}(\frc_{\bullet})
    \leq\lct^{\frq'}(\frc'_{\bullet})
    =\lct^{\frq'}(\fra'_{\bullet}),
  \end{equation}
  where the first inequality follows from the inclusions 
  $\fra_j\subseteq\frc_j$, and the second one from 
  Proposition~\ref{P302}.
  Since  $\lct^{\frq'}(\fra'_{\bullet})=\lct^{\frq}(\fra_{\bullet})$
  by the same proposition,
  it follows that all inequalities in~\eqref{eq_cor_thm5_1}
  are equalities.
  In particular, $\Arn^\frq(\frc_\bullet)=\Arn^\frq(\fra_\bullet)$.
    
  Suppose now that $v$ is a valuation that computes $\Arn^{\frq}(\frc_{\bullet})$. 
  Since $\fra_j\subseteq\frc_j$ for every $j$, we have $v(\frc_{\bullet})\leq
  v(\fra_{\bullet})$.
  Therefore
  \begin{equation}\label{eq_final}
  \Arn^{\frq}(\frc_{\bullet})=\frac{v(\frc_{\bullet})}{A(v)+v(\frq)}\leq\frac{v(\fra_{\bullet})}{A(v)+v(\frq)}
  \leq\Arn^{\frq}(\fra_{\bullet})=\Arn^{\frq}(\frc_{\bullet}).
  \end{equation}
  All the inequalities in (\ref{eq_final}) have to be equalities, hence $v$ also computes
  $\Arn^{\frq}(\fra_{\bullet})$.
  \end{proof}

\begin{proof}[Proof of Proposition~\ref{enlarge3}]
It follows from Proposition~\ref{prop2_3} that $\frq\not\subseteq \cJ(\fra_{\bullet}^{\lambda})$,
but $\frq\subseteq\cJ(\fra_{\bullet}^t)$ for every $t<\lambda$. In order to prove that
$\Arn^{\frq}(\fra_{\bullet})=\Arn^{\frr}(\fra_{\bullet})$, it is enough to show that under our assumptions, 
$\frm^N\subseteq\cJ(\fra_{\bullet}^{\lambda})$. This follows since 
\begin{equation*}\frm^N\subseteq\cJ(\frm^N)\subseteq\cJ(\frm^{\lambda p})\subseteq\cJ(\fra_1^{\lambda})\subseteq
\cJ(\fra_{\bullet}^{\lambda}).\end{equation*}

Suppose now that $v\in\Val_X^*$. Since 
\begin{equation*}\frac{v(\fra_{\bullet})}{A(v)+v(\frq)}\leq\frac{v(\fra_{\bullet})}{A(v)+v(\frr)},\end{equation*}
it follows that if $v$ computes $\Arn^{\frq}(\fra_{\bullet})$, then $v$ also computes
$\Arn^{\frr}(\fra_{\bullet})$. For the converse, it is enough to show that if $v$ computes
$\Arn^{\frr}(\fra_{\bullet})$, then $v(\frq)=v(\frr)$. Note that since $\frm^p\subseteq\fra_1$,
we have $v(\fra_{\bullet})\leq p\cdot v(\frm)$. Therefore
\begin{equation*}v(\frm)\geq \frac{v(\fra_{\bullet})}{p}=\frac{A(v)+v(\frr)}{\lambda p}>\frac{v(\frr)}{N},\end{equation*}
hence $v(\frm^N)>v(\frr)=\min\{v(\frq),v(\frm^N)\}$. This shows that $v(\frr)=v(\frq)$, and completes the proof of the proposition.
\end{proof}

\subsection{Proof of Theorem~\ref{main_thm1}}
Let $\frm=\frm_\xi$ be the ideal defining $\xi$. 
After applying Propositions~\ref{enlarge2} and~\ref{enlarge3} (and increasing $p$)
we may assume that $\frm^p\subseteq\fra_1$ 
and $\frm^p\subseteq\frq$ for some $p\ge1$.

Consider the canonical morphism 
$\phi\colon X'=\Spec\,R\to X$, where 
$R=\widehat{\cO_{X,\xi}}$. 
Since $X$ is regular, Cohen's structure theorem
yields an isomorphism 
$R\simeq k\llbracket x_1,\dots,x_d\rrbracket$ for a field $k$.
We put $\fra'_m=\fra_m\cdot R$, $\frq'=\frq\cdot R$, and $\frm'=\frm\cdot R$,
so $\frm'$ is the ideal defining the closed point $0$ of $X'$.  
By Proposition~\ref{P401} it suffices to find a valuation
$v'\in \Val_{X'}^*$ with center at $0$
that computes $\Arn^{\frq'}(\fra_{\bullet}')$.
Indeed, in this case the restriction $v$ of $v'$ to $X$ has center
$c_X(v)=\xi$
and computes $\Arn^{\frq}(\fra_{\bullet})$. 

Therefore we may assume 
$X=\Spec\,k\llbracket x_1,\dots,x_d\rrbracket$, 
and that $\fra_1$ and $\frq$ contain $\frm^p$ for some $p$, 
where $\frm$ is the ideal defining the closed point of $X$.
Fix $0<\epsilon<\Arn^{\frq}(\fra_{\bullet})$ and
suppose $v\in\Val^*_X$ is such that
$\frac{v(\fra_{\bullet})}{A(v)+v(\frq)}>\epsilon$. 
Since $\frm^p\subseteq \fra_1$, we have $v(\fra_{\bullet})\leq pv(\frm)$. 
In particular, $v$ has center at the closed point.
After rescaling $v$, we may assume $v(\frm)=1$,
so that $v(\fra_{\bullet})\leq p$, 
and therefore $A(v)\leq A(v)+v(\frq)\leq M$,
where $M=p/\epsilon$. 
We conclude that
$\Arn^{\frq}(\fra_{\bullet})=\sup_{v\in V_M}\frac{v(\fra_{\bullet})}{A(v)+v(\frq)}$,
where 
\begin{equation*}
  V_M=\{v\in\Val_X\mid v(\frm)=1, A(v)\leq M\}.
\end{equation*}
By Proposition~\ref{compact}, $V_M$ is compact.
Furthermore, by Proposition~\ref{lem_semicontinuity},
$A$ is lower semicontinuous 
and by Corollary~\ref{cor_prop2_arbitrary_val}
the functions $v\to v(\frq)$ and
$v\to v(\fra_{\bullet})$ are continuous on $V_M$
The function $v\to v(\fra_{\bullet})/(A(v)+v(\frq))$
is therefore upper semicontinuous on $V_M$, 
hence achieves its maximum at some $v\in V_M$.
This completes the proof of Theorem~\ref{main_thm1}.

\subsection{Proof of Theorem~\ref{main_thm2}}
Assume $\lambda:=\lct^{\frq}(\fra_{\bullet})<\infty$. 
The proof proceeds similarly to the proof of Theorem~\ref{main_thm1},
repeatedly using localization, completion, and field extensions.

We start by considering the weak versions of Conjectures~\ref{conj1}
and~\ref{conj2}.
Let $\xi$ be the generic point of an irreducible component 
of the subscheme defined by $(\cJ(\fra_{\bullet}^{\lambda})\colon\frq)$.
In view of Propositions~\ref{enlarge2} and~\ref{enlarge3}, we may assume 
that $\frm^p\subseteq\fra_1$ and $\frm^p\subseteq\frq$,
where $p\gg0$ and $\frm=\frm_\xi$ is the ideal defining $\xi$.

After invoking Proposition~\ref{P401} and Lemma~\ref{lemma_valuations}
we may replace $X$ by $\Spec\widehat{\cO_{X,\xi}}$. 
By Cohen's structure theorem, we may therefore assume
$X=\Spec k\llbracket x_1,\dots,x_d\rrbracket$ for a field $k$.
We still have 
that $\frm^p\subseteq\fra_1$ and $\frm^p\subseteq\frq$, 
where $\frm$ defines the closed point of $X$.
These inclusions allow us to apply Proposition~\ref{P401} 
and Lemma~\ref{lemma_valuations} ``in reverse'',
and assume $X=\AAA_k^n$ and $\xi=0$. 
Finally we can use Proposition~\ref{P402} and
Lemma~\ref{lemma_valuations2} with $K=\overline{k}$
to reduce to the case when $k$ is algebraically closed.
But then we are in the situation of Conjecture~\ref{conj2}.

Finally we consider the strong versions of Conjectures~\ref{conj1}
and~\ref{conj2}. Pick any $v\in\Val^*_X$ computing
$\lct^\frq(\fra_\bullet)$.
We must show that $v$ is quasi-monomial.
After replacing $X$ by a higher model and
using Lemma~\ref{compute_open_subset},
we may assume $\trdeg_X(v)=0$. The proof is now
almost identical to what we did for the weak version.
Let $\xi=c_X(v)$.
By Theorem~\ref{equivalent_description}, we may assume that
$\fra_m=\{f\mid v(f)\geq m\}$. In particular, there is $p\geq 1$ such that
$\frm^p\subseteq\fra_1$, where $\frm$ is the ideal defining $\xi$.
By applying Proposition~\ref{enlarge3}, we may also assume that $\frm^N
\subseteq\frq$ for some $N\geq 1$.
Two applications of Proposition~\ref{P401} 
and Lemma~\ref{lemma_valuations} reduce us to the case
when $X=\AAA_k^n$, $\trdeg(v)=0$ and $c_X(v)=\xi$,
where $\xi\in\AAA_k^n$ is a closed point. Invoking
Proposition~\ref{P402} and Lemma~\ref{lemma_valuations2} 
with $K=\overline{k}$ (note that $v$ extends to a valuation in
$\Val_{\AAA_{K}^n}$), we see that we may assume that
$k$ is algebraically closed, and then we are in
position to apply Conjecture~\ref{conj2}.
This completes the proof of Theorem~\ref{main_thm2}.

\section{The monomial case}\label{S101}
In this section we assume that $X=\AAA^n_k=\Spec(k[x_1,\dots,x_n])$
is the $n$-dimensional affine space over a 
field $k$ of characteristic zero, and 
$\fra_{\bullet}$ is a graded sequence 
of monomial ideals (that is, each $\fra_m$ is generated by monomials). 
In this case it is natural to focus on
\emph{monomial valuations}: these are the quasi-monomial valuations 
in $\QM(X,H)$, where $H=H_1+\dots+H_n$, with $H_i=V(x_i)$. 
Every such valuation $v$ is of the form $\val_\alpha$, where 
$\alpha=(\alpha_1,\dots,\alpha_n)\in\RR_{\geq 0}^n$ is given by  
$\alpha_i=v(x_i)$.
Note that the log discrepancy is then given by 
$A(\val_\alpha)=\langle e,\alpha\rangle$, where $e=(1,\dots,1)$, 
and where we put $\langle u,\alpha\rangle=\sum_{i=1}^nu_i\alpha_i$ 
whenever $u,\alpha\in\RR^n$. 

Denote by $r=r_{X,H}\colon\Val_X\to\QM(X,H)$ the retraction map.
Thus $\bar{v}:=r(v)$ is the monomial valuation for which 
$\bar{v}(x_i)=v(x_i)$ for all $i$. 
Thus $\bar{v}(\fra_\bullet)=v(\fra_\bullet)$ and
$\bar{v}(\frq)\le v(\frq)$ for any ideal $\frq$.
Moreover, by Lemma~\ref{lem_compatibility}, 
we have $A(\bar{v})\le A(v)$ with equality if and only if
$v=\bar{v}$ is monomial. This immediately implies that if
$v$ is not monomial, then 
\begin{equation*}
  \frac{v(\fra_\bullet)}{A(v)+v(\frq)}
  <\frac{\bar{v}(\fra_\bullet)}{A(\bar{v})+\bar{v}(\frq)}
  \le\Arn^\frq(\fra_\bullet);
\end{equation*}
hence $v$ does not compute $\Arn^\frq(\fra_\bullet)$.

On the other hand, consider the simplex 
$\Sigma=\{\alpha\in\RR_{\geq 0}^n\mid \langle e,\alpha\rangle=1\}$.
Then $A(\val_\alpha)=1$ for all $\alpha\in\Sigma$.
It is clear that $\alpha\to\val_\alpha(\frq)$ is
continuous on $\Sigma$ and by Lemma~\ref{cor_prop2_arbitrary_val}
the same is true for $\alpha\to\val_\alpha(\fra_\bullet)$.
Thus the 0-homogeneous function 
\begin{equation*}
  \alpha\to\frac{\val_\alpha(\fra_\bullet)}{A(\val_\alpha)+\val_\alpha(\frq)}
\end{equation*}
attains its supremum on $\Sigma$. We have proved the following
version of Conjecture~\ref{conj1}:
\begin{proposition}\label{P201}
  If $\fra_\bullet$ is a graded sequence of monomial ideals
  and $\frq$ is any ideal, then $\Arn^\frq(\fra_\bullet)$
  is computed by some monomial valuation.
  Furthermore, any valuation computing 
  $\Arn^\frq(\fra_\bullet)$ is monomial.
\end{proposition}

We now use this proposition to recover a 
formula by Howald~\cite{How} for the 
multiplier ideal $\cJ(\fra_\bullet^\lambda)$. 
First note that $\cJ(\fra_\bullet^\lambda)$ is 
a monomial ideal. To see this, let $f\in k[x_1,\dots,x_n]$
be any polynomial and let $\frq=\frq_f$ be the monomial
ideal generated by the monomials that appear in $f$ with nonzero coefficient.
It suffices to show that 
$\Arn^{(f)}(\fra_\bullet)=\Arn^\frq(\fra_\bullet)$. 
But this is clear by Proposition~\ref{P201} since
$v(f)=v(\frq)$ for any monomial valuation $v$.

To describe Howald's formula, we recall 
from~\cite{Mustata} (see also~\cite{Wolfe})
how to associate a convex region $P(\fra_\bullet)$ to $\fra_\bullet$.
For every $m\geq 1$, consider the Newton polyhedron
of $\fra_m$
\begin{equation*}
  P(\fra_m)=\text{convex hull of}\,\{u\in\ZZ_{\geq 0}^n\mid
  x^u\in\fra_m\}.
\end{equation*}
Our assumption that $\fra_m\neq (0)$ for some $m$ 
implies that some $P(\fra_m)$ is nonempty. 
The fact that $\fra_{\bullet}$ is a graded sequence of ideals gives
$P(\fra_{m})+P(\fra_{\ell})\subseteq P(\fra_{m+\ell})$
for all $m$ and $\ell$. In particular, we have 
$\frac{1}{m}P(\fra_m)\subseteq\frac{1}{pm}P(\fra_{mp})$. 
We put 
\begin{equation*}
  P(\fra_{\bullet}):=\overline{\bigcup_{m}\frac{1}{m}P(\fra_m)}.
\end{equation*}
This is a nonempty closed convex subset of $\RR_{\geq 0}^n$, with the property that 
\begin{equation}\label{eq_P}
  P(\fra_{\bullet})+\RR_{\geq 0}^n\subseteq P(\fra_{\bullet}).
\end{equation}
Indeed, each $P(\fra_m)$ satisfies the same property. 
\begin{remark}
  Given any nonempty closed convex subset $P\subseteq\RR_{\geq 0}^n$
  with the property~\eqref{eq_P} there exists a graded sequence 
  $\fra_\bullet$ of monomial ideals such that $P(\fra_\bullet)=P$.
  Indeed, we can take 
  $\fra_m=(x^u\mid u\in\ZZ_{\geq 0}^n\cap mP)$ for all $m\geq 1$.
  In general, the subset $P(\fra_\bullet)$ does not determine
  $\fra_\bullet$ uniquely. However, as the results below show, if 
  $P(\fra_\bullet)=P(\fra'_\bullet)$, then $\fra_\bullet$ 
  and $\fra'_\bullet$ should be regarded as equisingular.
\end{remark}
As an instance of basic convex analysis
we next show that the convex set $P=P(\fra_\bullet)$
determines, and is determined by, the concave function 
$w\to\val_w(\fra_\bullet)$ on $\RR_{\ge0}^n$.
\begin{lemma}\label{lemma_monomial1}
  If $\fra_{\bullet}$ is a sequence of monomial ideals on ${\mathbf A}_k^n$, then
  \begin{equation}\label{e202}
    \val_\alpha(\fra_{\bullet})=\inf\{\langle u,\alpha\rangle\mid u\in P(\fra_{\bullet})\}
    \quad\text{for $\alpha\in\RR_{\geq 0}^n$}.
  \end{equation}
  Conversely, we have 
  \begin{equation}\label{e203}
    P(\fra_\bullet)=\{u\in\RR_{\ge0}^n\mid \langle u,\alpha\rangle \ge \val_\alpha(\fra_\bullet)\ 
    \text{for all}\ \alpha\in\RR_{\ge0}^n\}.
  \end{equation}
\end{lemma}
\begin{proof}
  It is immediate from the definition that 
  $\val_\alpha(\fra_m)=\min\{\langle u,\alpha\rangle\mid u\in P(\fra_m)\}$. 
  It follows that
  \begin{equation*}
    \val_\alpha(\fra_{\bullet})=\inf_m\frac{\alpha(\fra_m)}{m}=\inf_m\inf_{u\in \frac{1}{m}P(\fra_m)}
    \langle u,\alpha\rangle=\inf_{u\in P(\fra_{\bullet})}\langle u,\alpha\rangle.
  \end{equation*}

The inclusion ``$\subseteq$" in (\ref{e203}) from from the description 
of $\val_\alpha(\fra_{\bullet})$.
On the other hand, if $u_0\not\in P(\fra_{\bullet})$, 
then we can find $v\in\RR^n$ and $b\in\RR$
such that
$\langle u,v\rangle\geq b$ for every $u\in P(\fra_{\bullet})$, while $\langle u_0,v\rangle<b$
(this is a general fact about closed convex subsets of $\RR^n$, see 
Theorem~4.5 in~\cite{Bro}). It follows from (\ref{eq_P}) that $v\in\RR_{\geq 0}^n$, hence 
$\langle u_0,v\rangle<\val_v(\fra_{\bullet})$. 
\end{proof}
We can now state and prove Howald's formula.
\begin{proposition}\label{multiplier_monomial_sequence}
  If $\fra_{\bullet}$ is a graded sequence 
  of monomial ideals, then
  \begin{equation}\label{e201}
    \cJ(\fra_{\bullet}^{\lambda})
    =(x^u\mid u+e\in \Int(\lambda P(\fra_{\bullet}))).
  \end{equation}
  Equivalently,  $\Arn^{(x^u)}(\fra_{\bullet})$ is equal to the unique 
  number $\alpha\ge0$ such that $\alpha(u+e)$ lies on the 
  boundary of $P=P(\fra_{\bullet})$. Moreover, a nontrivial monomial
  valuation $\val_\alpha$ computes $\Arn^{(x^u)}(\fra_\bullet)$
  if and only if $\alpha$ determines a supporting hyperplane of $P$
  at $\alpha(u+e)$,
  that is, $\langle\alpha(u+e),\alpha\rangle \le\langle u',\alpha\rangle$
  for all $u'\in P(\fra_\bullet)$.
\end{proposition}
If  $\fra_m=\fra^m$ for some
monomial ideal $\fra$, then $P(\fra_\bullet)=P(\fra)$,
$\cJ(\fra_\bullet^\lambda)=\cJ(\fra^\lambda)$
and~\eqref{e201} becomes Howald's original formula from~\cite{How}.
See also~\cite[Theorem~A]{Gue} for a similar result
in the context of toric plurisubharmonic functions
and~\cite{MZ} for an analytic approach to Howald's formula.
\begin{proof}
  By Proposition~\ref{P201}, $\Arn^{(x^u)}(\fra_{\bullet})$
  is the unique number $\alpha\ge0$ such that 
  $\val_\alpha(\fra_\bullet)\le\alpha\langle e+u,\alpha\rangle$
  for all $\alpha\in\RR_{\ge 0}^n$
  with equality for at least one $\alpha\ne0$.
  By~\eqref{e203}
  this means exactly that $\alpha(u+e)$ belongs
  to the boundary of $P(\fra_\bullet)$.
  Moreover, $\val_\alpha$ computes $\Arn^{(x^u)}(\fra_{\bullet})$
  if and only if $\val_\alpha(\fra_\bullet)=\alpha\langle e+u,\alpha\rangle$,
  and by~\eqref{e202} this means that $\alpha$ defines a
  supporting hyperplane of $P$ at $\alpha(u+e)$.
\end{proof} 

\begin{example}\label{rem_divisorial}
  If we put
  $P=\{(x,y)\in\RR_{\geq 0}^2\mid (x+1)y\geq 1\}$,
  we get a graded sequence of ideals $\fra_{\bullet}$ such that 
  $\Arn^{\cO_X}(\fra_{\bullet})=\frac{-1+\sqrt{5}}{2}$. 
  Furthermore, if $\alpha=\val_{(a,b)}$, then
  $\alpha(\fra_{\bullet})=2\sqrt{ab}-a$. We see that the nontrivial valuation $\alpha$
  computes $\Arn^{\cO_X}(\fra_{\bullet})$ if and only if $(a,b)=q(1-\alpha,1)$ for some
  $q\in\RR_{>0}$. In particular, this shows that $\Arn^{\cO_X}(\fra_{\bullet})$ is not computed
  by any divisorial monomial valuation.  
\end{example}

\section{The two-dimensional case}\label{S310}
Our goal in this section is to give a proof of the strong version of
Conjecture~\ref{conj2} in the two-dimensional case.
Let $k$ be an algebraically closed field of characteristic zero and
$X=\AAA_k^2=\Spec R$, where $R=k[x,y]$. We put $\frm=(x,y)$.
Consider a graded sequence $\fra_{\bullet}$ 
of $\frm$-primary ideals and a nonzero ideal $\frq$ on $X$.
Note that there exists $N\ge 1$ such that 
$\frm^{jN}\subseteq\fra_j$ for all $j$. We assume that $\Arn^{\frq}(\fra_{\bullet})>0$, and
we have to show that any valuation in $\Val_X$
with center at $0$ that computes $\Arn^{\frq}(\fra_{\bullet})$
must be quasi-monomial.

For $v\in\Val^*_X$ write 
\begin{equation*}
  \chi(v)=\frac{v(\fra_\bullet)}{A(v)+v(\frq)},
\end{equation*}
so that $\Arn^\frq(\fra_\bullet)$ is the supremum 
of $\chi$. As in the proof of Theorem~\ref{main_thm1},
it suffices to take the 
supremum over $v$ centered at the origin,
normalized by $v(\frm)=1$ and satisfying 
$A(v)\le M$ for some fixed $M<\infty$.
For such valuations, the Izumi-type estimate in~\eqref{e209} becomes
\begin{equation}\label{e104}
  \ord_0\le v\le A(v)\cdot \ord_0,
\end{equation}
on $R$, where $\ord_0$ is the divisorial valuation 
given by the order of vanishing at $0$.

Now assume $v_*\in\Val_X$ 
satisfies $v_*(\frm)=1$ and $A(v_*)\le M$
but that $v_*$ is not quasi-monomial.
We will show that $\chi(v_*)<\Arn^{\frq}(\fra_{\bullet})$.
The argument that follows is essentially equivalent to 
the one in~\cite{FJ3}, but it avoids appealing to the 
detailed structure of the valuative tree described in~\cite{FJ1}.
The key ingredient is a uniform control
on strict transforms of curves under birational
morphisms, see Lemma~\ref{L101}.

Note that $\trdeg(v_*)=0$ and 
$\ratrk(v_*)=1$,\footnote{Such a valuation is 
infinitely singular in the terminology of~\cite{FJ1}.}
or else $v_*$ would be an Abhyankar valuation,
hence quasi-monomial.
The idea is to find a suitably chosen increasing sequence 
of log-smooth pairs $(Y_n,D_n)$ above $\AAA^2$ 
such the corresponding retractions 
$v_n:=r_{Y_n,D_n}(v_*)$ increase to $v_*$.\footnote{This approach can be 
  used to classify valuations on surfaces and 
  recover the structure of the valuative tree
  as described in~\cite{FJ1}; see also~\cite{Spiv}.}
Furthermore, we will achieve $\chi(v_n)>\chi(v_{n+1})$
for $n\gg0$ and $\chi(v_n)\to\chi(v_*)$, which in particularly implies that
$\chi(v_*)<\Arn^{\frq}(\fra_{\bullet})$.

To start the procedure, let $\pi_0\colon Y_0\to\AAA^2$ be the blowup 
of $\AAA^2$ at the origin, with exceptional divisor $E_0$. 
Since $\trdeg v_*=0$, the center of $v_*$ on $Y_0$
is a closed point $p_0\in E_0$. 
\begin{lemma}\label{L2}
  There exist (algebraic) local coordinates $(z_0,w_0)$ at $p_0$
  on $Y_0$ such that $E_0=\{z_0=0\}$ and 
  $v_*(z_0)=1$, $v_*(w_0)=s_0/r_0$ for positive integers $r_0, s_0$ with
  $\gcd(r_0,s_0)=1$ and $r_0\ge 2$.
\end{lemma}
Here the key point is $r_0\ge 2$. The coordinate $w_0$ 
is not unique, but the numbers $r_0$ and $s_0$ are.
\begin{proof}
  Pick any coordinate $z_0\in\cO_{Y_0,p_0}$ such that 
  $E_0=\{z_0=0\}$. Then $v_*(z_0)=v_*(\frm)=1$.
  Note that $v_*(\cO_{Y_0,p_0}\smallsetminus\{0\})$ is a discrete 
  subsemigroup of $\RR_{\geq 0}$. Indeed, 
  if $v_*(f_1)<v_*(f_2)<\ldots\leq M$ is a bounded increasing sequence, 
  then we have a decreasing sequence of ideals
  $\{f\mid v_*(f)\geq v_*(f_i)\}$, all containing the zero-dimensional ideal
  $\{f\mid v_*(f)\geq M\}$.
  By the Izumi estimate~\eqref{e104} we have 
  $v_*(w)\le A_{Y_0}(v_*)\ord_{p_0}(w_0)$. Hence we can 
  pick $w_0\in\cO_{Y_0,p_0}$ such that
  $(z_0,w_0)$ form local coordinates at $p_0$ and such that 
  $v_*(w_0)$ is maximal.
  As $\ratrk v_*=1$, we have $v_*(w_0)\in\QQ$ and can write
  $v_*(w_0)=s_0/r_0$ for positive integers $r_0, s_0$ with $\gcd(r_0,s_0)=1$. We have to
  show that $r_0\ge2$.
  
  Suppose to the contrary that $r_0=1$. Since $v_*(z_0^{s_0})=v_*(w_0)$
  and $\trdeg(v_*)=0$, it follows that there is $\theta\in k^*$ such that
  $v_*(w_0+\theta z_0^{s_0})>v_*(w_0)$. Since $(z_0,w_0+\theta z_0^{s_0})$
  is a system of coordinates at $p_0$, this contradicts the maximality in the 
  choice of $v_*(w_0)$.
\end{proof}

With the notation in the lemma, let
$v_1$ be the monomial valuation in coordinates 
$(z_0,w_0)$ such that $v_1(z_0)=1$, $v_1(w_0)=s_0/r_0$.
Then $v_1$ is divisorial and $v_1(\frm)=1$.
Let $\rho_1\colon Y_1\to Y_0$ be a modification 
above $p_0$\footnote{By this, we mean that $\rho_1$ is proper, and an isomorphism over
  $Y_0\smallsetminus\{p_0\}$, with $Y_1$ regular.} such that the center of $v_1$
on $Y_1$ is an exceptional prime divisor  $E_1$.
We may and will assume that $\rho_1$ is a 
toroidal modification, in the sense
that the divisorial valuation $\ord_E$ associated to 
each exceptional prime divisor $E\subseteq Y_1$ is monomial in the 
coordinates $(z_0,w_0)$ at $p_0$.
(There is a minimal such $\rho_1$ which 
can be explicitly described by the continued fractions expansion
of $s_0/r_0$, but we don't need this information.)
The center of $v_*$ on $Y_1$ must be a 
\emph{free} point $p_1\in E_1$ (i.e.\ not belonging to
any other exceptional prime divisor)
or else $v_*$ would not take the correct value on $z_0$ or on $w_0$.
Moreover, if $D_1$ is the reduced exceptional divisor for
$\pi_0\circ\rho_1\colon Y_1\to\AAA^2$, then 
$v_1$ is equal to the retraction $r_{Y_1,D_1}(v_*)$.

Consider now $v_*$ as a valuation on $Y_1$ with center at $p_1$.
Up to a factor $r_0$, the
situation is then exactly the same 
as the one we had when considering $v_*$ at $(Y_0,p_0)$:
now $v_*(E_1)=r_0^{-1}$, whereas previously
$v_*(E_0)=1$.
We can find new coordinates $(z_1,w_1)$ at 
$p_1$ such that $E_1=\{z_1=0\}$ and $v_*(w_1)$ 
is maximal. 
The proof of Lemma~\ref{L2} gives
$v_*(w_1)=\frac{s_1}{r_0r_1}$ for positive integers $r_1, s_1$with $\gcd(r_1,s_1)=1$ 
and $r_1\ge2$. Let $v_2$ be the monomial valuation
in coordinates $(z_1,w_1)$ taking the same values as 
$v_*$ on these coordinates.
We can find a toroidal modification $\rho_2\colon Y_2\to Y_1$ 
above $p_1$ such that 
the center of $v_2$ (resp.\ $v_*$) on $Y_2$ is an exceptional prime divisor
$E_2$ (resp.\ a free point $p_2\in E_2$).

This procedure can be continued indefinitely, giving rise to
sequences $(v_j)_{j\ge1}$, $(E_j)_{j\ge 0}$, $(p_j)_{j\ge0}$,
$(z_j,w_j)_{j\ge0}$ and $(r_j,s_j)_{j\ge0}$.
We write $b_n=r_{n-1}r_{n-2}\dots r_0$. One can check that
$b_n=\ord_{E_n}(\frm)$. Since $r_j\ge2$ for all $j$, we 
have $b_n\ge 2^n$. By Corollary~\ref{cor_compatibility} 
we have $A(v_j)<A(v_*)$ for all $j$. 

We have the following estimate,
whose proof uses elementary intersection theory.
\begin{lemma}\label{L101}
  Let $\pi_0\colon Y_0\to\AAA^2$ be the blowup of the origin
  with exceptional divisor $E_0$, and consider a
  point $p_0\in E_0$.  
  Further, let $\rho\colon Y\to Y_0$ be a modification above $p_0$.
  Consider an exceptional prime divisor $E\subseteq Y$
  mapping to $p_0$ and a free point $p$ on $E$.
  Then, for any effective divisor $H\subseteq\AAA^2$ we have 
  \begin{equation}
    \ord_p(\widetilde{H}|_E)
    \le b^{-1}\cdot \ord_{p_0}(\widetilde{H}_0\vert_{E_0})
    \le b^{-1}\cdot \ord_0(H),
  \end{equation}
  where $\widetilde{H}_0$ and $\widetilde{H}$ are the strict transforms 
  of $H$ by $\pi_0$ and $\pi=\pi_0\circ\rho$,
  respectively, and where $b=\ord_E(\frm)$.
\end{lemma}

\noindent  We will apply Lemma~\ref{L101} to 
$\rho=\rho_n\circ\dots\circ\rho_1$.
We then have 
$b=b_n=\ord_{E_n}(\frm)\ge 2^n$,
so $\ord_p(\widetilde{H})\ll\ord_0(H)$ for $n\gg0$.
\begin{proof}
  The second inequality is clear since $E_0\simeq\PP^1$
  and the degree of $\widetilde{H}_0|_{E_0}$ equals $\ord_0(H)$.
  To prove the first inequality we write 
  $\rho^*E_0=bE+E'$,
  where $E'$ is a $\pi$-exceptional divisor whose support does not
  contain $p$. 
  It then follows that
  \begin{equation*}
    \ord_{p_0}(\widetilde{H}_0|_{E_0})
    =(\widetilde{H}_0\cdot E_0)_{p_0}
    =(\rho_*\widetilde{H}\cdot E_0)_{p_0}
    \ge(\widetilde{H}\cdot\rho^*E_0)_p
    =b\cdot (\widetilde{H}\cdot E)_p
    =b\cdot \ord_p(\widetilde{H}|_E).
  \end{equation*}
\end{proof}

\begin{lemma}\label{L5}
  The quasi-monomial valuations $v_n$ satisfy $v_n\leq v_{n+1}$ on 
  $\AAA^2$.
  Moreover, $v_n\to v_*$ and $\chi(v_n)\to\chi(v_*)$ as $n\to\infty$.
\end{lemma}
\begin{proof}
  It follows from Lemma~\ref{L201} and Corollary~\ref{C202}
  that $v_n\leq v_{n+1}\leq v_*$ on $\AAA^2$.
  We claim that 
  $v_n$ converges to $v_*$ as $n\to\infty$, that is,
  $v_n(f)\to v_*(f)$ for every $f\in R=k[x,y]$. 
  Now $v_*(f)>v_n(f)$ if and only if the strict transform $\widetilde{H}_n$ 
  of $H:=\{f=0\}$
  on $Y_n$ contains $p_n$, and the latter is equivalent to 
  $\ord_{p_n}(\widetilde{H}_n|_{E_n})\ge1$. 
  Thus Lemma~\ref{L101} implies
  that $v_n(f)=v_*(f)$ as soon as $2^n>\ord_0(f)$.

  Let us finally note that $\chi(v_n)\to\chi(v_*)$.
  Indeed, $v_n(\fra_\bullet)$ and $v_n(\frq)$
  increase to $v_*(\fra_\bullet)$ and $v_*(\frq)$,
  respectively, by Corollary~\ref{cor_prop2_arbitrary_val}. 
  Moreover, since $A$ is lower 
  semicontinuous we have $\liminf_n A(v_n)\ge A(v_*)$.
  But $A(v_*)\ge A(v_n)$, 
  so $\lim_{n\to\infty}A(v_n)=A(v_*)<\infty$.
  As $v^*(\frq)$ and $v^*(\fra_{\bullet})$ are finite, we conclude that
  $\lim_{n\to\infty}\chi(v_n)=\chi(v_*)$.
\end{proof}
\begin{lemma}\label{L6}
  We have $\chi(v_n)>\chi(v_{n+1})$ for $n\gg0$.
\end{lemma}

\noindent Together, Lemmas~\ref{L5} and~\ref{L6} show that
$\chi(v_*)<\chi(v_n)$ for $n$ large, and this completes
the proof of Conjecture~\ref{conj2} in dimension two.

\begin{proof}[Proof of Lemma~\ref{L6}]
  Pick $n_0$ such that $2^{n_0}>A(v_*)+v_*(\frq)$.
  In particular, $2^{n_0}>\ord_0(\frq)$.
  By Lemma~\ref{L101}, the 
  strict transform of $\frq$ on $Y_n$ 
  does not vanish at $p_n$ for $n\ge n_0$.

  Fix $n\ge n_0$ and 
  consider our local coordinates $(z_n,w_n)$
  at $p_n\in E_n\subseteq Y_n$.
  For $t>0$, let $v_{n,t}$ be the monomial valuation
  in $(z_n,w_n)$ with $v_{n,t}(z_n)=b_n^{-1}$ and
  $v_{n,t}(w_n)=b_n^{-1}t$. 
  Thus $v_{n,0}=v_n$ and 
  $v_{n,s_n/r_n}=v_{n+1}$.
  Note that $v_{n,t}(\frm)=1$ for all $t$.

  Let us study the function $t\to\chi(v_{n,t})$.
  First, $A(v_{n,t})=A(v_n)+b_n^{-1}t$.
  Second, $v_{n,t}(\frq)=v_n(\frq)$ for $n\ge n_0$.
  For an ideal $\fra\subseteq R$ with $V(\fra)\subseteq\{0\}$, the function,
  $t\to v_{n,t}(\fra)$ is concave 
  (and piecewise linear) for $t\ge0$.
  Let $\tilde{\fra}$ be the strict transform of $\fra$
  on $Y_n$. Then, for $0<t\ll 1$, we deduce
  using Lemma~\ref{L101}:
  \begin{equation*}
    v_{n,t}(\fra)
    =v_n(\fra)+b_n^{-1}t\cdot \ord_{p_n}(\tilde{\fra}|_{E_n})
    \le v_n(\fra)+2^{-n}b_n^{-1}t\cdot \ord_0(\fra).
  \end{equation*}
  By concavity, the same inequality holds for all $t>0$.
  Applying this with $\fra=\fra_m$, dividing by $m$, and then letting $m$ go to infinity,
  we obtain the inequality
  \begin{equation*}
    v_{n,t}(\fra_\bullet)
    \le v_n(\fra_\bullet)+2^{-n}b_n^{-1}t\cdot \ord_0(\fra_\bullet)
  \end{equation*}
  for all $t\ge0$. 
  Hence
  \begin{equation*}
    \chi(v_{n,t})\le\frac{v_n(\fra_\bullet)+2^{-n}b_n^{-1}t\cdot \ord_0(\fra_\bullet)}
    {A(v_n)+v_n(\frq)+b_n^{-1}t},
  \end{equation*}
  for all $t\ge0$, with equality for $t=0$.
  Here the right hand side is strictly decreasing in $t$ 
  if and only if
  \begin{equation}\label{e3}
    2^{-n}b_n^{-1}\ord_0(\fra_\bullet)\cdot(A(v_n)+v_n(\frq))
    <b_n^{-1}v_n(\fra_\bullet).
  \end{equation}
  Now $A(v_n)\leq A(v_*)<\infty$,
  $v_n(\frq)=v_*(\frq)<\infty$
  and $v_n(\fra_\bullet)\ge\ord_0(\fra_\bullet)$,
  so~\eqref{e3} holds for $n\ge n_0$
  by our choice of $n_0$. 
  Thus $\chi(v_{n+1})<\chi(v_n)$ for
  $n\ge n_0$, completing the proof.
\end{proof}

\appendix
\section{Multiplier ideals on regular excellent schemes over $\QQ$}
We explain how to deduce some basic results about multiplier ideals, the Restriction and  the Subadditivity Theorems, in our setting from the classical one. Recall that $X$ is a regular, connected, excellent scheme over $\QQ$. Our goal is to prove the following:

\begin{theorem}\label{restriction}
If $H$ is a regular closed subscheme of codimension one in $X$, then we have
$\cJ((\fra\cdot\cO_H)^{\lambda})\subseteq\cJ(\fra^{\lambda})\cdot\cO_H$ for every ideal
$\fra$ on $X$ and every $\lambda\in\RR_{\geq 0}$.
\end{theorem}

\begin{theorem}\label{subadditivity}
If $\fra$ and $\frb$ are ideals on $X$, and $\lambda,\mu$ are nonnegative real numbers, then
$\cJ(\fra^{\lambda}\frb^{\mu})\subseteq\cJ(\fra^{\lambda})\cdot \cJ(\frb^{\mu})$.
\end{theorem}

For the proofs, it will be convenient to consider the following reindexing of multiplier ideals.
If $t>0$, we put $\cJ(\fra^{t-}):=\cJ(\fra^{t-\epsilon})$ for $0<\epsilon\ll 1$. 
Of course, we have
$\cJ(\fra^t)=\cJ(\fra^{(t+\epsilon)-})$ for $0<\epsilon\ll 1$.
Similarly, if $\fra$ and $\frb$ are two ideals, and $s$, $t>0$, then we put
$\cJ(\fra^{s-}\frb^{t-}):=\cJ(\fra^{s-\epsilon}\frb^{t-\epsilon})$ for $0<\epsilon\ll 1$. 
Since $\cJ(\fra^s\frb^t)=\cJ(\fra^{(s+\epsilon)-}\frb^{(t+\epsilon)-})$ for $0<\epsilon\ll 1$, having
the statement
in Theorem~\ref{subadditivity}
for all $\lambda$, $\mu\geq 0$
is equivalent with having $\cJ(\fra^{\lambda-}\frb^{\mu-})\subseteq\cJ(\fra^{\lambda-})\cdot\cJ(\frb^{\mu-})$
for every $\lambda$, $\mu>0$. The same holds for Theorem~\ref{restriction}.

\begin{lemma}\label{lem_multiplier_ideals}
Suppose that $X=\Spec\,k\llbracket x_1,\dots,x_m\rrbracket$. 
If $\frm$ is the ideal defining the closed point in $X$, then
\begin{equation*}\cJ(\fra^{t-})=\bigcap_{N\geq 1}\cJ((\fra+\frm^N)^{t-})\end{equation*}
for every $t>0$.
\end{lemma}

\begin{proof}
  We may assume $\fra$ is nonzero: otherwise the assertion follows from
  $\bigcap_N\cJ(\frm^{N-})=\bigcap_N\frm^{N-m}=(0)$.
  Given $g\in\cO(X)$, we have $g\in\cJ(\fra^{t-})$ if and only if for every divisor $E$ over $X$
  \begin{equation}\label{eq_multiplier_ideals}
    \ord_E(g)+A(\ord_E)\geq t\cdot\ord_E(\fra).
  \end{equation}
  Furthermore, if this is not the case, then one can find a divisor 
  $E$ with center at the closed point
  such that~\eqref{eq_multiplier_ideals} fails (for this, one can
  argue as in the proof of~\cite[Lemma~2.6]{dFM}). 
  If $N>\ord_E(\fra)$, then $\ord_E(\fra)=\ord_E(\fra+\frm^N)$,
  and we see that $g\not\in\cJ((\fra+\frm^N)^{t-})$.
\end{proof}

\begin{remark}\label{rem_multiplier_ideals}
  Using the same proof, one sees that more generally, 
  if $\fra$ and $\frb$ are two ideals as in the lemma,
  and if $s, t>0$, then 
  \begin{equation*}\cJ(\fra^{s-}\frb^{t-})=\bigcap_{N\geq 1}\cJ((\fra+\frm^N)^{s-}(\frb+\frm^N)^{t-})\end{equation*}
  for every $s$, $t>0$.
\end{remark}

\begin{lemma}\label{complete_ring}
  Let $(R,\frm)$ be a complete local Noetherian ring, 
  and $(I_N)_{N\geq 1}$ and $(J_N)_{N\geq 1}$
  be sequences of ideals in $R$ with $I_{N+1}\subseteq I_N$
  and $J_{N+1}\subseteq J_N$ for all $N$.
  Write $I=\bigcap_{N\geq 1}I_N$ and $J=\bigcap_{N\geq 1}J_N$.
  \begin{enumerate}
  \item[(i)] We have 
    $IJ=\bigcap_{N\geq 1}I_NJ_N$.
  \item[(ii)] For every ideal $I'$ in $R$, we have $\bigcap_{N\geq 1}(I'+I_N)=I'+I$.
  \end{enumerate}
\end{lemma}
\begin{proof}
  Since $R/I$ is complete in the $\frm$-adic topology, 
  and the filtration given by $(I_N/I)_{N\geq 1}$
  is separated, it follows from a theorem of Chevalley 
  (see~\cite[Thm. 13, pp.270--271]{ZS})
  that given any $\ell$ there is $N$ such that 
  $I_N\subseteq I+\frm^{\ell}$. Similarly, we see
  that after possibly increasing $N$, 
  we may also assume that $J_N\subseteq J+\frm^{\ell}$.
  Therefore $I_NJ_N\subseteq IJ+\frm^{\ell}$, so
  \begin{equation*}
    \bigcap_{N\geq 1}I_NJ_N\subseteq \bigcap_{\ell\geq
      1}(IJ+\frm^{\ell})=IJ,
  \end{equation*}
  where the equality follows from Krull's Intersection Theorem. 
  As the other inclusion is trivial, this proves~(i).

  The argument for~(ii) is similar: we get from Chevalley's theorem that
  \begin{equation*}
    \bigcap_{N\geq 1}(I'+I_N)
    \subseteq \bigcap_{\ell\geq1}(I'+I+\frm^{\ell})=I'+I,
  \end{equation*}
  which completes the proof.
\end{proof}

\begin{proof}[Proof of Theorem~\ref{restriction}]
  If $X$ is a scheme of finite type over a field $k$, then the result is well-known 
  see~\cite[Section 9.5.A]{positivity}. 
  Note that since taking multiplier ideals commutes
  with passing to the algebraic closure
  (see Proposition~\ref{P302} and Example~\ref{E401}), 
  in this case one can assume that $k$ is algebraically closed.

  In the general case, it is enough to prove the two assertions 
  after replacing $X$ by $\Spec(\widehat{\cO_{X,\xi}})$, 
  where $\xi$ is any point of $X$. Indeed, this follows since taking
  multiplier ideals commutes with this operation by
  Proposition~\ref{P302} (recall that $\Spec(\widehat{\cO_{X,\xi}})\to X$ is regular since 
  $X$ is assumed to be excellent). 
  Therefore, by Cohen's structure theorem we may assume that 
  $X=\Spec\,k\llbracket x_1,\dots,x_m\rrbracket$, for some $m$, and that
  $H$ is defined by the ideal $(x_1)$.
  
  Note that the assertion in the theorem holds for every $\lambda$ if we replace
  $\fra$ by $\fra+\frm^N$, where $\frm$ is the ideal defining the closed point of $X$.
  Indeed, in this case there is an ideal $\fra_N$ on $\AAA_k^m$ such that
  $\fra_N\cdot\cO_X=\fra+\frm^N$. In this case, 
  we deduce the assertion on $X$ from the assertion on 
  $\AAA_k^m$, and the fact that taking multiplier ideals 
  commutes with completion at the origin. 
  
  As we have mentioned, this implies that
  \begin{equation*}
    \cJ(((\fra+\frm^N)\cdot\cO_H)^{\lambda-})
    \subseteq\cJ((\fra+\frm^N)^{\lambda-})\cdot\cO_H
  \end{equation*}
  for all $\lambda>0$.
  Intersecting over $N\geq 1$, and using Lemma~\ref{lem_multiplier_ideals} and
  Lemma~\ref{complete_ring}~(ii) we get
  \begin{equation*}
    \cJ((\fra\cdot\cO_H)^{\lambda-})\subseteq\cJ(\fra^{\lambda-})\cdot\cO_H
  \end{equation*}
  for every $\lambda>0$. As we have seen, this gives the assertion in the theorem.
\end{proof}
\begin{proof}[Proof of Theorem~\ref{subadditivity}]
Again, the result is known when $X$ is of finite type over a field (see~\cite[Section 9.5.B]{positivity}).
Arguing as in the proof of Theorem~\ref{restriction}, we see that we may assume
$X=\Spec\,k\llbracket x_1,\dots,x_m\rrbracket$, and that we have
\begin{equation*}\cJ((\fra+\frm^N)^{\lambda-}(\frb+\frm^N)^{\mu-})\subseteq 
\cJ((\fra+\frm^N)^{\lambda-})\cdot\cJ((\frb+\frm^N)^{\mu-})\end{equation*}
for all $\lambda$, $\mu>0$.
Taking the intersection over $N\geq 1$ and using Lemma~\ref{lem_multiplier_ideals}
(see also Remark~\ref{rem_multiplier_ideals}) and Lemma~\ref{complete_ring}~(i),
we deduce
\begin{equation*}\cJ(\fra^{\lambda-}\frb^{\mu-})\subseteq\cJ(\fra^{\lambda-})\cdot\cJ(\frb^{\mu-})\end{equation*}
for all $\lambda$, $\mu>0$. As we have seen, this implies the assertion in the theorem.
\end{proof}

\providecommand{\bysame}{\leavevmode \hbox \o3em
{\hrulefill}\thinspace}

\end{document}